\documentclass[a4paper,10pt]{amsart}

\usepackage{algorithm}
\usepackage{algorithmic}
\usepackage[foot]{amsaddr}
\usepackage{amsfonts,amssymb,amsthm,amsrefs} 
\usepackage{enumitem}
\usepackage[left=2.5cm, right=2.5cm]{geometry}
\usepackage{graphicx}
\usepackage{mathtools}
\usepackage[dvipsnames]{xcolor}
\usepackage{microtype} 
\usepackage{subcaption}

\usepackage[T1]{fontenc}
\usepackage[libertinus]{newtx} 
\usepackage{PTSans} 

\newtheorem{theorem}{Theorem}

\newtheorem{remark}{Remark}
\newtheorem{lemma}{Lemma}
\newtheorem{definition}{Definition}

\usepackage[colorlinks=true, citecolor=violet, urlcolor=blue, linkcolor=teal]{hyperref}

\DeclareMathOperator{\diag}{\mathrm{diag}\,}
\DeclareMathOperator{\proj}{\mathrm{Proj}\,}
\newcommand{\abs}[1]{|#1|\left.\vphantom{T_{j_0}^0}\!\!\right.}
\newcommand{\norm}[2]{\|#1\|\left.\vphantom{T_{j_0}^0}\!\!\right._{#2}}
\newcommand{\braket}[2]{\left\langle \begin{gathered} #1 \end{gathered} \,\middle|\, \begin{gathered} #2 \end{gathered} \right\rangle}
\newcommand{\conj}[1]{\overline{#1}}

\newcommand{\ket}[1]{\left\lvert #1 \right\rangle}
\newcommand{\intinterval}[2]{\left[#1:#2\right]} 
\newcommand{\ii}{\mathrm{i}} 

\renewcommand{\Re}{\mathrm{Re}}
\renewcommand{\Im}{\mathrm{Im}}

\newlength{\widthfigure} 
\newcommand{\folderImages}{.} 
\newcommand{\currentImgsLargeDt}{} 
\newcommand{\currentImgsSmallDt}{} 

\title[Numerical Bogoliubov approximation of bosonic many-body quantum dynamics]{Numerical Bogoliubov approximation of\\ bosonic many-body quantum dynamics}

\author{Yoann Le Hénaff}
\email{yoann.le-henaff@na.uni-tuebingen.de}
\author{Christian Lubich}
\author{Peter Pickl}
\address{Mathematisches Institut, University of Tübingen}

\date{}

\begin{document}

\begin{abstract}
    We propose an algorithm for the numerical solution of the time-dependent Schrödinger equation for \( N \gg 1 \) bosons interacting via a two-body interaction potential in the mean-field scaling regime. 
    Unlike existing numerical algorithms for quantum dynamics, the proposed algorithm has a computational complexity that is independent of the number \( N \) of particles.
    It is based on an ansatz for the wave function inspired by the Bogoliubov theory, which is an asymptotic theory for large \( N \). 
    
    In a first part, we follow Bogoliubov theory in evolving the condensate function by the Hartree equation, but we 
    use the Dirac--Frenkel time-dependent variational principle to derive nonlinear equations of motion for the excitation functions in the Bogoliubov ansatz. By carefully exploiting the bosonic symmetry, the so obtained approximate wave functions are shown to have derivatives bounded independently of $N$ apart from a time-dependent global phase factor which rotates with a frequency proportional to \( N \).
    
    In the second part, we numerically approximate the wave function (modulo the phase) by a composite algorithm, which employs an appropriate splitting of the equations of motion and uses a dynamical low-rank approximation of the fundamental two-particle excitation function by linear combinations of products of single-particle functions. The low-rank excitation function is evolved in time by a basis update and Galerkin (BUG) integrator that is robust to small singular values of the coefficient matrix and allows for rank adaptivity. Numerical results are provided to assess the performance of the proposed algorithm.
\end{abstract}

\maketitle

\section{Introduction}

This paper is concerned with the approximate numerical solution of the time-dependent Schrödinger equation for a system of \( N \gg 1 \) bosonic particles: 
\begin{equation}
    \label{eqn: general schrodinger equation}
    \ii \partial_t \psi(t, x) = H\psi(t, x), \quad \psi(0, \cdot) = \psi_0,
\end{equation}
where \( \psi(t, \cdot) \in L^2(\Omega^N) \) for all \( t \geq 0 \) and \( \psi_0 \in L^2(\Omega^N) \).
The domain \( \Omega = \mathbb{R}^d \), \( d\geq 1 \), is the single-particle space.
The bosonic framework requires that for any time \( t\geq 0 \), the wave function \( \psi(t, x_1, \dots, x_N) \) is invariant under permutation of the space variables \( x_1, \dots, x_N \), i.e.
\begin{equation*}
    \psi(t, x_1, \dots, x_N) = \psi(t, x_{\sigma(1)}, \dots, x_{\sigma(N)})\qquad \forall \sigma \in S_N,
\end{equation*}
where \( S_N \) denotes the set of all permutations of \( \{1, \dots, N\} \).

The Hamiltonian \( H \) is given by
\begin{equation}
    \label{eqn: full schrodinger hamiltonian}
    H = \sum_{j=1}^N -\Delta_{x_j} + \frac{1}{N-1} \sum_{1\leq j < k \leq N} v(x_j - x_k),
\end{equation}
where \( -\Delta_{x_j} \) denotes the Laplacian operator taken with respect to variable \( x_j \) and \( v: \Omega\to \mathbb{R} \) is a real-valued interaction potential assumed to be even and bounded.
The factor \( (N-1)^{-1} \) in the potential operator characterizes the scaling we consider, called the {\it mean-field regime}.
We thus have weak, long-range interactions.
With minimal changes to the current paper, it would also be possible to include an external trapping potential \( U_{\text{trap}}: \Omega \to \mathbb{R} \).
For the sake of presentation, we consider \( U_{\text{trap}} = 0 \).

It is a well-known fact that the exact solution of \eqref{eqn: general schrodinger equation} cannot be obtained in general, and that various approximations are needed to understand the behaviour of \( \psi \).
One situation where approximations to many-body bosonic systems are needed is in studying the phenomenon of superfluidity; see e.g. \cites{bogolubovTheorySuperfluidity1947,leggettSuperfluidity1999,tilleySuperfluiditySuperconductivity2019}.

\subsection{Existing numerical schemes}
\label{subsection: existing schemes}

A widely used algorithm originally developed for the numerical simulation of quantum molecular dynamics of small molecules is the \emph{Multi-Configurational Time-Dependent Hartree} (MCTDH) method.
It has been introduced in \cite{meyerMulticonfigurationalTimedependentHartree1990} and has been further developed over the years;
see \cites{beckMulticonfigurationTimedependentHartree2000,meyerIntroductionMCTDH2018,lubichQuantumClassicalMolecular2008}.
The MCTDH ansatz is very general and can be applied to various systems of different nature.
When the MCTDH method is directly applied to an 
\( N \)-particle bosonic system,  the wave function is approximated as
\begin{equation*}
    u(t,x) = \sum_{j_1=1}^{r} \dots \sum_{j_N=1}^{r} C_{j_1, \dots, j_N}(t) \phi_{j_1}(t,x_1) \cdot \ldots \cdot\phi_{j_N}(t,x_N).
\end{equation*}
The value \( r \) is the chosen rank, and \( N \) is the total number of particles.
The symmetric core tensor \( C \in \mathbb{C}^{r^N}\), which is invariant under permutations of indices, and the evolving orbitals \( \phi_j(\cdot,t) \) \( (j=1,\dots,r) \) are evolved in time based on the Dirac--Frenkel time-dependent variational principle.
However, the computational cost is exponential with respect to the number of particles, growing as $O(r^N)$, and so it is not applicable to systems where \( N \) is large.

Regarding the simulation of a larger number of bosons, there has been numerical work \cites{alonMulticonfigurationalTimedependentHartree2008,lodeNumericallyExactQuantum2012,lodeMulticonfigurationalTimedependentHartree2016} in that direction in the framework of second quantization.
There an MCTDH-like method, called MCTDHB, is proposed and used to study symmetric configurations of \( r \) time-dependent orbitals \( \varphi_j(\cdot,t) \) for $j=1,\dots,r$: the ansatz is
\begin{equation}
    \label{eqn: mctdhb ansatz}
    u(t) = \sum_{\vec{n}} C_{\vec{n}}(t) \ket{\vec{n}; t},
\end{equation}
where \( \vec{n} = (n_1, \dots, n_r) \) and 
\begin{equation}
    \ket{\vec{n};t} = \frac{1}{\sqrt{n_1! \dots n_r!}}\, a_1^*(t) ^{n_1} \dots  a_r^*(t)^{n_r} \ket{0},
\end{equation}
where \( \ket{0} \) is the vacuum state.
The operator \( a_j^*(t) \) is the creation operator associated to the orbital \( \varphi_j(t) \).
In other words, being in the state \( \ket{(n_1, \dots, n_r); t} \) means that there are \( n_1 \) bosons in the orbital \( \varphi_1(t) \), \( n_2 \) bosons in the orbital \( \varphi_2(t) \), \dots, and the wave function is symmetrized.
The summation in \eqref{eqn: mctdhb ansatz} is over all \( \vec{n} \) such that \( n_1 + n_2 + \dots + n_r = N \), in order to ensure that the total number of bosons is preserved.
The second quantization formalism helps in reducing the computational cost, but it still depends on \( N \) polynomially of degree $r-1$, making it not suitable for many-body systems with large \( N \).

\medskip

It is thus of considerable interest, both in physics and numerical analysis, to develop a computational scheme that is able to deal with a large number \( N \) of particles.

\subsection{Elements of Bogoliubov theory and approximation ansatz}
\label{section: elements of bogoliubov theory}

The study of the phenomenon of superfluidity gave rise to the Bogoliubov theory via the seminal paper \cite{bogolubovTheorySuperfluidity1947}.
Bogoliubov's ideas have then been expanded, see for example the recent works \cites{bossmannHigherOrderCorrections2020,bossmannBogoliubovDynamics2021,lewinFluctuationsHartreeStates2015,lewinBogoliubovSpectrumInteracting2015} and references therein, which describe the dynamics of quantum fluctuations around a Bose--Einstein condensate.
The aim of this subsection is not to give full details of the Bogoliubov theory, but simply to present some known results and to deduce an approximation ansatz for our numerical scheme.

The main result we shall use from the Bogoliubov theory is about the approximation to the wave function \( \psi \);   
see e.g. \cites{rodnianskiQuantumFluctuationsRate2009,grillakisSecondOrderCorrectionsMean2010,grillakisSecondorderCorrectionsMean2011,spohnKineticEquationsHamiltonian1980,lewinFluctuationsHartreeStates2015,napiorkowskiDynamicsInteractingBosons2023} and references therein: in many different variants, it is shown that as \( N\to\infty \), 
\begin{equation}
    \label{eqn: intro - bogoliubov approximation}
    \psi(t) \approx \sum_{j=0}^N \varphi(t)^{\otimes N-j} \otimes_s \widetilde{\chi}_j(t),
\end{equation}
where \( \widetilde{\chi}_j \in L^2_s(\Omega^j) \), and the \emph{condensate function} \( \varphi \) is the solution to the Hartree equation:
\begin{equation}
    \label{eqn: hartree equation}
    \ii \partial_{t} \varphi(t, x) = \left( -\Delta_x + v \ast \abs{\varphi(t, \cdot)}^2(x) - \mu_{\varphi(t)}\right) \varphi(t, x), \quad \varphi(0, \cdot) = \varphi_0, \quad \text{on } \Omega,
\end{equation}
with
\begin{equation}
    \label{eqn: mu in hartree equation}
    \mu_{\varphi(t)} := \frac{1}{2} \int_{\Omega} (v\ast \abs{\varphi(t, \cdot)}^2)(x) \abs{\varphi(t,x)}^2 dx.
\end{equation}
The symbol \( \otimes_s \) in \eqref{eqn: intro - bogoliubov approximation} denotes a symmetrization of the usual tensor product, \( L^2_s \) is the space of symmetric \( L^2 \) functions, and the notation \( f\ast g \) denotes the usual convolution \( (f\ast g)(x) = \int_{\Omega} f(x-y)g(y)dy \), for functions \( f,g \) such that the integral makes sense.
Precise definitions are given later in Section \ref{section: notations}.
The approximation \eqref{eqn: intro - bogoliubov approximation} can be understood as correction terms around the Bose--Einstein condensate \( \varphi \otimes \dots \otimes \varphi \).

\begin{remark}
    We follow \cites{bossmannBogoliubovDynamics2021,mitrouskasBogoliubovCorrectionsTrace2019,lewinFluctuationsHartreeStates2015,bossmannHigherOrderCorrections2020,namBogoliubovCorrectionMeanfield2017} for the normalization of the Hartree equation \eqref{eqn: hartree equation} using the coefficient \( \mu_{\varphi(t)} \) given in \eqref{eqn: mu in hartree equation}.
    The convenience of this coefficient will be clarified later in Section \ref{section: equations of motion - study of RHS}.
\end{remark}

The functions \( \widetilde{\chi}_j \) have been studied in detail -- for example in \cites{lewinBogoliubovSpectrumInteracting2015,namBogoliubovCorrectionMeanfield2017,namNoteValidityBogoliubov2017,bossmannBogoliubovDynamics2021,bossmannHigherOrderCorrections2020,schleinBogoliubovTheoryManybody2021,namBogoliubovExcitationSpectrum2023} -- but we shall only use a few essential results about them.
First, the marginals of the functions \( \widetilde{\chi}_j \) are orthogonal to the state \( \varphi \):
\begin{equation}
    \label{eqn: bogoliubov orthogonality relation}
    \int_{\Omega} \conj{\varphi(t, y)} \widetilde{\chi}_j(t, y, \cdot) \, dy \equiv 0,\ \forall\ j \geq 1,\ \forall t\geq 0.
\end{equation}
Second, it is also known that \( \norm{\widetilde{\chi}_{2j}}{2} \to 0 \) as \( j\to \infty \), and that \( \norm{\widetilde{\chi}_{2j+1}}{2} \approx 0 \) for \( j\geq 0 \) (with an equality as \( N\to \infty \)).
Third, the functions \( \widetilde{\chi}_{2j} \) for \( j > 1 \) are actually symmetric tensorizations of the function \( \widetilde{\chi}_2 \):
\begin{equation}
    \label{eqn: Bogoliubov tensorization}
    \widetilde{\chi}_{2j} = \sum_{\tau\in P_{2j}} \widetilde{\chi}_2^{\otimes j}(x_\tau),
\end{equation}
where \( P_{2j} = \left\{ \sigma \in S_{2j} : \sigma(2k-1) \leq \min(\sigma(2k), \sigma(2k+1)) \text{ for all } 1 \leq k \leq j \right\} \), with \( S_{2j} \) the symmetric group on the set \( \{1, \dots, 2j\} \).
It will be more convenient to work with functions of unit norm, hence we introduce the functions \( \chi_{2j} \) defined as follows:
\begin{equation}
    \label{eqn: Bogoliubov tensorization normalized}
    \chi_{2j} := e^{-\ii \theta_{2j}} \frac{\widetilde{\chi}_{2j}}{\sqrt{\braket{\widetilde{\chi}_{2j}}{\widetilde{\chi}_{2j}}}},\quad j\ge 1.
\end{equation}
The angle \( \theta_{2j} \in \mathbb{R} \) is determined according to the condition \( \braket{\chi_{2j}}{\dot{\chi}_{2j}} = 0 \).
The functions \( \chi_{2j} \) are called \( 2j \)-\emph{particle excitations} to the condensate function \( \varphi \).
We then have the following ansatz for our numerical method, which works in first quantization:
\begin{equation}
    \label{eqn: final ansatz u}
    \psi(t,x) \approx u(t, x) = \sum_{j=0}^n c_{2j}(t) \left( \varphi^{\otimes N-2j} \otimes_s \chi_{2j} \right)(t, x),
\end{equation}
where \( n \ll N/2 \), and where \( c_{2j} \in \mathbb{C} \), \( j=0, \dots, n \).
It can be assumed that \( \chi_0 = 1 \) for all times.
From now on, we may omit in the notation the time dependence of all functions, and unless explicitly stated otherwise, all functions will be time-dependent.

To the authors' knowledge, no numerical method has taken advantage of the ansatz given by the Bogoliubov theory in the time-dependent setting.
The only numerical works we are aware of that use some elements of the Bogoliubov theory are \cites{lewinNumericalPerspectiveHartreeFockBogoliubov2014,linNumericalSolutionLarge2021}, which are interested in finding ground states of the Hartree-Fock-Bogoliubov model.
We aim at developing a numerical method suited for the time-dependent bosonic Schrödinger equation.

\begin{remark}
    \label{remark: scope and applicability of this work}
    We do not claim that for small or moderate values of \( N \), the method introduced in this work can be as accurate as other methods, e.g. MCTDHB.
    Its interest lies in the large \( N \) limit, where other methods fail in practice.
    The method introduced in this work has to be understood as a first-order approximation, where the zero-th order term is the Hartree product state \( \varphi^{\otimes N} \): both methods can be applied in practice to any number of bosons, and they assume a non-physical (almost) complete Bose--Einstein condensation.
    Our method outperforms the Hartree method because it can additionally capture pair-interactions.
\end{remark}

\subsection{Overview \& main ideas}
\label{section: overview}

The goal of this paper is to numerically approximate the solution \( \psi \) to the Schrödinger equation \eqref{eqn: general schrodinger equation} when \( N\gg 1 \), using the ansatz \eqref{eqn: final ansatz u} for the approximation $u$.
This naturally raises two questions: \emph{How to evolve \( u \) in continuous time?} and then \emph{How to update \( u \) efficiently over a discrete time step?}
Detailed answers to both questions will be given in dedicated sections, but here we give a short and gentle summary of the main ideas that need to be realized in sometimes lengthy and less gentle calculations.

\subsubsection{How to evolve \texorpdfstring{\( u \)}{u} in continuous time?}
Our starting point is the approximation ansatz \eqref{eqn: final ansatz u}, viz.,
\begin{equation*}
    u(t, x) = \sum_{k=0}^n c_{2k}(t) \left( \varphi^{\otimes N-2k} \otimes_s \chi_{2k} \right)(t, x).
\end{equation*}
The function \( u \) is uniquely determined by the coefficients \( c_{2k} \), by the condensate function \( \varphi \) which is the unique solution of the Hartree equation \eqref{eqn: hartree equation}, and by the excitation functions \( \chi_{2k} \), which for $k\geq 2$ are determined from $\chi_2$ via \eqref{eqn: Bogoliubov tensorization}-\eqref{eqn: Bogoliubov tensorization normalized}.
We call \( \mathbf{P} := (c_0, \dots, c_{2n}, \chi_2)  \) the  {\it parameters} of the approximation \( u \) and consider $u=u(\mathbf{P})$. Focusing on the parameters,
we reduce the problem to dealing with the scalar coefficients $c_{2k}$ and the two-particle excitation function $\chi_2$ instead of the full \( N \)-particle wave function \( u \).

The case \( n>1 \) is more technical than the case \( n=1 \), but both cases rely on the same essential ingredients. 
For this reason, we only present \( n=1 \) here in this ``gentle'' overview.
We apply the Dirac--Frenkel time-dependent variational principle (see e.g. \cites{kramerGeometryTimedependentVariational1980,lubichQuantumClassicalMolecular2008}) to obtain differential equations for the parameters \( \mathbf{P} \) such that \( u \) approximately solves the Schrödinger equation \eqref{eqn: general schrodinger equation}.
Since the condensate function \( \varphi \) is determined by the Hartree equation \eqref{eqn: hartree equation}, the Dirac--Frenkel time-dependent variational principle is here applied on a time-dependent manifold that depends on \( \varphi(t) \). In the particular case where $\varphi$ is a stationary point of the Hartree equation, the approximation manifold is time-independent, which ensures energy conservation in the variational approximation.
We will obtain the {\it equations of motion} of the parameters in Theorem~\ref{theorem: equations of motion}:
\begin{equation}
    \label{eom-1}
    \begin{aligned}
        \dot{c}_{2k}(t) &= \braket{(\varphi^{\otimes N-2k} \otimes_s \chi_{2k})(t, x_1, \dots, x_N)}{-\ii Hu(t, x_1, \dots, x_N)}_{x_1, \dots, x_N} \\
        &\qquad - c_{2k} (N-2k) \braket{\varphi(t, x_1)}{\dot{\varphi}(t, x_1)}_{x_1},\quad k = 0, 1, \\
        \dot{\chi}_2(t, x_1, x_2) &= \frac{1}{c_2(t)} \sqrt{\binom{N}{2}} \ 
        \Pi_{\varphi(t)} \!
        \braket{\varphi^{\otimes N-2}(t, x_3, \dots, x_N)}{-\ii Hu(t, x_1, \dots, x_N)}_{x_3, \dots, x_N} \\
            &\qquad - \varphi(t, x_1) \braket{\dot{\varphi}(t, x_3)}{\chi_2(t, x_3, x_2)}_{x_3} - \braket{\dot{\varphi}(t, x_3)}{\chi_2(t, x_3, x_1)}_{x_3} \varphi(t, x_2),\quad x_1, x_2 \in \Omega.
    \end{aligned}
\end{equation}
In this work, the dot notation \( \dot{f} \) will be used to denote the differentiation with respect to time of a function \( f \).
In brakets of the form \( \braket{\dots}{\dots}_x \), the subscript \( {}_x \)  means that the braket is computed along the variable \( x \), and that the result is a function of any other variable.
The term ``braket'' can be replaced throughout this work by ``\( L^2(\Omega) \) inner product, anti-linear with respect to its first argument and linear with respect to its second argument''.
As an example, 
\begin{equation*}
    \braket{\dot{\varphi}(t, x_3)}{\chi_2(t, x_3, x_1)}_{x_3} = \int_{\Omega} \conj{\dot{\varphi}(t, x_3)} \chi_2(t, x_3, x_1) dx_3
\end{equation*}
is a function of the variables \( t \) and \( x_1 \).
As one can see, the notation is cumbersome and we will introduce lighter notation in Section \ref{section: notations}.
The projection $\Pi_{\varphi(t)}$, which appears in the equation for \( \dot{\chi}_2 \), originates from gauge conditions imposed in the Dirac--Frenkel variational principle and is not explicitly written out here.

At first sight, the right-hand sides of the equations of motion seem to be of size \( \mathcal{O}(N) \).
However, we will show the following results (Theorems \ref{theorem: setting n>1 no tensorization - chi2dot uniformly bounded wrt N} and \ref{theorem: unitary evolution coeffs}):
\begin{itemize}
    \item The equations of motion for the coefficients \( c_0, c_2 \) can be written in the form
        \begin{equation*}
            \begin{pmatrix}
                \dot{c}_0\\ \dot{c}_2
            \end{pmatrix}(t)
            = -\ii \check{\mathbf{H}}(t)
            \begin{pmatrix}
                c_0\\ c_2
            \end{pmatrix}(t),
        \end{equation*}
        where \( \check{\mathbf{H}}(t) \) is a Hermitian matrix with slowly varying components that are uniformly bounded with respect to \( N \) for all times.
        This shows that \( c(t) \) has a smooth, unitary evolution, and in particular that the norm of \( u \) is preserved in time since \( \norm{u}{2}^2 = \abs{c_0}^2 + \abs{c_2}^2 \).
    
    \item The right-hand side in the equation of motion for \( \chi_2 \) is uniformly bounded with respect to \( N \).
        This comes from the orthogonal projection \( \Pi_\varphi\), which makes \( \Pi_\varphi\! \braket{\varphi^{\otimes N-2}(t, x_3, \dots, x_N)}{-\ii Hu(t, x_1, \dots, x_N)}_{x_3, \dots, x_N} \) of size \( \mathcal{O}(1/N) \), as will be shown by a combinatorial argument.
\end{itemize}
Ultimately, at the end of Section \ref{section: equations of motion}, we will have obtained the time-evolution of the parameters \( \mathbf{P} \) of \( u \) via the Dirac--Frenkel time-dependent variational principle for arbitrary $n\ge 1$ in the ansatz \eqref{eqn: final ansatz u}.
The parameters \( \mathbf{P} \)  satisfy equations of motion that generalize those given for $n=1$ in \eqref{eom-1}, and a careful study of the right-hand sides of these equations shows that the parameters are slowly varying uniformly with respect to~\( N \).

\subsubsection{How to update \texorpdfstring{\( u \)}{u} efficiently over a discrete time step?}
For actual computations, we need to devise an efficient discretization of the equations of motion. This will be done with the principal aim that \textit{the computational effort is uniformly bounded in the number $N$ of particles.}.

Some hurdles to this aim and further challenging aspects come to mind immediately:
\begin{enumerate}[label=(\alph*)]
    \item The equations of motion require the computation of brakets along \( N \)  or \( N-2 \) variables. 
        Even if we are able in practice to compute efficiently a braket along one variable, having \( N \) of them appears challenging as \( N\to \infty \).
    \item The Hamiltonian operator \( H \) given in \eqref{eqn: full schrodinger hamiltonian} is actually a sum of \( \mathcal{O}(N^2) \) operators.
    \item The ansatz \eqref{eqn: final ansatz u} of \( u \) involves symmetric tensorizations of \( N \) variables, i.e. all permutations of \( N \) variables, hence a total number of \( N! \) tensor products that we have to deal with.
    \item The kinetic operator in \( H \) is the possible source of a stepsize restriction that we want to avoid while still using an explicit time integrator.
   \item \( \chi_2 \) is a 2-particle function, hence computationally more demanding than the 1-particle function \( \varphi \).
\end{enumerate}

Some hurdles can be easily dealt with using already existing techniques, some require involved (though conceptually ``simple'') calculations.
We now present a summary of the solutions used to overcome each of the above hurdles. For brevity we write \( \mathcal{O}(1) \) for any quantity that is uniformly bounded with respect to \( N \).
\begin{enumerate}[label=(\alph*)]
    \item The orthogonality of \( \varphi \) against the marginals of \( \chi_2 \) and the fact that \( \norm{\varphi}{2} = 1 \) allow us to reduce the brakets involving \( \mathcal{O}(N) \) variables to brakets involving only \( \mathcal{O}(1) \) variables.
    \item The symmetry of the functions involved in the brakets of the equations of motion allow us to reduce the number of operators from \( \mathcal{O}(N^2) \) to only \( \mathcal{O}(1) \).
    \item The orthogonality and the symmetries in the system allow us to get rid of most permutations in \( u \), leaving us only \( \mathcal{O}(1) \) permutations to deal with.
    \item An operator splitting is performed, to allow us to deal with the kinetic operator in Fourier space and to numerically integrate the equations of motion for the potential operator.
    The equations of motion will be the same as before, where \( H \) is replaced by \( V \) and the Hartree equation by its potential part.
    \item The 2-particle function \( \chi_2 \) is approximated by a low-rank linear combination of tensor products of 1-particle functions. For the kinetic operator, the time integration is then done by 1-particle Fast Fourier Transform. For the potential operator, the low-rank time integration uses a rank-adaptive BUG integrator as developed in \cites{cerutiRankadaptiveRobustIntegrator2022,cerutiTimeIntegrationSymmetric2020}.
\end{enumerate}

\subsection{Outline}
The rest of the paper is organized as follows.
Some notation is given in Section \ref{section: notations}, which ends the introduction.
The equations of motion are derived and studied in Section \ref{section: equations of motion}.
Section \ref{section: computational complexity equations of motion} is dedicated to reducing the computational complexity for the calculation of the right-hand side of the equations of motion and to make it uniformly bounded with respect to \( N \).
The numerical method is given in Section \ref{section: numerical method}, and is obtained mostly from existing techniques and approximations.
Finally, some numerical examples are given in Section \ref{section: numerical experiments} to showcase the satisfying conservation properties and to study the predictions of Bogoliubov theory in situations with \( 1 \ll N < \infty \).

\subsection{Notation}
\label{section: notations}

In this section, \( i \leq j \leq k \leq l \) denote arbitrary non-negative integers.
\begin{itemize}[leftmargin=*]
    \item All position variables are \( d \)-dimensional, and for \( N \) particles we write \( x = (x_1, \dots, x_N) \) where \( x_j \in \Omega := \mathbb{R}^d \).
    
    \item For a function \( f \) that is differentiable with respect to time \( t \), we use the dot notation as a shorthand: \( \dot{f} := \partial_{t} f \).
    
    \item Bra-ket notation: we consider brakets that are anti-linear in their first argument, and linear in their second argument: for \( f,g\in L^2(\Omega^i) \),
    \begin{equation}
        \braket{f}{g} := \int_{\Omega^i} \conj{f(x)} g(x) dx.
    \end{equation}
    The term ``braket'' can be replaced throughout this work by ``\( L^2(\Omega) \) inner product, anti-linear with respect to its first argument and linear with respect to its second argument''.

    \item When \( f\in L^2(\Omega^i) \) and \( g\in L^2(\Omega^j) \) with \( i \leq j \), we may need to compute the inner product between \( f \) and some marginals of \( g \). 
    In this case, we write
    \begin{equation}
        \braket{f(y)}{g(y, \cdot)}_{y} := \int_{\Omega^i} \conj{f(y)} g(y, \cdot) dy 
        \quad \text{ and }
        \braket{f(y)}{g(\cdot, y)}_{y} := \int_{\Omega^i} \conj{f(y)} g(\cdot, y) dy.
    \end{equation}
    Note the difference: the first \( i \) marginals of \( g \) are used in the first braket, while in the second braket it is the last \( i \) marginals of \( g \).
    If \( g \) is symmetric we may use the shorthand
    \begin{equation*}
        \braket{f}{g}_y := \braket{f(y)}{g(y, \cdot)}_y.
    \end{equation*}
    We use the subscript \( \mathbf{1}_y \) to indicate that a braket along the variable \( y \) is equal to one.
    For example, let \( f, g \in L^2(\Omega) \) with \( \braket{f}{f}=1 \), then
    \begin{equation*}
        \braket{f\otimes g}{f \otimes f} = \braket{f\otimes g}{f \otimes f}_{\mathbf{1}_{x_1}},
    \end{equation*}
    because we know that \( \braket{f}{f} = 1 \).
    The computation of the left braket along variables \( x_1, x_2 \) is then reduced to a computation on \( x_2 \) only.

    \item \( L_s^2(\Omega^k) \) is the subspace of \( L^2(\Omega^k) \) consisting of functions that are symmetric under permutations of all of the coordinates, i.e.
    \begin{equation}
        \label{eqn: definition symmetric L2 space}
        L_s^2(\Omega^k) = \bigotimes_{\mathrm{sym}}^k L^2(\Omega).
    \end{equation}
    The inherited $L^2$-norm on $L_s^2(\Omega^k)$ will be simply denoted by $\norm{\cdot}{}$.

    \item For \( i<j \) we write \( \intinterval{i}{j} = (i, i+1, \dots, j) \) as the interval of all integers between \( i \) and \( j \).

    \item We denote by \( S_N \) the set of permutations of \( \intinterval{1}{N} \).
    For \( j\leq N \), we denote by \( S_N^j \) the set of permutations of \( j \) elements chosen among \( \intinterval{1}{N} \). 
    If \( \mu \) is a permutation consisting of \( i \leq N-j \) elements, we write \( S_{\mu} \) the set of permutations of \( \{ \mu(1), \dots, \mu(i) \} \), and use the shorthand \( S^j_{N\setminus \mu} \) to denote the set of permutations of \( j \) elements chosen among the integer set \( \intinterval{1}{N} \setminus \{ \mu(1), \dots, \mu(i) \} \). 

    \item For \( x\in \Omega^l \) and \( \sigma \in S^k_l \), we write \( \sigma(\intinterval{i}{j}) = (\sigma(i), \sigma(i+1), \dots, \sigma(j-1), \sigma(j)) \) and
    \begin{equation*}
        x_{\sigma(\intinterval{i}{j})} := \left(x_{\sigma(i)}, x_{\sigma(i+1)}, \dots, x_{\sigma(j-1)}, x_{\sigma(j)} \right).
    \end{equation*}
    If \( i = 1 \) and \( j=k=l \), we write
    \begin{equation*}
        x_\sigma := (x_{\sigma(1)}, \dots, x_{\sigma(l)}).
    \end{equation*}

    \item For \( x\in \Omega^j \) and \( 1\leq i \leq j \), we write \( x\setminus x_i := \left( x_1, \dots, x_{i-1}, x_{i+1}, \dots, x_j \right) \in \Omega^{j-1} \).
    More generally, for \( i \leq j \) and \( \sigma \in S^i_j \), we write \( x\setminus x_\sigma := ((x\setminus x_{\sigma(1)}) \setminus \dots )\setminus x_{\sigma(i)} \).

    \item For two permutations \( \tau, \sigma \), we write \( \tau(\intinterval{i}{j}) =_s \sigma(\intinterval{i}{j}) \) if the sets \( \{\tau(i), \dots, \tau(j) \} \) and \( \{\sigma(i), \dots, \sigma(j) \} \) are equal, i.e. if the multi-indices \( \tau(\intinterval{i}{j}) \) and \( \sigma(\intinterval{i}{j}) \) are equal up to a permutation.

    \item Let \( f_1, \dots, f_j \in L^2(\Omega) \), we denote the orthogonal complement to the family \( \{f_1, \dots, f_j\} \) in \( L^2(\Omega) \) by
    \begin{equation*}
        \{f_1, \dots, f_j\}^\perp := \left\{ g\in L^2(\Omega) : \braket{g}{f_i} = 0,\, i=1, \dots, j \right\}.
    \end{equation*}
    We define \( \proj_{\text{Span}\{f_1, \dots, f_j\}} \) as the orthogonal projection onto \( \text{Span}\{f_1, \dots, f_j\} \).

    \item Let \( f \in L^2(\Omega^i) \), \( g\in L^2(\Omega^j) \), \( x\in \Omega^j \).
    Let \( \tau \in S^i_j \), we use the shorthand
    \begin{equation*}
        \braket{f}{g(x)}_{x_\tau} := \braket{f(x_\tau)}{g(x)}_{x_\tau} = \int_{\Omega^i} \conj{f(x_{\tau(1)}, \dots, x_{\tau(i)})} g(x_1, \dots, x_j) dx_{\tau(1)} \dots dx_{\tau(i)}.
    \end{equation*}

    \item Symmetric tensor product: for \( f \in L^2(\Omega^i) \) and \( g\in L^2(\Omega^j) \)
    \begin{equation}
        \label{eqn: definition symmetric tensor product}
        (f \otimes_s g)(x_1, \dots, x_{i+j})
        = \frac{1}{\mathsf{K}_i^{i+j}} \sum_{\sigma \in S_{i+j}} (f\otimes g)(x_{\sigma})
        = \frac{1}{\mathsf{K}_i^{i+j}} \sum_{\sigma \in S_{i+j}} f(x_{\sigma(1)}, \dots, x_{\sigma(i)}) g(x_{\sigma(i+1)}, \dots, x_{\sigma(i+j)}),
    \end{equation}
    where \( \mathsf{K}_i^{i+j} = \sqrt{i!j!(i+j)!} = \mathsf{K}_j^{i+j} \).

    \item Let \( f=f(t,x) \) a function that depends on time \( t \), we use the ``dot'' notation to denote a time derivative:
    \begin{equation*}
        \dot{f}(t,x) = \partial_{t} f(t,x).
    \end{equation*}

    \item For any matrix \( M \in \mathbb{C}^{i\times j} \), we denote \( M^\top \) the matrix transpose and \( M^* \) the matrix conjugate transpose:
    \begin{equation*}
        M^\top_{k,l} = M_{l,k} ,\quad
        M^*_{k,l} = \conj{M_{l,k}}.
    \end{equation*}
\end{itemize}

\section{Equations of motion}
\label{section: equations of motion}

In this section we derive equations of motion for the parameters \( (c_0, c_2, \dots, c_{2n}, \chi_2, \chi_4, \dots, \chi_{2n}) \) involved in the ansatz \eqref{eqn: final ansatz u}.
Together with \( \varphi \) solution to \eqref{eqn: hartree equation}, they define the approximate solution \( u \) to the Schrödinger equation \eqref{eqn: general schrodinger equation}.
In Section \ref{section: equations of motion - study of RHS}, the equations of motion are shown to be well-behaved uniformly with respect to \( N \).
We do not yet include the tensorization \eqref{eqn: Bogoliubov tensorization}-\eqref{eqn: Bogoliubov tensorization normalized} into our ansatz, and postpone its discussion to Section \ref{section: DFVP and tensorization}.

\subsection{Derivation}
\label{section: equations of motion - derivation}

Consider in this section the ansatz \eqref{eqn: final ansatz u} with \( n\geq 1 \), for an arbitrary continuously differentiable path of 1-particle functions $\varphi(t)\in L^2(\Omega)$. (Later \( \varphi \) will be chosen as the solution to the Hartree equation \eqref{eqn: hartree equation}, but this is irrelevant for the considerations that follow next.)
We do not assume in this section the tensorization \eqref{eqn: Bogoliubov tensorization}-\eqref{eqn: Bogoliubov tensorization normalized}.
The approximate solution \( u \) reads, omitting the omnipresent argument $t$ in the notation,
\begin{equation}
    \label{eqn: equations of motion - ansatz u}
    u = \sum_{j=0}^n c_{2j} \varphi^{\otimes N-2j} \otimes_s \chi_{2j}
\end{equation}
with coefficients \( c_0, c_2, \dots, c_{2n} \in \mathbb{C} \) and excitation functions \( \chi_{2j} \in L^2_s(\Omega^{2j}) \), \( j=1, \dots, n \), that satisfy the orthogonality \eqref{eqn: bogoliubov orthogonality relation} and are of unit norm.
Since we do not assume the tensorization \eqref{eqn: Bogoliubov tensorization}-\eqref{eqn: Bogoliubov tensorization normalized}, an equation of motion needs to be obtained for every excitation function \( \chi_{2j} \).

We assume that \( c_2, c_4, \dots, c_{2n} \neq 0 \), and determine the time evolution of the unknowns \( c_{2j}, \chi_{2j} \), \( 0 \leq j \leq n \), via the Dirac--Frenkel time-dependent variational principle.
We start by describing all functions of the form \eqref{eqn: equations of motion - ansatz u} as elements of the manifold 
\begin{equation*}
    \mathcal{M}_\varphi = \left\{ \sum_{k=0}^n c_{2k} \varphi^{\otimes N-2k} \otimes_s \chi_{2k} \,\middle|\, c_{2k} \in \mathbb{C}^*,\, \chi_{2k} \in L^2_s(\Omega^{2k}),\, \braket{\varphi(y)}{\chi_{2k}(y,\cdot)}_y = 0,\, \norm{\chi_{2k}}{L^2} = 1\text{ for } k=0, \dots, n \right\},
\end{equation*}
which becomes time-dependent via the time-dependence of the function \( \varphi \).
We refer the reader to \cites{lubichVariationalApproximationsQuantum2004} for details about the Dirac--Frenkel time-dependent variational principle with a time-dependent manifold.
Since the ansatz \eqref{eqn: equations of motion - ansatz u} is imposed for all times, the approximate solution \( u(t) \) is an element of \( \mathcal{M}_{\varphi(t)} \) at all times $t$.
We are now looking for an approximate solution to \( \ii \dot{u} = H u \), as a differentiable path on the time-dependent manifold:
\begin{equation*}
    t\mapsto u(t) = \sum_{k=0}^n c_{2k}(t) \varphi(t)^{\otimes N-2k} \otimes_s \chi_{2k}(t) \in \mathcal{M}_{\varphi(t)}.
\end{equation*}
For any such path, we have (omitting again the ubiquitous argument $t$):
\begin{equation*}
    \dot{u} = \sum_{k=0}^n \left( \dot{c}_{2k} \varphi^{\otimes N-2k} \otimes_s \chi_{2k} + c_{2k} \sqrt{N-2k} (\varphi^{\otimes N-2k-1} \otimes_s \dot{\varphi}) \otimes_s \chi_{2k} + c_{2k} \varphi^{\otimes N-2k} \otimes_s \dot{\chi}_{2k} \right),
\end{equation*}
with \( \Re\braket{\chi_{2k}}{\dot{\chi}_{2k}} = 0 \) due to \( \norm{\chi_{2k}}{2} = 1 \) for all times, and \( \braket{\dot{\varphi}(y)}{\chi_{2k}(y, \cdot)}_y + \braket{\varphi(y)}{\dot{\chi}_{2k}(y, \cdot)}_y = 0 \) due to \( \braket{\varphi(y)}{\chi_{2k}(y, \cdot)}_y = 0 \) for all times.
We have used the associativity of the symmetric tensor product (Lemma \ref{lemma: associativity symmetric tensor product} in Appendix \ref{appendix: technical lemma}), and Lemma \ref{lemma: Derivative of k-fold tensor product} in Appendix \ref{appendix: technical lemma} for computing the time derivative of \( \varphi^{\otimes N-2k} \).
The second constraint on \( \dot{\chi}_{2k} \) implies that
\begin{equation*}
    \dot{\chi}_{2k} = \kappa_{2k} - \sqrt{2k} \varphi \otimes_s \braket{\dot{\varphi}(y)}{\chi_{2k}(y, \cdot)}_y,
\end{equation*}
with \( \kappa_{2k} \in L^2_s(\Omega^{2k}) \) and \( \braket{\varphi(y)}{\kappa_{2k}(y, \cdot)}_y = 0 \).
Without loss of generality, we can impose the gauge condition \( \braket{\chi_{2k}}{\dot\chi_{2k}} = 0 \), which implies \( \braket{\chi_{2k}}{\kappa_{2k}} = 0 \).
Thus,
\begin{multline}
    \dot{u} = \sum_{k=0}^n \left( \dot{c}_{2k} \varphi^{\otimes N-2k} \otimes_s \chi_{2k} + c_{2k} \varphi^{\otimes N-2k} \otimes_s \kappa_{2k} \right) \\
    + \dot{\varphi} \otimes_s \left(\sum_{k=0}^n c_{2k} \sqrt{N-2k} \varphi^{\otimes N-2k-1} \otimes_s \chi_{2k} \right) - \sum_{k=1}^n c_{2k} \sqrt{2k(N-2k+1)} \varphi^{\otimes N-2k+1} \otimes_s \braket{\dot{\varphi}(y)}{\chi_{2k}(y, \cdot)}_y.
\end{multline}
The second line is obtained using Lemmas \ref{lemma: associativity symmetric tensor product} and \ref{lemma: symmetric tensor product of a function and itself} (also in Appendix \ref{appendix: technical lemma}).
The terms on the upper line of the right-hand side are in the tangent space $\mathcal{T}_u \mathcal{M}_\varphi$,
whereas the terms in the lower line only originate from the motion of $\varphi$.
We abbreviate the latter terms as $\Gamma_{\varphi,\dot\varphi}(u)$,
\begin{equation*}
    \Gamma_{\varphi, \dot{\varphi}}(u) = \dot{\varphi} \otimes_s \left(\sum_{k=0}^n c_{2k} \sqrt{N-2k} \varphi^{\otimes N-2k-1} \otimes_s \chi_{2k} \right) - \sum_{k=1}^n c_{2k} \sqrt{2k(N-2k+1)} \varphi^{\otimes N-2k+1} \otimes_s \braket{\dot{\varphi}(y)}{\chi_{2k}(y, \cdot)}_y.
\end{equation*}
The motion of $\varphi$ thus affects the motion of $u$ through the term $\Gamma_{\varphi, \dot{\varphi}}(u)$.
Note that the operator \( \Gamma_{\varphi,\dot\varphi} \) is well-defined on \( \mathcal{M}_\varphi \), since the knowledge of \( u \in \mathcal{M}_\varphi \) and \( \varphi \) determines uniquely the coefficient \( c_0 \) and the products \( c_{2k} \chi_{2k} \) (but not the coefficients \( c_{2k} \) nor the functions \( \chi_{2k} \) separately).
The functions \( \chi_{2k} \) and the coefficients \( c_{2k} \) only appear in \( \Gamma_{\varphi,\dot\varphi}(u) \) as the product \( c_{2k} \chi_{2k} \), making \( \Gamma_{\varphi,\dot\varphi}(u) \) uniquely determined given \( u \in \mathcal{M}_\varphi \) and \( \varphi \).
We now have
\begin{equation}
    \label{eqn: equations of motion - time derivative u}
    \dot{u} = \sum_{k=0}^n \left( \dot{c}_{2k} \varphi^{\otimes N-2k} \otimes_s \chi_{2k} + c_{2k} \varphi^{\otimes N-2k} \otimes_s \kappa_{2k} \right) + \Gamma_{\varphi, \dot{\varphi}}(u),
\end{equation}
where \( \dot{c}_{2k} \in \mathbb{C} \), and the functions \( \kappa_{2k} \in L^2_s(\Omega^{2k}) \) are such that \( \braket{\varphi(y)}{\kappa_{2k}(y, \cdot)}_y = 0 \) and \( \braket{\chi_{2k}}{\kappa_{2k}} = 0 \).
For $u\in \mathcal{M}_\varphi$ we define a subspace of symmetric \( 2k \)-particle functions that contains $\kappa_{2k}$,
\begin{equation*}
    \mathcal{N}^{(2k)}_\varphi(u) := \left\{ f \in L^2_s(\Omega^{2k}) : \braket{\chi_{2k}}{f} = 0,\ \braket{\varphi(y)}{f(y, \cdot)}_y = 0 \right\}
    = \{\chi_{2k}\}^\perp \cap \bigotimes_{\mathrm{sym}}^{2k} \{\varphi\}^\perp.
\end{equation*}
where \( \bigotimes_{\mathrm{sym}}^{2k} \{\varphi\}^\perp \) is shorthand for \( L^2_s(\Omega^{2k}) \cap \bigtimes_{i=1}^{2k} \{\varphi\}^\perp \), and \( f \in \bigtimes_{i=1}^{2k} \{\varphi\}^\perp \) means that the marginals of \( f \) are orthogonal to \( \varphi \).
Conversely, for any $t_0$ and $\kappa_{2k}\in \mathcal{N}^{(2k)}_\varphi(u_0)$ with $u_0\in\mathcal{M}_{\varphi(t_0)}$, a path $t\mapsto u(t)\in \mathcal{M}_{\varphi(t_0)}$ can be constructed such that $\dot u(t_0)$ is given by \eqref{eqn: equations of motion - time derivative u} without the term $\Gamma_{\varphi,\dot\varphi}(u)$.
The tangent space \( \mathcal{T}_u \mathcal{M}_\varphi \) at \( u\in \mathcal{M}_\varphi \) is given by:
\begin{equation}
    \label{eqn: equations of motion - tangent space}
    \mathcal{T}_u \mathcal{M}_\varphi 
    = \left\{ \sum_{k=0}^n \left(\delta c_{2k} \varphi^{\otimes N-2k} \otimes_s \chi_{2k} + c_{2k} \varphi^{\otimes N-2k} \otimes_s \delta\chi_{2k} \right) \,\middle|\, 
        \delta c_{2k} \in \mathbb{C}, \ \delta\chi_{2k} \in \mathcal{N}^{(2k)}_\varphi(u),\ k=0, \dots, n
    \right\}.
\end{equation}
Any function \( \delta u \in \mathcal{T}_u \mathcal{M}_\varphi \) determines uniquely its parameters \( \delta c_{2k} \) and \( \delta\chi_{2k} \).
It can be seen in \eqref{eqn: equations of motion - time derivative u} that the time derivative of any path \( t\mapsto u(t) \in \mathcal{M}_{\varphi(t)} \) is such that \( \dot{u} = \Theta + \Gamma_{\varphi,\dot\varphi}(u) \) for some \( \Theta \in \mathcal{T}_u \mathcal{M}_\varphi \).

The Dirac--Frenkel time-dependent variational principle determines the time derivative \( \dot{u} \) from the condition
\begin{equation}
    \label{eqn: equations of motion - variational principle}
    \norm{\ii \dot u - Hu}{L^2}\ \text{ is minimal subject to }\
    \dot{u} = \Theta + \Gamma_{\varphi,\dot\varphi}(u) \ \text{ for some }\ \Theta \in \mathcal{T}_u \mathcal{M}_\varphi.
\end{equation}
The minimization problem is solved by an orthogonal projection onto \( \mathcal{T}_u \mathcal{M}_\varphi \):
\begin{equation}
    \label{eqn: equations of motion - dirac frenkel variational principle}
    \text{Find } \Theta \in \mathcal{T}_u \mathcal{M}_\varphi \ \text{ such that }\
    \braket{\delta u}{\ii\Theta - \left( H u - \ii\Gamma_{\varphi,\dot\varphi}(u) \right)} = 0,\quad\ \forall \delta u\in \mathcal{T}_u \mathcal{M}_\varphi.
\end{equation}
We prove the following result.
\begin{theorem}
    \label{theorem: equations of motion}
    Let \( \varphi \in C^1(\mathbb{R}_+; L^2(\Omega)) \), \( c_0 \in \mathbb{C} \), \( c_{2k} \in \mathbb{C}^* \) and \( \chi_{2k} \in L^2_s(\Omega^{2k}) \) for \( 1 \leq k \leq n \).
    At \( u \in \mathcal{M}_\varphi \) given by \eqref{eqn: equations of motion - ansatz u}, the Dirac--Frenkel time-dependent variational principle \eqref{eqn: equations of motion - variational principle} yields
    \begin{equation}
        \label{eqn: equations of motion}
        \left\{ \ \ 
            \begin{aligned}
                \dot{c}_{2k} &= \braket{\varphi^{\otimes N-2k} \otimes_s \chi_{2k}}{-\ii H u} - c_{2k} (N-2k) \braket{\varphi}{\dot{\varphi}},\\
                \dot{\chi}_{2k} &= \frac{1}{c_{2k}} \sqrt{\binom{N}{2k}}\, \proj_{\!\mathcal{N}^{(2k)}_\varphi(u)} \braket{\varphi^{\otimes N-2k}}{-\ii H u}_{x_{\intinterval{2k+1}{N}}} - \sqrt{2k}\, \varphi \otimes_s \braket{\dot{\varphi}(y)}{\chi_{2k}(y, \cdot)}_y
                \ \ \text{ on }\ \Omega^{2k}, 
            \end{aligned}
        \right.\quad 0 \leq k \leq n,
    \end{equation}
    where $\proj_{\!\mathcal{N}^{(2k)}_\varphi(u)}$ is the orthogonal projection from $L^2_s(\Omega^{2k})$ onto $\mathcal{N}^{(2k)}_\varphi(u)$.
\end{theorem}

\begin{proof}
    Let \( \Theta, \delta u \in \mathcal{T}_u \mathcal{M}_\varphi \), and write
    \begin{align*}
        \Theta &= \sum_{k=0}^n \left(\dot{c}_{2k} \varphi^{\otimes N-2k} \otimes_s \chi_{2k} + c_{2k} \varphi^{\otimes N-2k} \otimes_s \kappa_{2k} \right),
        &\delta u &= \sum_{k=0}^n \left( \delta c_{2k} \varphi^{\otimes N-2k} \otimes_s \chi_{2k} + c_{2k} \varphi^{\otimes N-2k} \otimes_s \delta\chi_{2k} \right),
    \end{align*}
    for some coefficients \( \dot{c}_{2k}, \delta c_{2k} \in \mathbb{C} \) and functions \( \delta \chi_{2k}, \kappa_{2k} \in \mathcal{N}^{(2k)}_\varphi(u) \), which are uniquely determined by \( \Theta \) and \( \delta u \).
    The variational principle \eqref{eqn: equations of motion - dirac frenkel variational principle} is now equivalent to
    \begin{equation}
        \label{eqn: n>1 - rewrite dirac frenkel variational principle}
        \left\{ \ \ 
            \begin{aligned}
                \braket{\delta c_{2k} \varphi^{\otimes N-2k} \otimes_s \chi_{2k}}{\Theta - \left( -\ii H u - \Gamma_{\varphi,\dot\varphi}(u) \right)} &= 0\quad\ \forall \delta c_{2k} \in \mathbb{C}, \\
                \braket{c_{2k} \varphi^{\otimes N-{2k}} \otimes_s \delta \chi_{2k}}{\Theta - \left( -\ii H u - \Gamma_{\varphi,\dot\varphi}(u) \right)} &= 0\quad\ \forall \delta\chi_{2k} \in \mathcal{N}^{(2k)}_\varphi(u),
            \end{aligned}
        \right.\quad 0\leq k \leq n.
    \end{equation}
    We use the definition of the symmetric tensor product, the orthogonality \eqref{eqn: bogoliubov orthogonality relation}, and the unit norm of \( \varphi \) and \( \chi_2 \) to obtain:
    \begin{equation*}
        \braket{\varphi^{\otimes N-2k}\otimes_s \chi_{2k}}{\varphi^{\otimes N-2j}\otimes_s \chi_{2j}} = \delta_{j,k},\quad \text{ and }\quad
        \braket{\varphi^{\otimes N-2k} \otimes_s \chi_{2k}}{\varphi^{\otimes N-2j} \otimes_s \delta\chi_{2j}} = 0,\quad 0 \leq j,k \leq n.
    \end{equation*}
    The system \eqref{eqn: n>1 - rewrite dirac frenkel variational principle} implies
    \begin{equation}
        \left\{ \ \ 
            \begin{aligned}
                \dot{c}_{2k} &= \braket{\varphi^{\otimes N-2k} \otimes_s \chi_{2k}}{-\ii H u - \Gamma_{\varphi,\dot\varphi}(u)}, \\
                c_{2k} \kappa_{2k} &= \sqrt{\binom{N}{2k}} \proj_{\mathcal{N}^{(2k)}_\varphi(u)} \braket{\varphi^{\otimes N-2k}}{-\ii H u - \Gamma_{\varphi,\dot\varphi}(u)}_{x_{\intinterval{2k+1}{N}}} \quad \text{ on } \Omega^{2k},
            \end{aligned}
        \right.\quad 0\leq k \leq n.
    \end{equation}
    Direct computations give
    \begin{equation*}
        \braket{\varphi^{\otimes N-2j} \otimes_s \chi_{2j}}{\Gamma_{\varphi, \dot{\varphi}}} = c_{2j} (N-2j) \braket{\varphi}{\dot{\varphi}}
        \quad\text{ and }\quad
        \proj_{\mathcal{N}^{(2k)}_\varphi}\braket{\varphi^{\otimes N-2k}}{\Gamma_{\varphi, \dot{\varphi}}}_{x_{\intinterval{2k+1}{N}}} = 0.
    \end{equation*}
    We finally get
    \begin{equation*}
        \left\{ \ \ 
            \begin{aligned}
                \dot{c}_{2k} &= \braket{\varphi^{\otimes N-2k} \otimes_s \chi_{2k}}{-\ii H u} -  c_{2k}(N-2k) \braket{\varphi}{\dot{\varphi}}, \\
                \kappa_{2k} &= \frac{1}{c_{2k}}  \sqrt{\binom{N}{2k}} \proj_{\mathcal{N}^{(2k)}_\varphi(u)} \braket{\varphi^{\otimes N-2k}}{-\ii H u}_{x_{\intinterval{2k+1}{N}}} \quad \text{ on } \Omega^{2k},
            \end{aligned}
        \right.\quad 0 \leq k \leq n,
    \end{equation*}
    and the result follows from \( \dot{\chi}_{2k} = \kappa_{2k} - \sqrt{2k} \varphi \otimes_s \braket{\dot{\varphi}}{\chi_{2k}(y, \cdot)}_y \).
\end{proof}

\subsection{Tensorization and variational principle, and a variational crime}
\label{section: DFVP and tensorization}

In the previous section, we did not impose the tensorization \eqref{eqn: Bogoliubov tensorization}-\eqref{eqn: Bogoliubov tensorization normalized}.
In this section, we explain why we omitted this condition in the Dirac--Frenkel time-dependent variational principle.
Essentially, the difficulties boil down to ``dealing with symmetries can be tricky''.
In order to illustrate this, we will try to obtain the equations of motion with \( n>1 \) and show where it fails.

Using the tensorization property \eqref{eqn: Bogoliubov tensorization}, the ansatz \eqref{eqn: final ansatz u} reads \( u = \sum_{k=0}^n c_{2k} \varphi^{\otimes N-2k} \otimes_s \chi_2^{\otimes k} \).
Up to a modification of the coefficients \( c \) that accounts for the normalization \eqref{eqn: Bogoliubov tensorization normalized}, this is exactly the same ansatz as \eqref{eqn: final ansatz u} and is easier to manipulate in this section.
To underline the difficulties and show how they stem from the tensor product \( \chi_2^{\otimes k} \), we simplify our situation: we assume in this section that the coefficients \( c \) and the condensate function \( \varphi \) are given and constant with respect to time.
We apply the Dirac--Frenkel time-dependent variational principle to the manifold
\begin{equation*}
    \mathcal{M}_{c, \varphi} = \left\{ \sum_{k=0}^n c_{2k} \varphi^{\otimes N-2k} \otimes_s \chi_2^{\otimes k} \,\middle|\,
        \chi_2 \in L^2_s(\Omega^2),\ \braket{\varphi(y)}{\chi_2(y, \cdot)}_y = 0,\ \norm{\chi_2}{2} = 1
    \right\}.
\end{equation*}
For a variation \( \delta\chi_2 \) of \( \chi_2 \), we obtain the variation of \( \chi_2^{\otimes k} \): \(\delta(\chi_2^{\otimes k}) = \sum_{l=1}^k \chi_2^{l-1} \otimes \delta\chi_2 \otimes \chi_2^{k-l} \).
By symmetry, we get the variation \( \delta(c_{2k} \varphi^{\otimes N-2k}\otimes_s \chi_2^{\otimes k}) = k c_{2k} \varphi^{\otimes N-2k} \otimes_s (\delta\chi_2 \otimes \chi_2^{\otimes k-1}) \).
The tangent space \( \mathcal{T}_u \mathcal{M}_{c,\varphi} \) at \( u\in \mathcal{M}_{c,\varphi} \) is thus given by
\begin{equation*}
    \mathcal{T}_u \mathcal{M}_{c,\varphi}
    = \left\{ \sum_{k=1}^n k c_{2k} \varphi^{\otimes N-2k} \otimes_s \left( \delta \chi_2 \otimes \chi_2^{\otimes k-1} \right) \,\middle|\,
        \delta\chi_2 \in L^2_s(\Omega^2), \ \braket{\varphi(y, \cdot)}{\delta\chi_2(y, \cdot)}_y = 0,\, \Re \braket{\chi_2}{\delta\chi_2} = 0
    \right\}.
\end{equation*}
The Dirac--Frenkel time-dependent variational principle now consists in looking for \( \Theta \in \mathcal{T}_u \mathcal{M}_{c,\varphi} \) such that 
\begin{equation*}
    \braket{\delta u}{\ii \Theta - Hu} = 0, \quad \forall \delta u \in \mathcal{T}_u \mathcal{M}_{c,\varphi}.
\end{equation*}
Let
\begin{equation*}
    \delta u = \sum_{k=1}^n k c_{2k} \varphi^{\otimes N-2k} \otimes_s \left( \delta \chi_2 \otimes \chi_2^{\otimes k-1} \right)
    \quad \text{and} \quad
    \Theta = \sum_{j=1}^n j c_{2j} \varphi^{\otimes N-2j} \otimes_s \left( \kappa_2 \otimes \chi_2^{\otimes j-1} \right),
\end{equation*}
for some functions \( \delta \chi_2, \kappa_2 \in L^2_s(\Omega^2) \) that are admissible variations for the tangent space \( \mathcal{T}_u \mathcal{M}_{c, \varphi} \).
The orthogonalities between \( \varphi \) and the marginals of the functions \( \chi_2, \delta \chi_2, \kappa_2 \) yield
\begin{align*}
    \braket{\delta u}{\Theta}
    &= \sum_{k=1}^n C_{N,k} k^2 \abs{c_{2k}}^2 \braket{\delta \chi_2 \otimes_s \chi_2^{\otimes k-1}}{\kappa_2 \otimes_s \chi_2^{\otimes k-1}},
\end{align*}
where the coefficient \( C_{N, k} \) is a combinatorial coefficient depending on \( N \) and \( k \).
Therefore, we have to compute the following sum:
\begin{align*}
    \braket{\delta u}{\Theta}
    &= \sum_{k=1}^n (2k)! C_{N,k} k^2 \abs{c_{2k}}^2 \braket{\delta \chi_2 \otimes \chi_2^{\otimes k-1}}{\kappa_2 \otimes_s \chi_2^{\otimes k-1}}.
\end{align*}
Coming back to the Dirac--Frenkel time-dependent variational principle, we obtain the following condition:
\begin{equation*}
    \braket{\delta\chi_2}{\sum_{k=0}^n (2k)! C_{N, k} k^2 \abs{c_{2k}}^2 \braket{\chi_2^{\otimes k-1}}{\kappa_2 \otimes_s \chi_2^{\otimes k-1}}_{x_{\intinterval{3}{2k}}} - F}_{x_1, x_2} = 0,
\end{equation*}
that must be satisfied for all variations \( \delta\chi_2 \in L^2_s(\Omega^2) \) that are admissible for the tangent space \( \mathcal{T}_u \mathcal{M}_{c,\varphi} \).
The function \( F \) is an explicitly computable function involving \( Hu \).
The innermost braket is a linear combination of terms of the form
\begin{equation*}
    \kappa_2,\quad \braket{\chi_2(y_1, \cdot) \chi_2(y_2, \cdot)}{\kappa_2(y_1, y_2)}_{y_1, y_2},\quad
    \braket{\chi_2(y_1, \cdot) \chi_2(y_2, \cdot) \chi_2(y_3, y_4)}{\kappa_2(y_1, y_3)\chi_2(y_2, y_4)}_{y_1, y_2, y_3, y_4},\quad \dots
\end{equation*}
The sum \( \sum_{k=0}^n (2k)! C_{N, k} k^2 \abs{c_{2k}}^2 \braket{\chi_2^{\otimes k-1}}{\kappa_2 \otimes_s \chi_2^{\otimes k-1}}_{x_{\intinterval{3}{2k}}} \) can ultimately be represented by a linear operator depending on \( \chi_2 \) and acting on \( \kappa_2 \).
In order to obtain equations of motion, this linear operator needs to be inverted and its inverse is non-trivial to obtain and can be arbitrarily ill-conditioned (if it exists at all). We conclude that the Dirac--Frenkel time-dependent variational principle is not applicable under the tensorization constraints \eqref{eqn: Bogoliubov tensorization}-\eqref{eqn: Bogoliubov tensorization normalized}.

\medskip
We therefore commit a variational crime: {\it We replace $\chi_4,\dots,\chi_{2n}$ appearing in $u$ on the right-hand side of the differential equation for $\chi_2$ by the corresponding tensorized functions \eqref{eqn: Bogoliubov tensorization}-\eqref{eqn: Bogoliubov tensorization normalized}. We thus end up with a system of differential equations for the coefficients $c_0,c_2,\dots,c_{2n}$ and the excitation function $\chi_2$  (as given in Theorem \ref{theorem: equations of motion}) and algebraic equations for $\chi_4,\dots,\chi_{2n}$ (as given by \eqref{eqn: Bogoliubov tensorization}-\eqref{eqn: Bogoliubov tensorization normalized}).}

We also need to specify the angle \( \theta_{2j} \) that appears in the definition \eqref{eqn: Bogoliubov tensorization normalized} of \( \chi_{2j} \).
We choose this angle such that \( \braket{\chi_{2j}}{\dot{\chi}_{2j}} = 0 \), so that the equations of motion for \( {c}_{2j} \) (which have been derived using this equality) remain valid.

\begin{lemma}
    \label{lemma: equations of motion angles}
    Assume \eqref{eqn: Bogoliubov tensorization}-\eqref{eqn: Bogoliubov tensorization normalized} hold.
    The angles \( \theta_{2j} \) that enforce the orthogonalities \( \braket{\chi_{2j}}{\dot{\chi}_{2j}} = 0 \) are solutions to the following differential equation:
    \begin{equation}
        \label{eqn: equations of motion angles}
        \dot{\theta}_{2j} = \frac{\Im \braket{\widetilde{\chi}_{2j}}{\dot{\widetilde{\chi}_{2j}}}}{\braket{\widetilde{\chi}_{2j}}{\widetilde{\chi}_{2j}}}
            = \frac{(2j)!}{2^j (j-1)!} \frac{\Im\braket{\widetilde{\chi}_{2j}}{\dot{\chi}_2 \otimes \chi_2^{\otimes j-1}}}{\braket{\widetilde{\chi}_{2j}}{\widetilde{\chi}_{2j}}}, \quad j\geq1.
    \end{equation}
\end{lemma}

\begin{proof}
    The orthogonality \( \braket{\chi_2}{\dot{\chi}_2} = 0 \) is already satisfied by the equation of motion for \( \chi_2 \) given in \eqref{eqn: equations of motion}.
    For \( j>1 \), we have,
        \begin{equation*}
        \dot{\chi}_{2j}
        = -\ii \dot{\theta}_{2j} \chi_{2j} -\frac{e^{-\ii\theta_{2j}}}{2} \frac{\frac{\partial}{\partial t} \braket{\widetilde{\chi}_{2j}}{\widetilde{\chi}_{2j}}}{\braket{\widetilde{\chi}_{2j}}{\widetilde{\chi}_{2j}}^{3/2}} \widetilde{\chi}_{2j} + \frac{e^{-\ii\theta_{2j}} \dot{\widetilde{\chi}_{2j}}}{\sqrt{\braket{\widetilde{\chi}_{2j}}{\widetilde{\chi}_{2j}}}}
        = -\ii\dot{\theta}_{2j} \chi_{2j} + e^{-\ii\theta_{2j}} \left(\frac{\dot{\widetilde{\chi}_{2j}}}{\sqrt{\braket{\widetilde{\chi}_{2j}}{\widetilde{\chi}_{2j}}}} - \frac{\Re\braket{\widetilde{\chi}_{2j}}{\dot{\widetilde{\chi}_{2j}}}}{\braket{\widetilde{\chi}_{2j}}{\widetilde{\chi}_{2j}}^{3/2}} \widetilde{\chi}_{2j} \right),
    \end{equation*}
    and hence
    \begin{equation*}
        \braket{\chi_{2j}}{\dot{\chi}_{2j}}
        = -\ii \dot{\theta}_{2j} - \frac{\Re\braket{\widetilde{\chi}_{2j}}{\dot{\widetilde{\chi}_{2j}}}}{\braket{\widetilde{\chi}_{2j}}{\widetilde{\chi}_{2j}}} + \frac{\braket{\widetilde{\chi}_{2j}}{\dot{\widetilde{\chi}_{2j}}}}{\braket{\widetilde{\chi}_{2j}}{\widetilde{\chi}_{2j}}}
        = -\ii\dot{\theta}_{2j} + \ii \frac{\Im\braket{\widetilde{\chi}_{2j}}{\dot{\widetilde{\chi}_{2j}}}}{\braket{\widetilde{\chi}_{2j}}{\widetilde{\chi}_{2j}}} \in \ii \mathbb{R}.
    \end{equation*}
    The differential equation for \( \theta_{2j} \) is obtained by imposing the condition \( \braket{\chi_{2j}}{\dot{\chi}_{2j}} = 0 \).
    Now use \eqref{eqn: Bogoliubov tensorization}:
    \begin{equation*}
        \braket{\widetilde{\chi}_{2j}}{\dot{\widetilde{\chi}_{2j}}}
        = \sum_{\tau\in P_{2j}} \braket{\widetilde{\chi}_{2j}}{\sum_{k=1}^j (\chi_2^{\otimes k-1} \otimes \dot{\chi}_2 \otimes \chi_2^{\otimes j-k})(x_\tau)}
        = j \sum_{\tau\in P_{2j}} \braket{\widetilde{\chi}_{2j}}{\dot{\chi}_2 \otimes \chi_2^{\otimes j-1}},
    \end{equation*}
    where the last equality is obtained using the symmetry of \( \widetilde{\chi}_{2j} \), and we conclude using that \( \abs{P_{2j}} = \frac{(2j)!}{2^j j!} \).
\end{proof}

\subsection{Behaviour with respect to \texorpdfstring{\( N \)}{N}}
\label{section: equations of motion - study of RHS}

At first glance, the time derivatives of \( c_{2k} \) and \( \chi_{2k} \) given in \eqref{eqn: equations of motion} seem to become large when \( N \) grows large.
Fortunately, we can show that the derivatives \( \dot{\chi}_{2k} \) are uniformly bounded with respect to \( N \) (Theorem \ref{theorem: setting n>1 no tensorization - chi2dot uniformly bounded wrt N}) and that the vector of coefficients \( \mathbf{c}=(c_{2j})_{j=0}^n \) follows a smooth unitary evolution (Theorem \ref{theorem: unitary evolution coeffs}).

We start with the equations of motion for the excitation functions \( \chi_{2k} \) given in \eqref{eqn: equations of motion}, and show that they have bounded derivatives as \( N\to\infty \).
We have the following result:
\begin{theorem}
    \label{theorem: setting n>1 no tensorization - chi2dot uniformly bounded wrt N}
    Let \( c_0 \in \mathbb{C} \), \( c_2, \dots, c_{2n} \in \mathbb{C}^* \), \( \varphi, \phi \in L^2(\Omega) \), and \( \chi_{2k} \in L^2_s(\Omega^{2k}) \) such that
    \begin{equation*}
        \braket{\varphi(y)}{\chi_{2k}(y, \cdot)}_y = 0, \quad 1 \leq k \leq n.
    \end{equation*}
    Define \( u \) via \eqref{eqn: final ansatz u}, then for all \( 1 \leq k \leq n \), the quantity
    \begin{equation*}
        \frac{1}{c_{2k}} \sqrt{\binom{N}{2k}}\, \proj_{\!\mathcal{N}^{(2k)}_\varphi(u)} \braket{\varphi^{\otimes N-2k}}{-\ii H u}_{x_{\intinterval{2k+1}{N}}} - \sqrt{2k} \varphi \otimes_s \braket{\phi(y)}{\chi_{2k}(y, \cdot)}_y
    \end{equation*}
    is uniformly bounded with respect to \( N \).
\end{theorem}

\begin{remark}
    Theorem \ref{theorem: setting n>1 no tensorization - chi2dot uniformly bounded wrt N}  shows that the derivatives \( \dot{\chi}_{2k} \) are uniformly bounded with respect to \( N \), by taking \( \varphi \) as the solution to the Hartree equation \eqref{eqn: hartree equation} and \( \phi \) equal to \( \dot{\varphi} \).
\end{remark}

The proof of Theorem \ref{theorem: setting n>1 no tensorization - chi2dot uniformly bounded wrt N} can be greatly simplified if we show first that the number of ``effective'' one- and two-body operators in \( \braket{\varphi^{\otimes N-2}}{-\ii H u}_{x_{\intinterval{3}{N}}} \) is uniformly bounded with respect to \( N \).
More precisely, we first show that the Hamiltonian
\begin{equation*}
    H = \sum_{i=1}^N -\Delta_{x_i} + \frac{1}{N-1} \sum_{1\leq k < l \leq N} v(x_k - x_l) 
\end{equation*}
can be replaced in the computation of the braket \( \braket{\varphi^{\otimes N-2}}{-\ii H u}_{x_{\intinterval{3}{N}}} \) by an ``effective'' Hamiltonian made of a sum of only a small number of operators, uniformly in \( N \).
The essential property that allows this is the symmetry with respect to permutations of variables.
Lemma \ref{lemma: reducing number of effective kinetic operators} gives the ``effective'' kinetic operator, and Lemma \ref{lemma: reducing number of effective potential operators} the ``effective'' potential operator.
We use the notations \( \varphi \) and \( \chi_{2p} \) in both Lemmas merely to stick to our context, however it is not required that they satisfy any of the orthogonality or normalization previously used: only the symmetry of \( \chi_{2p} \) is used.

\begin{lemma}
    \label{lemma: reducing number of effective kinetic operators}
    Define integers \( m\in \intinterval{0}{N} \), \( p\in \intinterval{0}{\frac{N-m}{2} } \), and let \( \varphi \in L^2(\Omega) \), \( \chi_{2p} \in L_s^2(\Omega^{2p}) \) and \( w \in L^2_s(\Omega^N) \).
    Then,
    \begin{equation}
        \label{eqn: use symmetry of u to simplify the computation of braket with Delta u on the rhs}
        \braket{\varphi^{\otimes N-m-2p} \otimes \chi_{2p}}{-\Delta w(x)}_{x_{\intinterval{m+1}{N}}}
        = \begin{aligned}[t]
            &\sum_{1 \leq k \leq m} \braket{\varphi^{\otimes N-m-2p} \otimes \chi_{2p}}{-\Delta_{x_k} w(x)}_{x_{\intinterval{m+1}{N}}} \\
            &\quad+ (N-m-2p) \braket{\varphi^{\otimes N-m-2p} \otimes \chi_{2p}}{-\Delta_{x_{m+1}} w(x)}_{x_{\intinterval{m+1}{N}}} \\
            &\quad+ 2p \braket{\varphi^{\otimes N-m-2p} \otimes \chi_{2p}}{-\Delta_{x_{N-2p+1}} w(x)}_{x_{\intinterval{m+1}{N}}}.
        \end{aligned}
    \end{equation}
\end{lemma}

\begin{remark}
    The kinetic operator on the left-hand side is \( -\Delta = \sum_{i=1}^N -\Delta_{x_i} \) and involves \( N \) one-body operators, while the ``effective'' kinetic operator in the right-hand side involves only \( m+2 \) operators.
\end{remark}

\begin{proof}
    We have 
    \begin{align}
        \braket{\varphi^{\otimes N-m-2p} \otimes \chi_{2p}}{-\Delta w(x)}_{x_{\intinterval{m+1}{N}}}
        &= \sum_{1\leq k \leq N} \braket{\varphi^{\otimes N-m-2p} \otimes \chi_{2p}}{-\Delta_{x_k} w(x)}_{x_{\intinterval{m+1}{N}}} \nonumber  \\
        &= \begin{aligned}[t]
            &\sum_{1\leq k \leq m} \braket{\varphi^{\otimes N-m-2p} \otimes \chi_{2p}}{-\Delta_{x_k} w(x)}_{x_{\intinterval{m+1}{N}}} \\
            &\quad+ \sum_{m+1 \leq k \leq N-2p} \braket{\varphi^{\otimes N-m-2p} \otimes \chi_{2p}}{-\Delta_{x_k} w(x)}_{x_{\intinterval{m+1}{N}}} \\
            &\quad+ \sum_{N-2p+1 \leq k} \braket{\varphi^{\otimes N-m-2p} \otimes \chi_{2p}}{-\Delta_{x_k} w(x)}_{x_{\intinterval{m+1}{N}}}.
        \end{aligned}
        \label{eqn: proof reduction nb operators laplacian - expansion total sum}
    \end{align}
    For any \( k\in \intinterval{m+1}{N-2p} \), by symmetry of \( w \)
    \begin{equation*}
        \braket{\varphi^{\otimes N-m-2p} \otimes \chi_{2p}}{-\Delta_{x_k} w(x)}_{x_{\intinterval{m+1}{N}}}
        = \braket{\varphi^{\otimes N-m-2p} \otimes \chi_{2p}}{-\Delta_{x_k} w(x_1, \dots, x_{k}, \dots, x_{m+1}, \dots, x_N)}_{x_{\intinterval{m+1}{N}}}.
    \end{equation*}
    Only the variables \( x_k \) and \( x_{m+1} \) have been exchanged.
    Due to the symmetry of \( \varphi^{\otimes N-m-2p}(x_{m+1}, \dots, x_{N-2p}) \), we perform a change of variables \( x_k \leftrightarrow x_{m+1} \) and obtain
    \begin{equation*}
        \braket{\varphi^{\otimes N-m-2p} \otimes \chi_{2p}}{-\Delta_{x_k} w(x)}_{x_{\intinterval{m+1}{N}}}
        = \braket{\varphi^{\otimes N-m-2p} \otimes \chi_{2p}}{-\Delta_{x_{m+1}} w(x)}_{x_{\intinterval{m+1}{N}}}.
    \end{equation*}
    Note that the right-hand side does not depend on \( k \) anymore, thus the second sum in \eqref{eqn: proof reduction nb operators laplacian - expansion total sum} simplifies to 
    \begin{equation*}
        (N-m-2p) \braket{\varphi^{\otimes N-m-2p} \otimes \chi_{2p}}{-\Delta_{x_{m+1}} w(x)}_{x_{\intinterval{m+1}{N}}}.
    \end{equation*}
    For the third sum in \eqref{eqn: proof reduction nb operators laplacian - expansion total sum}, we use the same arguments: for \( k\in \intinterval{N-2p+1}{N} \), the symmetry of \( w \) gives us
    \begin{equation*}
        w(x) = w(x_1, \dots, x_{N-2p+1}, \dots, x_k, \dots, x_N) = w(x_1, \dots, x_k, \dots, x_{N-2p+1}, \dots, x_N),
    \end{equation*}
    where only the variables \( x_k \) and \( x_{N-2p+1} \) have been exchanged.
    The symmetry of \( \chi_{2p} \) allows us to perform the change of variables \( x_k \leftrightarrow x_{N-2p+1} \), and the kinetic operator is now on variable \( x_{N-2p+1} \).
    Finally, the third sum in \eqref{eqn: proof reduction nb operators laplacian - expansion total sum} simplifies to 
    \begin{equation*}
        2p \braket{\varphi^{\otimes N-m-2p} \otimes \chi_{2p}}{-\Delta_{x_{N-2p+1}} w(x)}_{x_{\intinterval{m+1}{N}}}.
    \end{equation*}
\end{proof}

Similar ideas can be used for the two-body potential operator.
\begin{lemma}
    \label{lemma: reducing number of effective potential operators}
    Define integers \( m\in \intinterval{0}{N} \), \( p\in \intinterval{0}{\frac{N-m}{2} } \), and let \( \varphi \in L^2(\Omega) \), \( \chi_{2p} \in L_s^2(\Omega^{2p}) \) and \( w \in L^2_s(\Omega^N) \).
    Then,
    \begin{multline*}
        \label{eqn: use symmetry of u to simplify the computation of braket with Vu on the rhs}
            (N-1)\braket{\varphi^{\otimes N-m-2p} \otimes \chi_{2p}}{Vw(x)}_{x_{\intinterval{m+1}{N}}} \\
            = \begin{aligned}[t]
                &\sum_{1 \leq k < l \leq m} \braket{\varphi^{\otimes N-m-2p} \otimes \chi_{2p}}{v(x_{k} - x_{l}) w(x)}_{x_{\intinterval{m+1}{N}}} \\
                &\quad+ (N-m-2p) \sum_{k=1}^m \braket{\varphi^{\otimes N-m-2p} \otimes \chi_{2p}}{v(x_k - x_{m+1}) w(x)}_{x_{\intinterval{m+1}{N}}} \\
                &\quad + 2p \sum_{k=1}^m \braket{\varphi^{\otimes N-m-2p} \otimes \chi_{2p}}{v(x_k - x_{N-2p+1}) w(x)}_{x_{\intinterval{m+1}{N}}} \\
                &\quad+ \frac{(N-m-2p-1)(N-m-2p)}{2} \braket{\varphi^{\otimes N-m-2p} \otimes \chi_{2p}}{v(x_{m+1} - x_{m+2}) w(x)}_{x_{\intinterval{m+1}{N}}} \\
                &\quad+ 2p(N-m-2p) \braket{\varphi^{\otimes N-m-{2p}} \otimes \chi_{2p}}{v(x_{m+1} - x_{N-2p+1}) w(x)}_{x_{\intinterval{m+1}{N}}} \\
                &\quad+ \frac{2p(2p-1)}{2} \braket{\varphi^{\otimes N-m-{2p}} \otimes \chi_{2p}}{v(x_{N-2p+1} - x_{N-2p+2}) w(x)}_{x_{\intinterval{m+1}{N}}}.
            \end{aligned}
    \end{multline*}
\end{lemma}

\begin{remark}
    The potential operator on the left-hand side is \( V = \sum_{1\leq k < l \leq N} v(x_k - x_l) \) and involves \( N(N-1)/2 \) two-body operators, while the ``effective'' potential operator on the right-hand side involves only \( m(m-1)/2 + 2m + 3 \) operators.
\end{remark}

\begin{proof}
    The proof is very similar to the proof of Lemma \ref{lemma: reducing number of effective kinetic operators}, we just need to deal with a two-body operator instead of a one-body operator.
    We have
    \begin{equation}
        \begin{aligned}
            (N-1)\braket{\varphi^{\otimes N-m-2p} \otimes \chi_{2p}}{Vw(x)}_{x_{\intinterval{m+1}{N}}}
            &=\sum_{k=1}^m \sum_{l=k+1}^m \braket{\varphi^{\otimes N-m-2p} \otimes \chi_{2p}}{v(x_k - x_l) w(x)}_{x_{\intinterval{m+1}{N}}} \\
            &\quad+ \sum_{k=1}^m \sum_{l=m+1}^N \braket{\varphi^{\otimes N-m-2p} \otimes \chi_{2p}}{v(x_k - x_l) w(x)}_{x_{\intinterval{m+1}{N}}} \\
            &\quad+ \sum_{k=m+1}^{N-1} \sum_{l=k+1}^N \braket{\varphi^{\otimes N-m-2p} \otimes \chi_{2p}}{v(x_k - x_l) w(x)}_{x_{\intinterval{m+1}{N}}}.
        \end{aligned}
        \label{eqn: use symmetry of u to simplify the computation of braket with V on the rhs - split sums over k and l}
    \end{equation}
    We distinguish the following situations:
    \begin{itemize}
        \item \( k \in \intinterval{1}{m}, l\in \intinterval{m+1}{N-2p} \)
        \item \( k \in \intinterval{1}{m}, l\in \intinterval{N-2p+1}{N} \)
        \item \( k \in \intinterval{m+1}{N-2p}, l\in \intinterval{k+1}{N-2p} \)
        \item \( k \in \intinterval{m+1}{N-2p}, l\in \intinterval{N-2p+1}{N} \)
        \item \( k \in \intinterval{N-2p+1}{N}, l\in \intinterval{k+1}{N} \),
    \end{itemize}
    and for each of them the same arguments as before can be used.
\end{proof}

Lemmas \ref{lemma: reducing number of effective kinetic operators} and \ref{lemma: reducing number of effective potential operators} allow us to prove Theorem \ref{theorem: setting n>1 no tensorization - chi2dot uniformly bounded wrt N} in a rather concise way.

\begin{proof}[Proof of Theorem \ref{theorem: setting n>1 no tensorization - chi2dot uniformly bounded wrt N}]
    The functions \( \varphi, \phi \) and \( \chi_{2k} \) are of unit norm, therefore the quantity \( \sqrt{2k} \varphi \otimes_s \braket{\phi}{\chi_{2k}(y, \cdot)}_y \) is uniformly bounded with respect to \( N \).
    We now need to show that \( \frac{1}{c_{2k}} \sqrt{\binom{N}{2k}} \proj_{\mathcal{N}^{(2k)}_\varphi(u)} \braket{\varphi^{\otimes N-2k}}{-\ii H u}_{x_{\intinterval{2k+1}{N}}} \) is uniformly bounded with respect to \( N \).
    By Lemmas \ref{lemma: reducing number of effective kinetic operators} and \ref{lemma: reducing number of effective potential operators}, it is sufficient to show that it is the case for the following quantities:
    \begin{enumerate}[label=(\alph{*})]
        \item \( \sqrt{\binom{N}{2k}} \proj_{\mathcal{N}^{(2k)}_\varphi(u)}\braket{\varphi^{\otimes N-2k}}{-\Delta_{x_i} u}_{x_{\intinterval{2k+1}{N}}} \), \( i=1, \dots, 2k \).
        
        \item \( (N-2k)\sqrt{\binom{N}{2k}} \proj_{\mathcal{N}^{(2k)}_\varphi(u)}\braket{\varphi^{\otimes N-2k}}{-\Delta_{x_{2k+1}} u}_{x_{\intinterval{2k+1}{N}}} \).
        
        \item \( \frac{\sqrt{\binom{N}{2k}}}{N-1} \proj_{\mathcal{N}^{(2k)}_\varphi(u)}\braket{\varphi^{\otimes N-2k}}{v(x_{i_1} - x_{i_2}) u}_{x_{\intinterval{2k+1}{N}}} \), \( 1 \leq i_1 < i_2 \leq 2k \).
        
        \item \( \frac{\sqrt{\binom{N}{2k}}}{N-1} (N-2k) \proj_{\mathcal{N}^{(2k)}_\varphi(u)}\braket{\varphi^{\otimes N-2k}}{v(x_i - x_{2k+1}) u}_{x_{\intinterval{2k+1}{N}}} \), \( 1 \leq i \leq 2k \).
        
        \item \( \frac{\sqrt{\binom{N}{2k}}}{N-1} \frac{(N-2k)(N-2k-1)}{2} \proj_{\mathcal{N}^{(2k)}_\varphi(u)}\braket{\varphi^{\otimes N-2k}}{v(x_{2k+1} - x_{2k+2}) u}_{x_{\intinterval{2k+1}{N}}} \).
    \end{enumerate}

    In what follows, we will use the orthogonality \eqref{eqn: bogoliubov orthogonality relation}, the symmetry of all excitation functions \( \chi_{2j} \), and the fact that the orthogonal projection onto \( \mathcal{N}^{(2k)}_\varphi(u) \) eliminates all multiples of \( \chi_{2k} \) as well as all functions of \( 2k \) variables with a marginal that is a multiple of \( \varphi \).
    These arguments will be used many times throughout this proof. We gather them here to avoid repeating them every time they are used.

    For the term (a), with \( 1 \leq i \leq 2k \), we have
    \begin{multline*}
        \sqrt{\binom{N}{2k}} \proj_{\mathcal{N}^{(2k)}_\varphi(u)}\braket{\varphi^{\otimes N-2k}}{-\Delta_{x_i} u}_{x_{\intinterval{2k+1}{N}}} \\
        = \sqrt{\binom{N}{2k}} \sum_{j=0}^n c_{2j} \proj_{\mathcal{N}^{(2k)}_\varphi(u)}\braket{\varphi^{\otimes N-2k}}{-\Delta_{x_i} (\varphi^{\otimes N-2j} \otimes_s \chi_{2j})}_{x_{\intinterval{2k+1}{N}}}.
    \end{multline*}
    The right-hand side has a vanishing summand for \( j\neq k \), hence
    \begin{align*}
        \sqrt{\binom{N}{2k}} \proj_{\mathcal{N}^{(2k)}_\varphi(u)}\braket{\varphi^{\otimes N-2k}}{-\Delta_{x_i} u}_{x_{\intinterval{2k+1}{N}}} 
        &= c_{2k} \proj_{\mathcal{N}^{(2k)}_\varphi(u)} (-\Delta_{x_i} \chi_{2k})
    \end{align*}
    is uniformly bounded with respect to \( N \). 
    For the term (b), we have
    \begin{multline*}
        (N-2k)\sqrt{\binom{N}{2k}} \proj_{\mathcal{N}^{(2k)}_\varphi(u)}\braket{\varphi^{\otimes N-2k}}{-\Delta_{x_{2k+1}} u}_{x_{\intinterval{2k+1}{N}}} \\
        = (N-2k)\sqrt{\binom{N}{2k}} \sum_{j=0}^n c_{2j} \proj_{\mathcal{N}^{(2k)}_\varphi(u)}\braket{\varphi^{\otimes N-2k}}{-\Delta_{x_{2k+1}} \varphi^{\otimes N-2j} \otimes_s \chi_{2j}}_{x_{\intinterval{2k+1}{N}}}
    \end{multline*}
    All summands vanish.
    For the term (c), with \( 1 \leq i_1 < i_2 \leq 2k \), we have
    \begin{multline}
        \frac{\sqrt{\binom{N}{2k}}}{N-1} \proj_{\mathcal{N}^{(2k)}_\varphi(u)}\braket{\varphi^{\otimes N-2k}}{v(x_{i_1} - x_{i_2}) u}_{x_{\intinterval{2k+1}{N}}} \\
        = \frac{\sqrt{\binom{N}{2k}}}{N-1} \sum_{j=0}^n c_{2j} \proj_{\mathcal{N}^{(2k)}_\varphi(u)}\braket{\varphi^{\otimes N-2k}}{v(x_{i_1} - x_{i_2}) \varphi^{\otimes N-2j}\otimes_s \chi_{2j}}_{x_{\intinterval{2k+1}{N}}}.
    \end{multline}
    The right-hand side has a vanishing summand for \( j<k-1 \) and \( j>k \).
    With \( j=k-1 \), we get
    \begin{multline*}
        \proj_{\mathcal{N}^{(2k)}_\varphi(u)}\braket{\varphi^{\otimes N-2k}}{v(x_{i_1} - x_{i_2}) \varphi^{\otimes N-2k+2}\otimes_s \chi_{2k-2}}_{x_{\intinterval{2k+1}{N}}} \\
        = \frac{1}{\sqrt{\binom{N}{2k-2}}} \proj_{\mathcal{N}^{(2k)}_\varphi(u)} \left( v(x_{i_1}-x_{i_2}) \varphi(x_{i_1}) \varphi(x_{i_2}) \chi_{2k-2}(x_{\intinterval{1}{2k}\setminus \{i_1, i_2\}}) \right).
    \end{multline*}
    With \( j=k \), we get
    \begin{align*}
        \braket{\varphi^{\otimes N-2k}}{v(x_{i_1} - x_{i_2}) \varphi^{\otimes N-2k}\otimes_s \chi_{2k}}_{x_{\intinterval{2k+1}{N}}}
        &= \frac{1}{\sqrt{\binom{N}{2k}}} \proj_{\mathcal{N}^{(2k)}_\varphi(u)} \left( v(x_{i_1} - x_{i_2}) \chi_{2k}(x_{\intinterval{1}{2k}}) \right).
    \end{align*}
    Thus, 
    \begin{multline*}
        \frac{\sqrt{\binom{N}{2k}}}{N-1} \proj_{\mathcal{N}^{(2k)}_\varphi(u)}\braket{\varphi^{\otimes N-2k}}{v(x_{i_1} - x_{i_2}) u}_{x_{\intinterval{2k+1}{N}}} \\
        = \frac{\sqrt{\binom{N}{2k}}}{(N-1)\sqrt{\binom{N}{2k-2}}} \proj_{\mathcal{N}^{(2k)}_\varphi(u)} \left( v(x_{i_1}-x_{i_2}) \varphi(x_{i_1}) \varphi(x_{i_2}) \chi_{2k-2}(x_{\intinterval{1}{2k}\setminus \{i_1, i_2\}}) \right) \\
        + \frac{\sqrt{\binom{N}{2k}}}{(N-1)\sqrt{\binom{N}{2k}}} \proj_{\mathcal{N}^{(2k)}_\varphi(u)} \left( v(x_{i_1} - x_{i_2}) \chi_{2k}(x_{\intinterval{1}{2k}}) \right),
    \end{multline*}
    and the right-hand side is uniformly bounded with respect to \( N \).
    For the term (d), with \( 1 \leq i \leq 2k \), we have
    \begin{multline*}
        \frac{\sqrt{\binom{N}{2k}}}{N-1} (N-2k) \proj_{\mathcal{N}^{(2k)}_\varphi(u)}\braket{\varphi^{\otimes N-2k}}{v(x_i - x_{2k+1}) u}_{x_{\intinterval{2k+1}{N}}} \\
        = \frac{\sqrt{\binom{N}{2k}}}{N-1} (N-2k) \sum_{j=0}^n c_{2j} \proj_{\mathcal{N}^{(2k)}_\varphi(u)}\braket{\varphi^{\otimes N-2k}}{v(x_i - x_{2k+1}) \varphi^{\otimes N-2j} \otimes_s \chi_{2j}}_{x_{\intinterval{2k+1}{N}}}.
    \end{multline*}
    The right-hand side has a vanishing summand for \( j\neq k \), hence
    \begin{multline*}
        \frac{\sqrt{\binom{N}{2k}}}{N-1} (N-2k) \proj_{\mathcal{N}^{(2k)}_\varphi(u)}\braket{\varphi^{\otimes N-2k}}{v(x_i - x_{2k+1}) u}_{x_{\intinterval{2k+1}{N}}} \\
        = \frac{N-2k}{N-1} c_{2k} \proj_{\mathcal{N}^{(2k)}_\varphi(u)} \left( \begin{gathered}
            \proj_{\mathcal{N}^{(2k)}_\varphi(u)}\braket{\varphi}{v(x_i - x_{2k+1}) \varphi(x_{2k+1}) \chi_{2k}(x_{\intinterval{1}{2k}})}_{x_{2k+1}} \\
            + \proj_{\mathcal{N}^{(2k)}_\varphi(u)}\braket{\varphi}{v(x_i - x_{2k+1}) \varphi(x_i) \chi_{2k}(x_{\intinterval{1}{2k+1}\setminus\{i\}})}_{x_{2k+1}}
        \end{gathered} \right),
    \end{multline*}
    which is uniformly bounded with respect to \( N \).
    For the term (e), we have
    \begin{multline*}
        \frac{\sqrt{\binom{N}{2k}}}{N-1} \frac{(N-2k)(N-2k-1)}{2} \proj_{\mathcal{N}^{(2k)}_\varphi(u)}\braket{\varphi^{\otimes N-2k}}{v(x_{2k+1} - x_{2k+2}) u}_{x_{\intinterval{2k+1}{N}}} \\
        = \frac{\sqrt{\binom{N}{2k}}}{N-1} \frac{(N-2k)(N-2k-1)}{2} \sum_{j=0}^n c_{2j} \proj_{\mathcal{N}^{(2k)}_\varphi(u)}\braket{\varphi^{\otimes N-2k}}{v(x_{2k+1} - x_{2k+2}) \varphi^{\otimes N-2j} \otimes_s \chi_{2j}}_{x_{\intinterval{2k+1}{N}}}.
    \end{multline*}
    The right-hand side has a vanishing summand for \( j\neq k+1 \).
    Hence,
    \begin{multline*}
        \frac{\sqrt{\binom{N}{2k}}}{N-1} \frac{(N-2k)(N-2k-1)}{2} \proj_{\mathcal{N}^{(2k)}_\varphi(u)}\braket{\varphi^{\otimes N-2k}}{v(x_{2k+1} - x_{2k+2}) u}_{x_{\intinterval{2k+1}{N}}} \\
        = \frac{\sqrt{\binom{N}{2k}}}{N-1} \frac{(N-2k)(N-2k-1)}{2} \frac{1}{\sqrt{\binom{N}{2k+2}}} c_{2k+2} \proj_{\mathcal{N}^{(2k)}_\varphi(u)}\braket{\varphi^{\otimes 2}}{v(x_{2k+1} - x_{2k+2}) \chi_{2k+2}}_{x_{\{2k+1, 2k+2\}}},
    \end{multline*}
    which is uniformly bounded with respect to \( N \).
\end{proof}

We can now proceed with the equation of motion of the coefficients \( c_{2k} \).
Define the \( (n+1)\times (n+1) \) time-dependent Hermitian matrix
\begin{equation*}
    \mathbf{H}(t) := \left( \braket{\varphi(t)^{\otimes N-2j}\otimes_s \chi_{2j}(t)}{H\varphi(t)^{\otimes N-2k} \otimes_s \chi_{2k}(t)} \right)_{j,k=0}^n,
\end{equation*}
then
\begin{equation}
    \label{eqn: setting n>1 - equations of motion coeffs}
    \ii \dot{\mathbf{c}}(t) = \check{\mathbf{H}}(t) \mathbf{c}(t), \quad \text{ with }
    \check{\mathbf{H}}(t) = \mathbf{H}(t) - \ii \braket{\varphi(t)}{\dot{\varphi}(t)} \diag(N, N-2, \dots, N-2n).
\end{equation}
We have the following result:
\begin{theorem}
    \label{theorem: unitary evolution coeffs}
    Assume \( \varphi \) is solution to the Hartree equation \eqref{eqn: hartree equation}.
    Then, for all times \( t \), the entries of the Hermitian matrix \( \check{\mathbf{H}}(t) \) are slowly varying with respect to time and uniformly bounded with respect to \( N \).
\end{theorem}

\begin{proof}
    The matrix \( \check{\mathbf{H}} \) is Hermitian since \( \mathbf{H} \) is Hermitian and \( \braket{\varphi}{\ii\dot{\varphi}} \in \mathbb{R} \).
    For \( 0 \leq j \leq n \), explicit computations give (details in Appendix \ref{appendix: components bfH}):
    \begin{equation*}
        \mathbf{H}_{j,j} 
        = \begin{aligned}[t]
            &\frac{(N-2j)(N-2j-1)}{2(N-1)} \braket{\varphi^{\otimes 2}}{v(x_1-x_2) \varphi^{\otimes 2}} + \frac{2j(2j-1)}{2(N-1)} \braket{\chi_{2j}}{v(x_1-x_2) \chi_{2j}} \\
            &+\frac{2j(N-2j)}{N-1} \Big( \braket{\chi_{2j} \otimes \varphi}{v(x_1-x_{2j+1}) \varphi \otimes \chi_{2j}} + \braket{\chi_{2j} \otimes \varphi}{v(x_1 - x_{2j+1}) \chi_{2j} \otimes \varphi} \Big) \\
            &+ (N-2j) \braket{\varphi}{-\Delta \varphi} + 2j \braket{\chi_{2j}}{-\Delta_{x_1}\chi_{2j}}.
        \end{aligned}
    \end{equation*}
    All the brakets in \( \mathbf{H}_{j,j} \) are uniformly bounded with respect to \( N \), thus we can write \( \mathbf{H}_{j,j} = \beta_1 N + \beta_2 + \beta_3 / N \) for some coefficients \( \beta_1, \beta_2, \beta_3 \in \mathbb{R} \) that are independent of \( N \).
    More precisely, we have
    \begin{equation*}
        \mathbf{H}_{j,j} = N \braket{\varphi}{-\Delta\varphi} + \frac{N}{2} \braket{\varphi^{\otimes 2}}{v(x_1-x_2) \varphi^{\otimes 2}} + \mathcal{O}(1), \quad 0 \leq j \leq n,
    \end{equation*}
    where \( \mathcal{O}(1) \) denotes a quantity that is uniformly bounded with respect to \( N \).
    The coefficient \( \frac{1}{2} \braket{\varphi^{\otimes 2}}{v(x_1-x_2) \varphi^{\otimes 2}} \) is exactly the coefficient \( \mu_{\varphi(t)} \) defined in \eqref{eqn: mu in hartree equation}.
    Since \( \varphi \) is solution to the Hartree equation \eqref{eqn: hartree equation}, we can write
    \begin{equation*}
        \mathbf{H}_{j,j} = N \braket{\varphi}{\ii\dot{\varphi}} + \mathcal{O}(1).
    \end{equation*}
    Therefore, the diagonal elements of the matrix \( \check{\mathbf{H}} \) are all uniformly bounded with respect to \( N \).
    For the off-diagonal terms, we have (see again Appendix \ref{appendix: components bfH}):
    \begin{equation*}
        \conj{\mathbf{H}_{j+1, j}} = \mathbf{H}_{j,j+1}
        = \frac{\sqrt{(N-2j)(N-2j-1)}}{2(N-1)} \sqrt{(2j+2)(2j+1)} \braket{\varphi^{\otimes 2} \otimes \chi_{2j}}{v(x_1-x_2) \chi_{2j+2}},\quad 0 \leq j \leq n-1,
    \end{equation*}
    which are also uniformly bounded with respect to \( N \).
    The smoothness of the matrix \( \check{\mathbf{H}} \) with respect to time is simply due to the fact that it depends only on \( \varphi \) and on the excitation functions \( \chi_{2j} \), which are smooth with respect to time by Theorem \ref{theorem: setting n>1 no tensorization - chi2dot uniformly bounded wrt N}.
\end{proof}

The value of the coefficient \( \mu_{\varphi(t)} \) given by \eqref{eqn: mu in hartree equation} can now be explained: if any other value for \( \mu \) were to be chosen in the Hartree equation \eqref{eqn: hartree equation}, the diagonal elements of the matrix \( \check{\mathbf{H}} \) would be of the form \( \check{\mathbf{H}}_{j,j} = \frac{N}{2} \braket{\varphi^{\otimes 2}}{v(x_1-x_2)\varphi^{\otimes 2}} - N \mu + \mathcal{O}(1) \).
This would mean that the coefficient vector \( c \) is highly oscillatory, with a phase proportional to \( N \).
By choosing \( \mu_{\varphi(t)} \) as given by \eqref{eqn: mu in hartree equation}, we eliminate the highly oscillatory behaviour.

Theorems \ref{theorem: setting n>1 no tensorization - chi2dot uniformly bounded wrt N} and \ref{theorem: unitary evolution coeffs} show that the equations of motion \eqref{eqn: equations of motion} have bounded right-hand sides as \( N \) increases.
However, the boundedness of the terms in the equations of motion with respect to \( N \) is not the only property of interest: we also want to be able to actually compute the right-hand side of the equations of motion \eqref{eqn: equations of motion}.
This is the content of the next section, where we show that it can be done with a computational complexity that is uniformly bounded with respect to \( N \).

\section{Computational complexity of the terms in the equations of motion}
\label{section: computational complexity equations of motion}

We will prove the following theorem, loosely stated on purpose:
\begin{theorem}
    \label{theorem: goal O(1) complexity}
    The computational cost of evaluating the terms on the right-hand side of the equations of motion \eqref{eqn: equations of motion} is uniformly bounded with respect to the number \( N \) of particles.
\end{theorem}

Part of the proof has already been shown in Lemmas \ref{lemma: reducing number of effective kinetic operators} and \ref{lemma: reducing number of effective potential operators}, where we used symmetries to reduce the number of operators that we have to deal with.
The second part of the proof of Theorem \ref{theorem: goal O(1) complexity} consists of reducing the number of permutations involved in the symmetric tensor products as well as reducing the dimensionality of all brakets.
This is detailed in Section \ref{section: reducing the number of permutations and variables in braket}.

\subsection{Reduction of terms by a pigeonhole argument}
We start with a result that allows us to skip the computation of some brakets using a simple argument.
\begin{lemma}
    \label{lemma: when does braket varphiN-m-p otimes chip vs A varphiN-j otimes chiq vanishes depending on p q s m}
    Let \( m \in \intinterval{0}{N} \), \( p, q \in \intinterval{0}{n} \), \( \tau\in S_N \), and \( A_{\ell_1, \ell_2} \) a two-body operator acting on variables \( \ell_1, \ell_2 \).
    Let \( \varphi \in L^2(\Omega) \), \( \chi_{2p} \in L^2_s(\Omega^{2p}) \) and \( \chi_{2q} \in L^2_s(\Omega^{2q}) \), such that \( \norm{\varphi}{2} = 1 \) and the orthogonality \eqref{eqn: bogoliubov orthogonality relation} is satisfied.
    Then the braket
    \begin{equation*}
        \braket{\varphi^{\otimes N-m-2p} \otimes \chi_{2p}}{A_{\ell_1, \ell_2} \varphi^{\otimes N-2q}(x_{\tau(\intinterval{2q+1}{N})})\chi_{2q}(x_{\tau(\intinterval{1}{2q})})}_{x_{\intinterval{m+1}{N}}}
    \end{equation*}
    vanishes when
    \begin{equation}
        2p > 2q+2 \quad \text{or}\quad 2q > 2p+m+2.
    \end{equation}
\end{lemma}

These two conditions are only sufficient conditions, they are not necessary.

\begin{proof}
    We start with a simple observation: the braket over \( N-m \) variables vanishes as soon as the braket over one of the variables is of the form \( \braket{\varphi(y)}{\chi_{2p}(y, \cdot)} \) or \( \braket{\varphi(y)}{\chi_{2q}(y, \cdot)} \), since those one-dimensional brakets vanish by \eqref{eqn: bogoliubov orthogonality relation}.
    Such brakets necessarily occur in the following situations:
    \begin{enumerate}
        \item one of the functions \( \varphi \) on the left is paired with a marginal of \( \chi_{2q} \) on the right.
        It means we couldn't pair all functions \( \varphi \) from the left argument of the braket to those on the right argument or to marginals of \( \chi_{2q} \) upon which \( A \) is acting.
        This situation occurs when \( N-m-2p > N-2q+2 \iff 2q > 2p+m+2 \).

        \item one of the marginals of \( \chi_{2p} \) from the left argument of the braket is paired with a function \( \varphi \) from the right argument.
        It means we couldn't pair all marginals of \( \chi_{2p} \) on the left to marginals of \( \chi_{2q} \) on the right or to functions \( \varphi \) on which the operator acts.
        This situation occurs when \( 2p > 2q+2 \).
        
        \item a function \( \varphi \) from the right argument on the braket is paired with a marginal of \( \chi_{2p} \) from the left argument.
        It means we couldn't pair all functions \( \varphi \) from the right to those from the left argument or to marginals of \( \chi_{2p} \) where \( A \) is acting, or to put them outside the braket.
        This situation occurs when \( N-2q > N-m-2p + m + 2 \iff 2p > 2q+2 \).

        \item one of the marginals of \( \chi_{2q} \) from the right argument of the braket is paired with a function \( \varphi \) from the left argument. It means we couldn't pair all marginals of \( \chi_{2q} \) on the right to marginals of \( \chi_{2p} \) on the left or to functions \( \varphi \) on which the operator acts, or to put them outside the braket.
        This situation occurs when \( 2q > 2p+m+2 \).
    \end{enumerate}
    When any of these conditions is met, the braket vanishes by \eqref{eqn: bogoliubov orthogonality relation}.
\end{proof}

\subsection{Reducing the number of operators}
\label{section: reducing number of operators}

The content of this section has already been stated in Lemmas \ref{lemma: reducing number of effective kinetic operators} and \ref{lemma: reducing number of effective potential operators}.
We write the very short current subsection to emphasize that those lemmas are not just a way to make the proof of Theorem \ref{theorem: setting n>1 no tensorization - chi2dot uniformly bounded wrt N} shorter, but are essential when trying to compute the right-hand side of the equations of motion with a computational complexity that is uniformly bounded with respect to \( N \).
For example, Lemma \ref{lemma: reducing number of effective kinetic operators} states that the \( N \) one-body operators involved in the kinetic operator can be reduced to:
\begin{itemize}
    \item only 2 one-body operators with appropriate scalings when computing the equation of motion for the vector of coefficients \( \mathbf{c} \).
    \item only 3 one-body operators with appropriate scalings when computing the equation of motion for the excitation function \( \chi_2 \) (because in this situation \( m=2 \) and \( p=0 \) in Lemma \ref{lemma: reducing number of effective kinetic operators}, and the braket multiplied by \( 2p \) can be omitted).
\end{itemize}
Similarly, Lemma \ref{lemma: reducing number of effective potential operators} states that the \( N(N-1)/2 \) two-body operators involved in the potential operator can be reduced to:
\begin{itemize}
    \item only 3 two-body operators with appropriate scalings when computing the equation of motion for the vector of coefficients \( \mathbf{c} \).
    \item only 4 two-body operators with appropriate scalings when computing the equation of motion for the excitation function \( \chi_2 \) (because in this situation \( m=2 \) and \( p=0 \) in Lemma \ref{lemma: reducing number of effective potential operators}, and there are \( m+2 \) operators multiplied by \( 2p \) which can then be omitted).
\end{itemize}

\subsection{Reducing the number of permutations and variables in the braket}
\label{section: reducing the number of permutations and variables in braket}

The objective of this section is two-fold: first we reduce the number of permutations that are involved in the symmetric tensor products in the definition of \( u \), and second we also reduce the number of variables along which to compute the braket.
The first part is due to many permutations leading to a multi-particle braket involving a one-particle braket of the form \( \braket{\varphi(y)}{\chi_{2j}(y, \cdot)}_y = 0 \).
The second part is due to many variables in the multi-particle braket along which the one-particle braket is of the form \( \braket{\varphi}{\varphi}=1 \).
Because those equalities are known beforehand, we can skip all permutations that yield the former and skip the one-particle brakets along variables leading to the latter.

Due to Lemmas \ref{lemma: reducing number of effective kinetic operators} and \ref{lemma: reducing number of effective potential operators} and the equations of motion \eqref{eqn: equations of motion}, we only need to deal with brakets of the form
\begin{equation}
    \label{eqn: the two types of brakets we have to deal with in reducing nb of permutations and variables}
    \braket{\varphi^{\otimes N-m-2p} \otimes \chi_{2p}}{-\Delta_{x_k} u}_{x_{\intinterval{m+1}{N}}}
    \quad\text{or}\quad
    \braket{\varphi^{\otimes N-m-2p} \otimes \chi_{2p}}{V_{\ell_1, \ell_2} u}_{x_{\intinterval{m+1}{N}}},
\end{equation}
where \( m\in \{0,2\} \), \( p\in \intinterval{0}{n} \), and where the integers \( 1 \leq k,\ell_1,\ell_2 \leq N \) are those involved in Lemmas \ref{lemma: reducing number of effective kinetic operators} and \ref{lemma: reducing number of effective potential operators}.
The function \( \varphi \in L^2(\Omega) \) is the Hartree condensate function, and \( \chi_{2p} \in L^2_s(\Omega^{2p}) \) and \( \chi_{2q} \in L^2_s(\Omega^{2q}) \) are excitation functions: this means in particular that they are normalized and satisfy the orthogonality \eqref{eqn: bogoliubov orthogonality relation}.
The tensorization \eqref{eqn: Bogoliubov tensorization}-\eqref{eqn: Bogoliubov tensorization normalized} is not required in this section.

We show now how to reduce the number of permutations and variables in the braket for the case of a general two-body operator \( V_{\ell_1, \ell_2} \) that acts on variables \( \ell_1 \) and \( \ell_2 \).
Its exact definition will not be used, but we have in mind the multiplicative operator given by \( (V_{\ell_1, \ell_2} f)(x) = v(x_{\ell_1} - x_{\ell_2}) f(x) \) for \( f\in L^2(\Omega^N) \).
By allowing a general two-body operator \( V_{\ell_1, \ell_2} \) that acts on variables \( \ell_1, \ell_2 \), we are also able to deal with a one-body operator by applying our results to the case \( \ell_1 = \ell_2 \).

We start with a preliminary reduction of the number of permutations:
\begin{lemma}
    \label{lemma: symmetric tensor product rewritten with only unique permutations}
    Let \( i,j \geq 1 \), and \( f \in L^2_s(\Omega^i) \), \( g\in L^2_s(\Omega^j) \).
    Then
    \begin{equation}
        \label{eqn: symmetric tensor product rewritten with only unique permutations}
        (f\otimes_s g)(x) = \binom{i+j}{j}^{-1/2} \sum_{\substack{\tau \in S^{j}_{i+j}\\\tau\,\mathrm{sorted}}} f(x \setminus x_\tau) g(x_{\tau}),
    \end{equation}
    where ``\( \tau \in S^j_{i+j} \) is sorted'' means that \( \tau \) is sorted according to a given total ordering.
\end{lemma}

We recall that \( S^j_{i+j} \) is the set of permutations of \( j \) elements chosen among the set \( \intinterval{1}{i+j} \), see Section \ref{section: notations}.

\begin{proof}
    We have
    \begin{equation*}
        (f\otimes_s g)(x)
        = \frac{1}{\mathsf{K}_{i}^{i+j}} \sum_{\sigma\in S_{i+j}} f(x_{\sigma(\intinterval{1}{i})}) g(x_{\sigma(\intinterval{i+1}{i+j})}) 
        = \frac{1}{\mathsf{K}_{i}^{i+j}} \sum_{\sigma_1 \in S^{j}_{i+j}} \sum_{\sigma_2 \in S_{i+j\setminus \sigma_1}} f(x_{\sigma_2}) g(x_{\sigma_1}),
    \end{equation*}
    where we have decomposed \( \sigma\in S_{i+j} \) as \( \sigma = (\sigma_2, \sigma_1) \), with \( \sigma_1 \in S_{i+j}^j \) and \( \sigma_2 \in S_{i+j\setminus \sigma_1} \).
    By symmetry of the function \( f \) we have \( f(x_{\sigma_2}) = f(x\setminus x_{\sigma_1}) \), hence
    \begin{equation*}
        (f\otimes_s g)(x)
        = \frac{1}{\mathsf{K}_{i}^{i+j}} \sum_{\sigma_1 \in S^{j}_{i+j}} \sum_{\sigma_2 \in S_{i+j\setminus \sigma_1}} f(x \setminus x_{\sigma_1}) g(x_{\sigma_1})
        = \frac{i!}{\mathsf{K}_{i}^{i+j}} \sum_{\sigma_1 \in S^{j}_{i+j}} f(x \setminus x_{\sigma_1}) g(x_{\sigma_1}).
    \end{equation*}
    By symmetry of the function \( g \) we have \( g(x_{\sigma_1}) = g(x_{\tau}) \) for all \( \tau \in S_{\sigma_1} \).
    Moreover, for each \( \sigma_1 \in S^j_{i+j} \), there is a unique permutation \( \tau \) such that \( \tau \) is sorted according to the given total ordering and \( \tau(\intinterval{1}{j}) =_s \sigma(\intinterval{1}{j}) \).
    Iterating over all \( \sigma_1 \in S^j_{i+j} \) is equivalent to iterating over all sorted permutations \( \tau \), and for each \( \tau \), over all its permutations.
    Thus,
    \begin{equation*}
        (f\otimes_s g)(x)
        = \frac{i!}{\mathsf{K}_{i}^{i+j}} \sum_{\substack{\tau\in S^j_{i+j}\\\tau\,\mathrm{sorted}}} \sum_{\mu \in S_\tau} f(x \setminus x_{\tau}) g(x_{\tau})
        = \frac{i!j!}{\mathsf{K}_{i}^{i+j}} \sum_{\substack{\tau\in S^j_{i+j}\\\tau\,\mathrm{sorted}}} f(x \setminus x_{\tau}) g(x_{\tau}),
    \end{equation*}
    and we conclude using the definition of the coefficient \( \mathsf{K}^{i+j}_i \).
\end{proof}

Let us come back to the braket involving the two-body operator in \eqref{eqn: the two types of brakets we have to deal with in reducing nb of permutations and variables}.
Using Lemma \ref{lemma: symmetric tensor product rewritten with only unique permutations}, the symmetric tensor product in the summand of
\begin{equation*}
    \braket{\varphi^{\otimes N-m-2p} \otimes \chi_{2p}}{V_{\ell_1, \ell_2} u}_{x_{\intinterval{m+1}{N}}} 
    = \sum_{q=0}^n c_{2q}\braket{\varphi^{\otimes N-m-2p} \otimes \chi_{2p}}{V_{\ell_1, \ell_2} \varphi^{\otimes N-2q} \otimes_s \chi_{2q}}_{x_{\intinterval{m+1}{N}}}
\end{equation*}
can be replaced by a sum over all sorted permutations from the set \( S_N^{2q} \):
\begin{equation}
    \braket{\varphi^{\otimes N-m-2p} \otimes \chi_{2p}}{V_{\ell_1, \ell_2} u}_{x_{\intinterval{m+1}{N}}}
    = \sum_{q=0}^n \frac{c_{2q}}{\sqrt{\binom{N}{2q}}} \sum_{\substack{\tau \in S_N^{2q}\\\tau\,\mathrm{sorted}}} \braket{\varphi^{\otimes N-m-2p} \otimes \chi_{2p}}{V_{\ell_1, \ell_2} \varphi^{\otimes N-2q}(x \setminus x_\tau) \chi_{2q}(x_{\tau})}_{x_{\intinterval{m+1}{N}}}.
    \label{eqn: reducing nb operators - intermediate simplification}
\end{equation}
This is not good enough, since the number of permutations still depends on \( N \).
Using the orthonormality of the family \( \{\varphi, \eta_1, \dots, \eta_r\} \) it is possible to simplify further that sum over the permutations in order to get rid of the dependency with respect to \( N \).
Let us first define the following sets:
\begin{gather}
    \mathcal{T} = \intinterval{m+1}{N-2p}\setminus \{\ell_1, \ell_2\}, \qquad
    \mathcal{T}^C = \intinterval{1}{N} \setminus \mathcal{T}, \label{eqn: def calT}\\
    \mathcal{B} = (\{\ell_1, \ell_2\} \cup \intinterval{N-2p+1}{N}) \setminus \intinterval{1}{m}. \label{eqn: def calB}
\end{gather}
where the indices \( \ell_1, \ell_2 \) are those involved in \eqref{eqn: the two types of brakets we have to deal with in reducing nb of permutations and variables}.
For the sake of readability we omit in \( \mathcal{T} \) and \( \mathcal{B} \) their explicit dependence on \( m \), \( p \), \( \ell_1 \) and \( \ell_2 \).
The set \( \mathcal{T} \) is the set of variables that correspond to a function \( \varphi \) on the left argument of the braket \eqref{eqn: the two types of brakets we have to deal with in reducing nb of permutations and variables}, omitting those acted on by the operator \( V_{\ell_1, \ell_2} \).
The set \( \mathcal{T}^C \) is the complement of \( \mathcal{T} \) in \( \intinterval{1}{N} \).
The set \( \mathcal{B} \) is the set of marginals from the left argument of the braket that are not orthogonal to marginals of \( \chi_{2q} \).
As we will see, \( \mathcal{T} \) is the set of variables that can be ``skipped'' during the computation of the braket, and \( \mathcal{B} \) is the set of variables that actually determine the value of the braket.

We are now able to express in a concise manner the reduction of the number of permutations that need to be considered, as well as the reduction of the number of variables along which the braket actually needs to be computed.

\begin{definition}
    \label{def: valid permutation tau}
    For given \( N, p, q, \ell_1, \ell_2 \), a permutation \( \tau \in S^{2q}_{\mathcal{T}^C} \) is said to be \emph{valid} if \( \intinterval{N-2p+1}{N} \subset \{\ell_1, \ell_2\}\cup \tau \).
    In other words, a permutation \( \tau \in S^{2q}_{\mathcal{T}^C} \) is valid if the braket \( \braket{\varphi^{\otimes N-m-2p} \otimes \chi_{2p}}{V_{\ell_1, \ell_2} \varphi^{\otimes N-2q}(x \setminus x_\tau) \chi_{2q}(x_{\tau})}_{x_{\intinterval{m+1}{N}}} \) doesn't involve any one-dimensional braket of the form \( \braket{\chi_{2p}(y, \cdot)}{\varphi(y)}_y \).
\end{definition}

\begin{lemma}
    \label{lemma: sum over tau in Sjn with no braket along phi}
    Let integers \( m \in \intinterval{0}{N} \), \( p, q \in \intinterval{0}{n} \).
    Let functions \( \varphi \in L^2(\Omega) \), \( \chi_{2p} \in L^2_s(\Omega^{2p}) \) and \( \chi_{2q} \in L^2_s(\Omega^{2q}) \), such that \( \norm{\varphi}{2} = 1 \) and the orthogonality \eqref{eqn: bogoliubov orthogonality relation} is satisfied.
    For any \( \ell_1, \ell_2 \in \intinterval{1}{N} \),
    \begin{multline}
        \label{eqn: sum over tau in Sjn with no braket along phi}
        \sum_{\substack{\tau\in S^{2q}_N\\\tau\,\mathrm{sorted}}} \braket{\varphi^{\otimes N-m-2p} \otimes \chi_{2p}}{V_{\ell_1, \ell_2}\varphi^{\otimes N-2q}(x\setminus x_\tau) \chi_{2q}(x_\tau)}_{x_{\intinterval{m+1}{N}}} \\
        = \sum_{\substack{\tau\in S^{2q}_{\mathcal{T}^C}\\\tau\,\mathrm{sorted,\,valid}}} \braket{\varphi^{\otimes N-m-2p} \otimes \chi_{2p} }{V_{\ell_1, \ell_2} \varphi^{\otimes N-2q}(x\setminus x_\tau) \chi_{2q}(x_\tau)}_{\substack{x_{\mathcal{B}} \\ \mathbf{1}_{\mathcal{T}}}}.
    \end{multline}
\end{lemma}

\begin{proof}
    The proof is divided into two steps: the first one consists in showing that the set \( \{\tau \in S^{2q}_N : \tau\,\mathrm{sorted} \} \) over which the sum on the left-hand side is performed can be replaced by \( \{\tau \in S^{2q}_{\mathcal{T}^C} : \tau\,\mathrm{sorted,\,valid}\}\).
    The second step consists in showing that the brakets along variables indexed by \( \mathcal{T} = \intinterval{m+1}{N-2p} \setminus \{\ell_1, \ell_2\} \) are all equal to one.

    The first step relies on the fact that there are many permutations \( \tau \) such that the summand is necessarily zero due to the orthogonality \( \braket{\varphi(y)}{\chi_{2q}(y, \cdot)}_y = 0 \), and we can thus get rid of them.
    We are computing
    \begin{equation}
        \label{eqn: proof simplification sum over tau sorted - before doing anything}
        \sum_{\substack{\tau\in S^{2q}_N\\\tau\,\mathrm{sorted}}} \braket{\varphi^{\otimes N-m-2p} \otimes \chi_{2p}}{V_{\ell_1, \ell_2}\varphi^{\otimes N-2q}(x\setminus x_\tau) \chi_{2q}(x_\tau)}_{x_{\intinterval{m+1}{N}}},
    \end{equation}
    and the set \( \mathcal{T} \) is defined as the set of variables such that the left argument of the braket along each of these variables is \( \varphi \) and such that the operator \( V_{\ell_1, \ell_2} \) is not acting on them.
    If the permutation \( \tau \in S^{2q}_N \) is such that there is an index \( k \in \intinterval{1}{2q} \) for which \( \tau(k) \in \mathcal{T} \), then the braket along variable \( x_{\tau(k)} \) is \( \braket{\varphi(y)}{\chi_{2q}(y, \cdot}_y = 0 \) and the summand of \eqref{eqn: proof simplification sum over tau sorted - before doing anything} for that permutation vanishes.
    Therefore, the permutations \( \tau \in S^{2q}_N \) that don't necessarily give a vanishing summand are those for which \( \tau(\intinterval{1}{2q}) \subset \mathcal{T}^C := \intinterval{1}{m} \cup \{\ell_1, \ell_2\} \cup \intinterval{N-2p+1}{N} \), i.e. all \( \tau \in S^{2q}_{\mathcal{T}^C} \).
    We can argue similarly for brakets of the form \( \braket{\chi_{2p}(y, \cdot)}{\varphi(y)}_y = 0 \): they occur if a variable with index in \( \intinterval{N-2p+1}{N} \setminus \{\ell_1, \ell_2 \} \) corresponds to a function \( \varphi \) in the right argument of the braket, and this happens if \( \intinterval{N-2p+1}{N} \not\subset \tau \cup \{\ell_1, \ell_2\} \).
    Finally, the permutations we need to consider are the \emph{valid} (according to Definition \ref{def: valid permutation tau}) sorted permutations \( \tau\in S^{2q}_{\mathcal{T}^C} \).

    The second step of the proof is much shorter: for any such valid and sorted permutation \( \tau \), all the variables indexed by \( \mathcal{T} \) in \eqref{eqn: proof simplification sum over tau sorted - before doing anything} correspond to a braket \( \braket{\varphi}{\varphi} = 1 \).
    In other words, the multi-dimensional braket can be taken only on the variables \( \intinterval{m+1}{N} \setminus \mathcal{T} = (\{\ell_1, \ell_2\} \cup \intinterval{N-2p+1}{N}) \setminus \intinterval{1}{m} \), which is exactly the set \( \mathcal{B} \) defined in \eqref{eqn: def calB}.
\end{proof}

Finally, applying Lemma \ref{lemma: sum over tau in Sjn with no braket along phi} to \eqref{eqn: reducing nb operators - intermediate simplification} yields:
\begin{equation}
        \braket{\varphi^{\otimes N-m-2p} \otimes \chi_{2p}}{V_{\ell_1, \ell_2} u}_{x_{\intinterval{m+1}{N}}} 
        = \sum_{q=0}^n \frac{c_{2q}}{\sqrt{\binom{N}{2q}}} \sum_{\substack{\tau \in S_{\mathcal{T}^C}^{2q}\\\tau\,\mathrm{sorted,\, valid}}} \braket{\varphi^{\otimes N-m-2p} \otimes \chi_{2p} }{V_{\ell_1, \ell_2} \varphi^{\otimes N-2q}(x\setminus x_\tau) \chi_{2q}(x_\tau)}_{\substack{x_{\mathcal{B}}\\\mathbf{1}_{\mathcal{T}}}}.
    \label{eqn: braket after reducing nb of permutations and variables - potential}
\end{equation}
The number of permutations required to compute the right-hand side of \eqref{eqn: braket after reducing nb of permutations and variables - potential} is at most
\begin{equation*}
    \left| \left\{\tau \in S_{\mathcal{T}^C}^{2q} : \tau\,\mathrm{sorted,\,valid}\right\} \right|
    \leq \left| \left\{\tau \in S_{\mathcal{T}^C}^{2q} : \tau\,\mathrm{sorted}\right\} \right|
    = \frac{1}{(2q)!} \frac{\abs{\mathcal{T}^C}!}{(\abs{\mathcal{T}^C}-2q)!} 
    = \frac{1}{(2q)!} \frac{(m+2+2p)!}{(m+2+2p-2q)!}
\end{equation*}
permutations, which is a number that is indeed independent of \( N \) and more importantly manageable in practice.
Note that the denominator is always the factorial of a nonnegative number by Lemma \ref{lemma: when does braket varphiN-m-p otimes chip vs A varphiN-j otimes chiq vanishes depending on p q s m}, since the braket is known to vanish for \( m+2+2p-2q < 0 \) and this situation can thus be skipped.
The computation of the summand of \eqref{eqn: braket after reducing nb of permutations and variables - potential} only needs to be performed over \( \abs{\mathcal{B}} \) variables, i.e. at-most \( 2p+2 \) variables instead of \( N-m \).
We have thus shown Theorem \ref{theorem: goal O(1) complexity}, and at the same time paved the way for a numerically-viable method in the limit \( N\to \infty \).

\section{Numerical method}
\label{section: numerical method}

Let us quickly summarize what we have done so far.
We have shown in Section \ref{section: equations of motion} that it is possible to update the approximation \( u \) given by \eqref{eqn: final ansatz u} via the equations of motion \eqref{eqn: equations of motion} for the coefficients \( c_{2k} \) and the excitation functions \( \chi_{2k} \), and where the condensate function \( \varphi \) is solution to the Hartree equation \eqref{eqn: hartree equation}.
The derivation of the equations of motion \eqref{eqn: equations of motion} assumed that the excitation functions were independent of each other: otherwise, when the tensorization \eqref{eqn: Bogoliubov tensorization}-\eqref{eqn: Bogoliubov tensorization normalized} is assumed according to the Bogoliubov theory,  this causes the difficulties outlined in Section \ref{section: DFVP and tensorization}.
We have also shown in Section \ref{section: equations of motion - study of RHS} that the equations of motion \eqref{eqn: equations of motion} are well-behaved with right-hand sides that are uniformly bounded for all \( N \).
In Section \ref{section: computational complexity equations of motion} we have shown by combinatorial arguments that the computation of the right-hand side of the equations of motion can be performed with a computational effort independent of \( N \).
Hence, we have eliminated the computational bottleneck associated with \( N\to\infty \), and we can now devise a numerical method for the time-evolution of the approximate solution \( u \) to the Schrödinger equation.
To avoid the time integration of the high-dimensional equations of motion for \( \chi_{2k} \), we propose in Section \ref{section: integration eqns of motion n>1} a numerical scheme based on the time-discretization presented in Section \ref{section: time and space discretizations} and on the tensorization \eqref{eqn: Bogoliubov tensorization}-\eqref{eqn: Bogoliubov tensorization normalized}.

In order to obtain a numerical method that can be used in practice, we have to approximate the continuous problem in several ways but most of the approximations we perform are not new.
The novelty of this section lies mostly in how those approximations are assembled into a numerical method rather than the approximations themselves -- with the exception of Section \ref{section: integration eqns of motion n>1} related to the case \( n>1 \).

\subsection{Time and space discretizations}
\label{section: time and space discretizations}

To compute an approximate solution \( u \) of the form \eqref{eqn: final ansatz u}
to the Schrödinger equation \eqref{eqn: general schrodinger equation}, we need discretizations in time and space.

For the time discretization, we consider a final time \( t_{\text{final}} > 0 \) and a time step \( \Delta t > 0 \).
We want to obtain approximations of \( u \) at the times \( t_k = k \Delta t \) for \( 0 \leq k \leq \lfloor t_{\text{final}} / \Delta t \rfloor \).
Given that \( u \) will only be known at those times \( t_k \), a numerical time integrator will be needed.
Details are given in Section \ref{section: numerical integration of equations of motion}.

For the space discretization, we consider the simple case of a box \( [x_{\min}, x_{\max}]^d \) for some \( x_{\min} < x_{\max} \).
The box need not be a hypercube, but we make this assumption in order to make the presentation simpler.
Also for simplicity, we assume the number of discretization points to be the same for each dimension, and it will be denoted \( M \).
The discretization points are chosen according to the trapezoidal rule over \( [x_{\min}, x_{\max}]^d \): this means that the grid nodes along each dimension are given by \( x_j = x_{\min} + j\Delta x \), where \( j\in [0, M-1] \) and \( \Delta x = (x_{\max}-x_{\min})/M \).
We denote by \( \Omega_{discr} \) the computational box:
\begin{equation*}
    \Omega_{discr} = (x_{\min}, \dots, x_{\min}) + \intinterval{0}{M-1}^d \Delta x.
\end{equation*}
Furthermore, we apply periodic boundary condition to the box \( \Omega_{discr} \), which needs to be sufficiently large to contain the domain where the solution is not negligibly small.

\subsection{Numerical solution with \texorpdfstring{\( n>1 \)}{n>1}}
\label{section: integration eqns of motion n>1}

We have shown in Section \ref{section: DFVP and tensorization} that it is non-trivial to obtain equations of motion that incorporate the tensorization \eqref{eqn: Bogoliubov tensorization}-\eqref{eqn: Bogoliubov tensorization normalized} when \( n>1 \), because of a linear operator that needs to be inverted but which is difficult to study (and even possibly non-invertible or ill-conditioned in practice).
In this section we propose a 
time integration scheme that
approximates the functions \( \chi_4, \dots, \chi_{2n} \) 
using the tensorization \eqref{eqn: Bogoliubov tensorization}-\eqref{eqn: Bogoliubov tensorization normalized} once the function \( \chi_2 \) has been integrated in time.
The method we propose keeps \( \chi_4, \dots, \chi_{2n} \)  constant with respect to time during the time integration of \( \chi_2 \) over one time step, and then updates \( \chi_4, \dots, \chi_{2n} \) at the end of the time step using the tensorization property \eqref{eqn: Bogoliubov tensorization}-\eqref{eqn: Bogoliubov tensorization normalized}.  This approach gets rid of the difficulties outlined in Section \ref{section: DFVP and tensorization}, while retaining the Bogoliubov ansatz \eqref{eqn: final ansatz u} with the tensorization \eqref{eqn: Bogoliubov tensorization}-\eqref{eqn: Bogoliubov tensorization normalized} after every time step.

The time integration of the excitation functions over a time interval \( [0, \Delta t] \) would involve the following steps:
\begin{enumerate}
    \item Keep the excitation functions \( \chi_4, \dots, \chi_{2n} \)  constant with respect to time over the time interval \( [0, \Delta t] \).
    \item Integrate the equations of motion for the parameters \( (c_0, c_2, \dots, c_{2n}, \chi_2) \) over the time interval \( [0, \Delta t] \).
    \item At time \( t=\Delta t \), use the tensorization \eqref{eqn: Bogoliubov tensorization}-\eqref{eqn: Bogoliubov tensorization normalized} to reconstruct the excitation functions \( \chi_{2k}(\Delta t) \) from \( \chi_2(\Delta t) \), for \( 2\leq k \leq n \).
\end{enumerate}

\subsection{Operator splitting}
\label{section: operator splitting}

After having performed a time and space discretization, it seems natural to integrate the equations of motion \eqref{eqn: equations of motion} and \eqref{eqn: equations of motion angles} using a standard time-integrator, e.g. of Runge-Kutta type.
However, this naive approach to the time integration of the equations of motion is known to impose a stepsize restriction, due to the oscillations created by the unbounded operator \( \ii \Delta \).

On the other hand, the solution to the equation \( \ii \dot{u} = -\Delta u \) is very easy to obtain: the integration is reduced to multiplication with  scalar exponentials in the Fourier domain.
We want to be able to use that fact in order to avoid the stepsize restriction, and this is exactly what an operator splitting allows us to do: instead of solving \( \ii \dot{u} = (-\Delta + V) u \) at once over a time interval, we consider a series of substeps involving the solutions of \( \ii \dot{u} = -\Delta u \) and \( \ii \dot{u} = Vu \).
The details of the substeps depend on the splitting considered, and it can be high-order in time by composing the two basic steps.
We refer to \cite{lubichQuantumClassicalMolecular2008} for more detail about splitting methods for the Schrödinger equation, including error bounds and higher-order composition.
We will consider in this work the Strang splitting.
It has the favorable property of being second-order in time, of being time-reversible, and requiring only three substeps per time step.
It will be applied as follows: over a time step \( [0, \Delta t] \),
\begin{enumerate}
    \item Solve the equation \( \ii \dot{u} = -\Delta u \) over the time interval \( [0, \Delta t/2] \).
    Let \( \widetilde{u}_{1/2} \) be the solution at time \( t=\Delta t/2 \).

    \item Solve the equation \( \ii \dot{u} = Vu \) over the time interval \( [0, \Delta t] \), with \( u(0, \cdot) = \widetilde{u}_{1/2} \).
    Let \( \widetilde{u}_1 \) be the solution at time \( t=\Delta t \).
    The integration is performed by using the equations of motion with \( H=V \).
    The time integrator remains to be specified.

    \item Solve the equation \( \ii \dot{u} = -\Delta u \) over the time interval \( [\Delta t/2, \Delta t] \), with \( u(\Delta t/2, \cdot) = \widetilde{u}_1 \).
\end{enumerate}
The result of the third substep at time \( t=\Delta t \) is an approximation to the exact solution of \( \ii\dot{w} = (-\Delta + V)w \) at time \( t=\Delta t \) with \( w(0,\cdot) = u(0, \cdot) \), and the error of the approximation is of size \( \mathcal{O}(\Delta t^3) \).

We mentioned previously that the solution of the equation \( \ii\dot{u} = -\Delta u \) can be computed exactly in time, let us give some details.
We have
\begin{equation*}
    \dot{u}  = \sum_{k=0}^n \left[ \dot{c}_{2k} \varphi^{\otimes N-2k} \otimes_s \chi_{2k} + c_{2k} \sqrt{N-2k} (\varphi^{\otimes N-2k-1} \otimes_s \dot{\varphi}) \otimes_s \chi_{2k} + c_{2k} \varphi^{\otimes N-2k} \otimes_s \dot{\chi}_{2k} \right]
\end{equation*}
and
\begin{equation*}
    \Delta u = \sum_{k=0}^n \left[ c_{2k} \sqrt{N-2k} (\varphi^{\otimes N-2k-1} \otimes_s \Delta \varphi) \otimes_s \chi_{2k} + c_{2k} \varphi^{\otimes N-2k} \otimes_s \Delta\chi_{2k} \right].
\end{equation*}
By solving the free Schrödinger equations
\begin{equation}
    \label{eqn: operator splitting - free schrod eqns phi chi2}
    \ii \dot{\varphi} = -\Delta \varphi
    \quad\text{ and }\quad
    \ii \dot{\chi}_{2k} = -\Delta \chi_{2k},
\end{equation}
and letting \( \dot{c}_0 = \dots = \dot{c}_{2n} = 0 \), we are effectively solving the free Schrödinger equation \( \ii \dot{u} = -\Delta u \).
After a Fourier transform, each equation of the form \( \ii \dot{f} = -\Delta f \) can be reformulated as a system of linear ordinary differential equations for the Fourier coefficients of \( f \), and this differential system can be integrated exactly in time.

\bigskip
For the second step in the operator splitting described above, we remove the kinetic operator from \( H \).
The equation satisfied by the function \( \varphi \) is now
\begin{equation}
    \label{eqn: operator splitting - potential hartree equation}
    \ii \dot{\varphi} = (v \ast \abs{\varphi}^2 - \mu_{\varphi}) \varphi,\text{ on } \Omega, \quad \varphi(0) = \varphi_0,
\end{equation}
where the coefficient \( \mu_{\varphi} \) is given by \eqref{eqn: mu in hartree equation}.
We have the following classical result (see e.g. \cite{faouGeometricNumericalIntegration2012}*{Section~2.2} for a very similar statement):
\begin{lemma}
    \label{lemma: operator splitting - potential hartree equation with phi0}
    Let \( \varphi \) be the solution to \eqref{eqn: operator splitting - potential hartree equation} with initial condition \( \varphi_0 \).
    Then, for all \( x\in \Omega \) and all times \( t\geq 0 \),
    \begin{equation}
        \label{eqn: operator splitting - potential hartree equation with phi0}
        \ii \dot{\varphi}(t) = (v\ast \abs{\varphi_0}^2 - \mu_{\varphi_0}) \varphi(t).
    \end{equation}
\end{lemma}

\begin{proof}
    For \( x \in \Omega \) and \( t\geq 0 \),
    \begin{align*}
        \frac{\partial}{\partial t} \abs{\varphi(t, x)}^2
        &= \frac{\partial}{\partial t} \left( \conj{\varphi(t, x)} \varphi(t, x) \right)
            = \conj{\dot{\varphi}(t, x)} \varphi(t, x) + \conj{\varphi(t, x)} \dot{\varphi}(t, x)
            = 2 \Re\left( \conj{\varphi(t, x)} \dot{\varphi}(t, x) \right) \\
        &= 2 \Re\left( -\ii \conj{\varphi(t, x)} (v \ast \abs{\varphi(t)}^2 - \mu_{\varphi(t)}) \varphi(t, x) \right) 
            = 2 \Re\left( -\ii (v \ast \abs{\varphi(t)}^2 - \mu_{\varphi(t)}) \abs{\varphi(t, x)}^2 \right).
    \end{align*}
    The potential function \( v \) is a real function hence \( (v \ast \abs{\varphi(t)}^2 - \mu_{\varphi(t)}) \abs{\varphi(t, x)}^2 \) is real, and we get \( \frac{\partial}{\partial t} \abs{\varphi(t, x)}^2 = 0 \).
\end{proof}

Note the difference between \eqref{eqn: operator splitting - potential hartree equation} and \eqref{eqn: operator splitting - potential hartree equation with phi0}: in the latter equation, the function \( v \) is convolved with the initial condition \( \varphi_0 \) and the coefficient \( \mu \) is computed with respect to the initial condition too, making the function \( v\ast \abs{\varphi_0}^2 - \mu_{\varphi_0} \) time-independent.
If the initial condition \( \varphi_0 \) is moreover continuous, we have
\begin{equation}
    \label{eqn: exact integration potential hartree}
    \varphi(t,x) = e^{-\ii t (v\ast \abs{\varphi_0}^2 - \mu_{\varphi_0})(x)} \varphi_0(x),\quad x\in \Omega.
\end{equation}

\subsection{Low-rank approximation of \texorpdfstring{\( \chi_{2j} \)}{chi2j}}
\label{section: LRA of excitation functions}

Dealing numerically with functions of \( 2j \) variables is computationally demanding, and it is also difficult to enforce the orthogonality \eqref{eqn: bogoliubov orthogonality relation} as well as the symmetry with respect to permutations of variables.
In other words, dealing numerically with the functions \( \chi_{2j} \) can be challenging.

A way to bypass these issues consists in considering a low-rank approximation to \( \chi_2 \), which means we assume that the following approximation holds for some rank \( r \geq 1 \):
\begin{equation}
    \label{eqn: LRA chi2}
    \chi_2(t,x_1, x_2) \approx \sum_{l_1, l_2=1}^r b_{l_1, l_2}(t) \eta_{l_1}(t,x_1) \eta_{l_2}(t, x_2),
\end{equation}
where \( B = (b_{k,l})_{k,l=1}^r \in \mathbb{C}^{r\times r} \) is symmetric and invertible at all times, and \( \eta_l \in L^2(\Omega) \) at all times.
At time \( t=0 \), the family \( \{\varphi, \eta_1, \dots, \eta_r\} \) is assumed to be orthonormal.
The normalization of \( \chi_2(t) \) means that the Frobenius norm of \( B(t) \), namely \( \sqrt{\sum_{k,l=1}^r \abs{b_{k,l}(t)}^2} \), is equal to one for all times \( t \).
Due to the tensorization property \eqref{eqn: Bogoliubov tensorization}-\eqref{eqn: Bogoliubov tensorization normalized} of the functions \( \chi_{2j} \), we also have
\begin{equation}
    \label{eqn: LRA chi2j}
    \chi_{2j}(t, x_1, \dots, x_{2j}) = \sum_{l_1, \dots, l_{2j}=1}^r b_{l_1, \dots, l_{2j}}(t) \eta_{l_1}(t, x_1) \dots \eta_{l_{2j}}(t, x_{2j}),
\end{equation}
with the tensor \( b_{l_1, \dots, l_{2j}} \) having a Frobenius norm equal to one for all times and determined from the matrix \( B \):
\begin{equation}
    \label{eqn: core tensor chi2j using tensorization}
    b_{l_1, \dots, l_{2j}} = \frac{\widetilde{b}_{l_1, \dots, l_{2j}}}{\norm{\widetilde{b}}{F}},\qquad
    \widetilde{b}_{l_1, \dots, l_{2j}} := \sum_{\tau\in P_{2j}} b_{l_{\tau(1)}, l_{\tau(2)}} \dots b_{l_{\tau(2j-1)}, l_{\tau(2j)}}.
\end{equation}
Note that for simplicity we use the same notation ``\( b \)'' for the core tensors of \( \chi_2 \) and \( \chi_{2j} \), and they are distinguished only by the number of their indices: \( b_{l_1, \dots, l_{2j}} \) refers to the core tensor of \( \chi_{2j} \) because it has \( 2j \) indices.

After performing the space discretization described in Section \ref{section: time and space discretizations}, each function \( \eta_l \) is discretized.
We can consider the matrix \( U \), where each column \( U_{\cdot, l} \) corresponds to the values of \( \eta_l \) on the computational box \( \Omega_{discr} \).
This leads to the discretization of \( \chi_2 \) on \( \Omega_{discr} \) as a matrix of the form \( U B U^\top \).
At time \( t=0 \), the matrix \( U \) is unitary.
This particular structure of the discretization of \( \chi_2 \) simply stems from the symmetry of \( \chi_2 \) that is assumed at all times, hence we need to take care of preserving it during the time integration of \( \chi_2 \).
In particular, the low-rank structure of \( \chi_2 \) needs to be kept at all times, and this means that the right-hand side of the \( \dot{\chi}_2 \) equation in \eqref{eqn: equations of motion} has to be projected onto the tangent space of the manifold of rank-\( r \) matrices.
More details on how this projection is performed can be found in \cite{kochDynamicalLowRankApproximation2007}, where evolution equations for the factors \( U \) and \( B \) are given, but they involve the inverse matrix \( B^{-1} \).
The main issue is that the matrix \( B \) typically has small singular values, hence it is ill-conditioned and so it is not advised to invert it numerically.
There are fortunately alternatives to this, and among them we choose the BUG integrator from \cites{cerutiUnconventionalRobustIntegrator2022,cerutiTimeIntegrationSymmetric2020,cerutiRankadaptiveRobustIntegrator2022}, which is robust to arbitrarily small singular values.
The essence of this integrator is that the basis matrix \( U \) is updated in a first step by integrating the differential equation for $K=UB$ over a time step and computing the updated $U$ as an orthonormal basis of the range of the updated $K$, and in a second step the matrix \( B \) is updated in time with respect to the new basis $U$. 

In this work, we will consider the second-order rank-adaptive BUG(\( 2r \)) integrator proposed in \cite{cerutiRobustSecondorderLowrank2024}.
It is given for rank-\( r \) matrices of the form \( U S V^* \) with \( U, V \) unitary and \( S \) invertible.
The integrator needs to be slightly modified in order to preserve the complex symmetric structure \( U B U^\top \).
Some details are given in Appendix \ref{appendix: BUG for complex symmetric matrices}.
This integrator also allows us to have a rank-adaptive algorithm, where the rank \( r \) of the approximation of \( \chi_2 \) is adapted at every time step to a new rank \( \tilde{r} \leq 4r \).
The exact value of \( \tilde{r} \) may depend on numerical parameters: for example, we may want to have the singular values of \( B \) greater than some numerical threshold, or to keep \( \tilde{r} \leq r_{\max} \) for some given maximal rank \( r_{\max} \).
We will discuss this in more detail in Section \ref{section: numerical experiments}.

We also note that the low-rank approximation of the excitation functions \( \chi_{2k} \) means that the free Schrödinger equations given in \eqref{eqn: operator splitting - free schrod eqns phi chi2} can be solved by solving the following one-particle free Schrödinger equations:
\begin{equation}
    \label{eqn: free Schrodinger equations on phi and etas}
    \ii\dot{\varphi} = -\Delta \varphi, \quad \ii\eta_j = -\Delta \eta_j, \quad 1 \leq j \leq r,
\end{equation}
along with \( \dot{B} = 0 \).

\subsubsection{Using low-rank approximations to further reduce computational complexity}

The low-rank approximations \eqref{eqn: LRA chi2} of \( \chi_2 \) and \eqref{eqn: LRA chi2j} of \( \chi_{2j} \) allow us to simplify some brakets of Section \ref{section: computational complexity equations of motion} and to further reduce the computational complexity of the method.
Recall that we have shown the following equality in Section \ref{section: reducing the number of permutations and variables in braket}:
\begin{equation*}
    \braket{\varphi^{\otimes N-m-2p} \otimes \chi_{2p}}{V_{\ell_1, \ell_2} u}_{x_{\intinterval{m+1}{N}}}
    = \sum_{q=0}^n \frac{c_{2q}}{\sqrt{\binom{N}{2q}}} \sum_{\substack{\tau \in S_{\mathcal{T}^C}^{2q}\\\tau\,\mathrm{sorted,\,valid}}} \braket{\varphi^{\otimes N-m-2p} \otimes \chi_{2p}}{V_{\ell_1, \ell_2} \varphi^{\otimes N-2q}(x \setminus x_\tau) \chi_{2q}(x_{\tau})}_{\substack{x_{\mathcal{B}} \\ \mathbf{1}_{\mathcal{T}}}},
\end{equation*}
where the sets \( \mathcal{T} \) and \( \mathcal{T}^C \) were defined in \eqref{eqn: def calT}, \( \mathcal{B} \) in \eqref{eqn: def calB}, and a \emph{valid permutation} is defined in Definition \ref{def: valid permutation tau}.
We also recall that \( \mathcal{T} \) is the set of variables in \( \intinterval{m+1}{N} \) that can be ``skipped'' during the computation of the braket, and \( \mathcal{B} \) is the set of variables in \( \intinterval{m+1}{N} \) that actually determine the value of the braket.
Using the low-rank approximation \eqref{eqn: LRA chi2j}, we can write
\begin{multline}
    \label{eqn: subsect LRA braket from previous section on reducing nb of perms and vars}
    \braket{\varphi^{\otimes N-m-2p} \otimes \chi_{2p}}{V_{\ell_1, \ell_2} u}_{x_{\intinterval{m+1}{N}}} \\
    = \sum_{q=0}^n \frac{c_{2q}}{\sqrt{\binom{N}{2q}}} \sum_{\substack{\tau \in S_{\mathcal{T}^C}^{2q}\\\tau\,\mathrm{sorted,\,valid}}} \sum_{\substack{i'_1, \dots, i'_{2p}=1\\i_1, \dots, i_{2q}=1}}^r \conj{b_{i'_1, \dots, i'_{2p}}} b_{i_1, \dots, i_{2q}}
        \braket{\varphi^{\otimes N-m-2p} \otimes \\ \eta_{i'_1} \otimes \dots \otimes \eta_{i'_{2p}}
        }{V_{\ell_1, \ell_2} \left(\begin{gathered} \varphi^{\otimes N-2q}(x \setminus x_\tau)\\ \eta_{i_1}(x_{\tau(1)}) \cdots \eta_{i_{2q}}(x_{\tau(2q)}) \end{gathered}\right)}_{\substack{x_{\mathcal{B}}\\\mathbf{1}_{\mathcal{T}}}}.
\end{multline}
For any \( \tau\in S^{2q}_{\mathcal{T}^C} \), let us define the following sets:
\begin{equation}
    \begin{aligned}
        \mathcal{Q} &= \left\{ k\in \intinterval{1}{2q} : \tau(k) \in \intinterval{N-2p+1}{N} \setminus \{\ell_1, \ell_2\} \right\}, && \mathcal{Q}^C := \intinterval{1}{2q} \setminus \mathcal{Q}, \\
        \mathcal{P} &= \tau(\mathcal{Q})-N+2p \in \intinterval{1}{2p}, && \mathcal{P}^C = \intinterval{1}{2p} \setminus \mathcal{P}.
    \end{aligned}
\end{equation}
The set \( \mathcal{Q} \) is the set of marginals in \( \chi_{2q} \) from the right argument of the braket that are paired with marginals of \( \chi_{2p} \) from the left argument of the braket, according to the permutation \( \tau \) and omitting the variables that are acted on by the operator.
The set \( \mathcal{P} \) is the set of those marginals in \( \chi_{2p} \).
For the sake of readability, we omit in \( \mathcal{P} \) and \( \mathcal{Q} \) the explicit dependence on \( \ell_1, \ell_2 \) and \( \tau \).
Moreover, we will use the following notation: if \( \mathcal{I} \) is a set of indices and \( i_{\mathcal{I}} \) a multi-index, the sum \( \sum_{i_{\mathcal{I}}=1}^r \) corresponds to a sum from \( 1 \) to \( r \) of every index in \( i_{\mathcal{I}} \).

The orthonormality of the functions \( \eta_k \) can be used to reduce the computational cost of the summand in \eqref{eqn: subsect LRA braket from previous section on reducing nb of perms and vars}:
\begin{lemma}
    \label{lemma: get rid of all one-dimensional brakets of the form phi-phi}
    Let \( p,q\in \intinterval{0}{n} \) and \( \tau \in S^{2q}_{\mathcal{T}^C} \) a valid sorted permutation.
    We have
    \begin{multline}
        \label{eqn: simplification braket potential part with multiindices calQ and calP}
        \braket{\varphi^{\otimes N-m-2p} \otimes \chi_{2p}}{V_{\ell_1, \ell_2} \varphi^{\otimes N-2q}(x\setminus x_{\tau}) \chi_{2q}(x_\tau)}_{\substack{x_{\mathcal{B}} \\ \mathbf{1}_{\mathcal{T}}}} \\
        = \sum_{i'_{\mathcal{P}^C}=1}^r \sum_{i_{\mathcal{Q}^C}=1}^r \braket{\varphi^{\otimes N-m-2p}\otimes \\  \eta_{i'_1} \otimes \dots \otimes \eta_{i'_{2p}}}{V_{\ell_1, \ell_2}\left( \begin{gathered} \varphi^{\otimes N-2q}(x\setminus x_{\tau}) \\ (\eta_{i_1} \otimes \dots \otimes \eta_{i_{2q}})(x_\tau) \end{gathered} \right)}_{\substack{x_{\{\ell_1, \ell_2\}\setminus\intinterval{1}{m}}\\\mathbf{1}_{\intinterval{m+1}{N}\setminus\{\ell_1, \ell_2\}}}}
            \sum_{i'_{\mathcal{P}}=1}^r \conj{b_{i'_{\mathcal{P}}, i'_{\mathcal{P}^C}}} b_{i'_\mathcal{P}, i_{\mathcal{Q}^C}}.
    \end{multline}
\end{lemma}

\begin{proof}
    We use the rank-\( r \) approximation \eqref{eqn: LRA chi2j} of \( \chi_{2p} \) and \( \chi_{2q} \):
    \begin{multline*}
        \braket{\varphi^{\otimes N-m-2p} \otimes \chi_{2p}}{V_{\ell_1, \ell_2}\left( \varphi^{\otimes N-2q}(x\setminus x_{\tau}) \chi_{2q}(x_\tau) \right)}_{\substack{x_{\mathcal{B}} \\ \mathbf{1}_{\mathcal{T}}}} \\
        = \sum_{\substack{i'_1, \dots, i'_{2p}=1\\i_1, \dots, i_{2q}=1}}^r \conj{b_{i'_1, \dots, i'_{2p}}} b_{i_1, \dots, i_{2q}} \braket{\varphi^{\otimes N-m-2p}\otimes \eta_{i'_1} \otimes \dots \otimes \eta_{i'_{2p}}}{V_{\ell_1, \ell_2} \varphi^{\otimes N-2q}(x\setminus x_\tau) (\eta_{i_1} \otimes \dots \otimes \eta_{i_{2q}})(x_\tau)}_{\substack{x_{\mathcal{B}} \\ \mathbf{1}_{\mathcal{T}}}}.
    \end{multline*}
    We can write the sum over \( (i_1, \dots, i_{2q}) \) as two sums: one over the multi-index \( i_{\mathcal{Q}^C} \) and one over the multi-index \( i_{\mathcal{Q}} \):
    \begin{multline*}
        \braket{\varphi^{\otimes N-m-2p} \otimes \chi_{2p}}{V_{\ell_1, \ell_2} \varphi^{\otimes N-2q}(x\setminus x_{\tau}) \chi_{2q}(x_\tau)}_{\substack{x_{\mathcal{B}} \\ \mathbf{1}_{\mathcal{T}}}} \\
        = \sum_{i'_1, \dots, i'_{2p}=1}^r \sum_{i_{\mathcal{Q}^C}=1}^r \sum_{i_{\mathcal{Q}}=1}^r \conj{b_{i'_1, \dots, i'_{2p}}} b_{i_1, \dots, i_{2q}} \braket{\varphi^{\otimes N-m-2p} \otimes \\ \eta_{i'_1} \otimes \dots \otimes \eta_{i'_{2p}}}{V_{\ell_1, \ell_2} 
            \left(\begin{gathered}
                \varphi^{\otimes N-2q}(x\setminus x_{\tau}) \\
                (\eta_{i_1} \otimes \dots \otimes \eta_{i_{2q}})(x_\tau)
            \end{gathered}\right)}_{\substack{x_{\mathcal{B}} \\ \mathbf{1}_{\mathcal{T}}}}.
    \end{multline*}
    The set \( \mathcal{Q} \) is the set of indices \( k \) for which the braket along variable \( x_{\tau(k)} \) is of the form \( \braket{\eta_{i'_{\tau(k)-N+2p}}}{\eta_{i_k}} = \delta_{i'_{\tau(k)-N+2p}, i_k} \). 
    By definition of the set \( \mathcal{P} \), we get
    \begin{multline*}
        \braket{\varphi^{\otimes N-m-2p} \otimes \chi_{2p}}{V_{\ell_1, \ell_2} \varphi^{\otimes N-2q}(x\setminus x_{\tau}) \chi_{2q}(x_\tau)}_{\substack{x_{\mathcal{B}} \\ \mathbf{1}_{\mathcal{T}}}} \\
        = \sum_{i'_1, \dots, i'_{2p}=1}^r \sum_{i_{\mathcal{Q}^C}=1}^r \sum_{i_{\mathcal{Q}}=1}^r \conj{b_{i'_1, \dots, i'_{2p}}} b_{i_1, \dots, i_{2q}} \braket{\varphi^{\otimes N-m-2p} \otimes \\ \eta_{i'_1} \otimes \dots \otimes \eta_{i'_{2p}}}{V_{\ell_1, \ell_2}\left( \begin{gathered} \varphi^{\otimes N-2q}(x\setminus x_{\tau}) \\ (\eta_{i_1} \otimes \dots \otimes \eta_{i_{2q}})(x_\tau) \end{gathered} \right)}_{\substack{x_{\mathcal{B}\setminus \tau(\mathcal{Q})} \\ \mathbf{1}_{\mathcal{T} \cup \tau(\mathcal{Q})}}} \delta_{i'_{\mathcal{P}}, i_\mathcal{Q}}.
    \end{multline*}
    The indices \( i_1, \dots, i_{2q} \) still appear above in the arguments of the braket to simplify the notation, but the one-dimensional brakets along variables indexed by \( \tau(\mathcal{Q}) \) are known to be equal to one and thus do not need to be computed in practice.
    Moreover, the permutation \( \tau \) being valid, by Definition \ref{def: valid permutation tau} we have \( \intinterval{N-2p+1}{N}\setminus \{\ell_1, \ell_2\} \subset \tau \).
    It means that
    \begin{equation*}
        \forall j\in \intinterval{N-2p+1}{N}\setminus \{\ell_1, \ell_2\}, \exists k\in \intinterval{1}{2q} \quad s.t. \quad \tau(k) = j.
    \end{equation*}
    This implies \( \tau(\mathcal{Q}) = \intinterval{N-2p+1}{N}\setminus \{\ell_1, \ell_2\} \).
    Thus we have \( \mathcal{B}\setminus \tau(\mathcal{Q}) = \{\ell_1, \ell_2\} \setminus \intinterval{1}{m} \) and \( \mathcal{T} \cup \tau(\mathcal{Q}) = \intinterval{m+1}{N} \setminus \{\ell_1, \ell_2\} \).
    Now we do a similar decomposition on the sum over the indices \( (i'_1, \dots, i'_{2p}) \): it is a sum over the multi-index \( i'_{\mathcal{P}^C} \) and a sum over the multi-index \( i'_\mathcal{P} \).
    Then,
    \begin{multline*}
        \braket{\varphi^{\otimes N-m-2p} \otimes \chi_{2p}}{V_{\ell_1, \ell_2} \varphi^{\otimes N-{2q}}(x\setminus x_{\tau}) \chi_{2q}(x_\tau)}_{\substack{x_{\mathcal{B}} \\ \mathbf{1}_{\mathcal{T}}}} \\
        = \sum_{i'_{\mathcal{P}^C}=1}^r \sum_{i_{\mathcal{Q}^C}=1}^r \braket{\varphi^{\otimes N-m-2p} \\ \otimes \eta_{i'_1} \otimes \dots \otimes \eta_{i'_{2p}}}{V_{\ell_1, \ell_2}\left( \begin{gathered} \varphi^{\otimes N-2q}(x\setminus x_{\tau}) \\ (\eta_{i_1} \otimes \dots \otimes \eta_{i_{2q}})(x_\tau) \end{gathered} \right)}_{\substack{x_{\{\ell_1, \ell_2\} \setminus \intinterval{1}{m}} \\ \mathbf{1}_{\intinterval{m+1}{N} \setminus \{\ell_1, \ell_2\}}}} \sum_{i'_{\mathcal{P}}=1}^r \sum_{i_{\mathcal{Q}}=1}^r \delta_{i'_\mathcal{P}, i_\mathcal{Q}} \conj{b_{i'_1, \dots, i'_{2p}}} b_{i_1, \dots, i_{2q}}.
    \end{multline*}
    Again, we have kept all the indices \( i'_1, \dots, i'_{2p} \) in the braket to simplify the notation, but only the braket along variables indexed by \( \{\ell_1, \ell_2\} \setminus \intinterval{1}{m} \) needs to be computed and the braket along all other variables indexed by \( \intinterval{m+1}{N} \setminus \{\ell_1, \ell_2\} \) is equal to one.
    We conclude using the symmetry of the core tensors.
\end{proof}

Finally, we have shown the following:
\begin{multline}
    \label{eqn: summary all efficient sums braket varphi Nmp chip vs Vu}
        \braket{\varphi^{\otimes N-m-2p} \otimes \chi_{2p}}{Vu}_{x_{\intinterval{m+1}{N}}} \\
        = \sum_{q=0}^n \sum_{\ell_1, \ell_2} \frac{\mathsf{M}_{\ell_1, \ell_2}}{\sqrt{\binom{N}{2q}}} \sum_{\substack{\tau \in S^{2q}_{\mathcal{T}^C}\\\tau\,\mathrm{sorted,\,valid}}} \sum_{\substack{i'_{\mathcal{P}^C}=1\\i_{\mathcal{Q}^C}=1}}^r 
        \mathsf{P}^{(p,q)}_{i'_{\mathcal{P}^C}, i_{\mathcal{Q}^C}} \braket{\varphi^{\otimes N-m-2p} \otimes \\ \eta_{i'_1} \otimes \dots \otimes \eta_{i'_{2p}}}{V_{\ell_1, \ell_2}\left( \begin{gathered} 
            \varphi^{\otimes N-2q}(x\setminus x_\tau) \\ (\eta_{i_1} \otimes \dots \otimes \eta_{i_{2q}})(x_\tau) 
        \end{gathered} \right)}_{\substack{x_{\{\ell_1, \ell_2\} \setminus \intinterval{1}{m}} \\ \mathbf{1}_{\intinterval{m+1}{N} \setminus \{\ell_1, \ell_2\}}}}
\end{multline}
with
\begin{equation}
    \label{eqn: coeff Ppq calPc calQc}
    \mathsf{P}^{(p,q)}_{i'_{\mathcal{P}^C}, i_{\mathcal{Q}^C}} := \sum_{i'_{\mathcal{P}}=1}^r \conj{b_{i'_\mathcal{P}, i'_{\mathcal{P}^C}}} b_{i'_\mathcal{P}, i_{\mathcal{Q}^C}},
\end{equation}
and where the sum over \( \ell_1, \ell_2 \) is over appropriate values of \( \ell_1 \) and \( \ell_2 \) given by Lemma \ref{lemma: reducing number of effective potential operators}.
The coefficient \( \mathsf{M}_{\ell_1, \ell_2} \) corresponds to the multiplicative coefficient associated to \( V_{\ell_1, \ell_2} \) in Lemma \ref{lemma: reducing number of effective potential operators}.
The equivalent expression for the kinetic operator is trivially obtained using Lemma \ref{lemma: reducing number of effective kinetic operators}: the sum over \( \ell_1, \ell_2 \) is transformed into a sum over \( k \in \intinterval{1}{m} \cup \{m+1, N-2p+1\} \), and the associated coefficient \( \mathsf{M}_k \) is also given by Lemma \ref{lemma: reducing number of effective kinetic operators}.
The sets \( \mathcal{B}, \mathcal{P}, \mathcal{Q}, \mathcal{T} \) are defined in this case by letting \( \ell_1 = \ell_2 = k \).

Note that the computations with a low-rank approximation to the excitation functions can be performed efficiently in the case of a diagonal matrix \( B \), e.g. computing \( b_{i_1, \dots, i_{2j}} \) for \( j\in \intinterval{2}{n} \) or computing the coefficient \( \mathsf{P}^{(p,q)}_{i'_{\mathcal{P}^C}, i_{\mathcal{Q}^C}} \).
This requires in particular that the matrix \( B \) remains diagonal, even during the substeps of the BUG integration.
This is not guaranteed with the standard BUG integrator but can be imposed at a moderate computational cost, some details are given in Appendix \ref{appendix: BUG for complex symmetric matrices}.
Some efficient implementation details that use \( B \) diagonal are given in Appendix \ref{appendix: implementation details}.

\subsection{Dealing with the potential using convolutions}

One of the advantages and reasons that we introduce the low-rank approximation \eqref{eqn: LRA chi2} for \( \chi_2 \), in addition to those mentioned previously, is that it is now enough to deal with functions of one particle.
This greatly reduces the computational cost in practice, compared to dealing with \( \chi_2 \) as a genuine two-particle function.

However, by looking at brakets of the form \eqref{eqn: summary all efficient sums braket varphi Nmp chip vs Vu}, it is clear that we still need to deal with potential operators \( V_{\ell_1, \ell_2} \) which are multiplicative two-body operators defined by \( (V_{\ell_1, \ell_2} f)(x) = v(x_{\ell_1} - x_{\ell_2}) f(x) \) for \( f\in L^2(\Omega^N) \).
The function \( v(x_{\ell_1} - x_{\ell_2}) \) is a function of two particles, and we want to reduce the computational cost associated to dealing with it.
A first possibility is that the potential function \( v \) is separable in the following (low-rank) sense: there are functions \( (h_i)_{i=1}^{r_v} \) and a matrix \( (m_{i,j})_{i,j=1}^{r_v} \) such that \( v(x-y) = \sum_{i,j=1}^{r_v} m_{i,j} h_i(x)h_j(y) \), for some \( r_v \geq 1 \) and all \( x,y\in \Omega \).
In this case, two-particle brakets involving the potential function \( v(x_{\ell_1}-x_{\ell_2}) \) can be rewritten using one-particle brakets only.
If the potential function \( v \) is not separable in the above sense, it can be approximated by a sum of separable functions, for example a sum of exponentials \cites{beylkinApproximationFunctionsExponential2005,kumePracticalMethodTreating1999}.

A second possibility consists in looking carefully at how the potential \( v(x-y) \) is involved in our method, and noticing that it is never needed to actually deal with \( v(x-y) \) as a two-particle function.
The computations that we need to perform involving the potential operator are of the form \eqref{eqn: summary all efficient sums braket varphi Nmp chip vs Vu}, and the potential function \( v(x_{\ell_1}-x_{\ell_2}) \) appears under one of the following forms:
\begin{enumerate}[label=(\alph*)]
    \item \( v(x_{\ell_1}-x_{\ell_2})g_1(x_{\ell_1}) g_2(x_{\ell_2}) \)
    \item \( \braket{f_1(x_{\ell_1})}{v(x_{\ell_1}-x_{\ell_2}) g_1(x_{\ell_1}) g_2(x_{\ell_2})}_{x_{\ell_1}} = g_2(x_{\ell_2}) (v\ast [\conj{f_1} g_1])(x_{\ell_2}) \)
    \item \( \braket{f_1(x_{\ell_1}) f_2(x_{\ell_2})}{v(x_{\ell_1}-x_{\ell_2}) g_1(x_{\ell_1}) g_2(x_{\ell_2})}_{x_{\ell_1}, x_{\ell_2}} = \braket{f_1(x_{\ell_1})}{g_1(x_{\ell_1}) (v\ast [\conj{f_2}g_2])(x_{\ell_1})} \)
\end{enumerate}
for some functions \( f_1, f_2, g_1, g_2 \in \{\varphi, \eta_1, \dots, \eta_r\} \), and where \( f\ast g \) denotes the usual convolution \( (f\ast g)(x) = \int_{\mathbb{R}^d} f(x-y)g(y)dy \).
Note that the terms of type (a) and (b) only appear during the computation of the equation of motion associated to \( \chi_2 \), hence for those terms we have \( f_1 = \varphi \).
For the convolution involved in terms of type (b) and (c), we recall that a convolution in the space domain is equivalent to a pointwise multiplication in the Fourier domain.
In practice, after the space discretization is performed, we use a Fast Fourier Transform to compute the terms of type (b) and (c) in time \( \mathcal{O}(d M^d \log M) \), where \( M \) is the number of points along each dimension of the computational grid (see Section \ref{section: time and space discretizations} for details about the space discretization).
It remains to deal with the terms of type (a), which are arguably the most problematic.
Such terms appear in the equation of motion for \( \chi_2 \) in \eqref{eqn: equations of motion}, more precisely with terms of the form
\begin{equation*}
    \proj_{\mathcal{N}^{(2)}_\varphi(u)} \braket{\varphi^{\otimes N-2}}{-\ii v(x_1-x_2) u}_{x_{\intinterval{3}{N}}}.
\end{equation*}
This braket is a sum of terms of type (a), but the sum is such that the braket results in a symmetric function of variables \( (x_1, x_2) \).
Thus, after the space discretization, the braket is discretized as a symmetric matrix \( F \in \mathbb{C}^{M^d \times M^d} \).
In the following, we make a slight abuse of notation and let \( \varphi,\chi_2 \) denote the discretization of the functions \( \varphi, \chi_2 \) over the computational grid \( \Omega_{discr} \).
That is, we now have in mind for the rest of this section that \( \varphi \) is a \( \mathbb{C}^{M^d} \) vector, and \( \chi_2 \) is a \( \mathbb{C}^{M^d \times M^d} \) matrix.
In order to compute the right-hand side of the equation of motion \eqref{eqn: equations of motion} for \( \chi_2 \), we need to perform the projection \( \proj_{\mathcal{N}^{(2)}_\varphi(u)}(F) \), where \( \mathcal{N}^{(2)}_\varphi(u) \) is now defined using the vector \( \varphi \) and matrix \( \chi_2 \) instead of the continuous functions.
The projection of \( F \) onto \( \{\chi_2\}^\perp \) is easy to perform:
\begin{equation*}
    \proj_{\{\chi_2\}^\perp} (F) = F - \braket{\chi_2}{F} \chi_2.
\end{equation*}
Here, the braket is the discretized analog of the continuous braket \( \braket{\cdot}{\cdot} \), and we again use the same notation for both.
Since the projection of \( F \) onto \( \{\chi_2\}^\perp \) can be performed easily, we assume without loss of generality that \( \proj_{\{\chi_2\}^\perp}(F) = F \).
The projection of \( F \) onto \( \{\varphi\}^\perp \times_s \{\varphi\}^\perp \) is more challenging.
Note that, since \( F \) is symmetric, the projection onto \( \{\varphi\}^\perp \times_s \{\varphi\}^\perp \) is equivalent to a projection onto \( \{\varphi\}^\perp \times \{\varphi\}^\perp \).

We observe that the equation of motion for \( \chi_2 \) is used only within the BUG integrator (see Section \ref{section: LRA of excitation functions}), and that the BUG integrator only involves the right-hand side of the differential equation within matrix-matrix products.
Therefore, after the space discretization, the right-hand side of the equation of motion for \( \chi_2 \) is never treated as a dense matrix but always involved in matrix-matrix products of the form \( \proj_{\mathcal{N}^{(2)}_\varphi(u)}(F)A \) for some (tall) matrix \( A \).
For simplicity, we treat each column of \( A \) separately: we now have to compute matrix-vector products of the form
\begin{equation*}
    \proj_{\{\varphi\}^\perp \times \{\varphi\}^\perp}(F) w,
\end{equation*}
for some vector \( w \in \mathbb{C}^{M^d} \).
For any \( \mathbf{x} \in \Omega_{discr} \), the vector \( \left(\left[\proj_{\{\varphi\}^\perp \times \{\varphi\}^\perp}(F)\right]_{\mathbf{x}, \mathbf{y}}\right)_{\mathbf{y}\in \Omega_{discr}} \in \mathbb{C}^{M^d} \) is orthogonal to \( \varphi \) with respect to the inner product \( \braket{v}{\tilde{v}} = v^* \tilde{v} \), which equivalently means that it is orthogonal to \( \conj{\varphi} \) with respect to the product \( (v, \tilde{v}) = v^\top \tilde{v} \).
We are interested in this non-conjugated product due to the complex symmetric BUG integrator that involves this product with no complex conjugation, see for instance \eqref{eqn: appendix - usual Kstep BUG integrator} in Appendix \ref{appendix: BUG for complex symmetric matrices}.
Thus, the component of \( w \) along \( \conj{\varphi} \) will vanish in the above matrix-vector product, and we get
\begin{equation*}
   \proj_{\{\varphi\}^\perp \times \{\varphi\}^\perp}(F) w
   =  \proj_{\{\varphi\}^\perp \times C^{M^d}}(F) \proj_{\{\conj{\varphi}\}} w.
\end{equation*}
The remaining projection onto \( \{\varphi\}^\perp \) can be performed after the matrix-vector product.
Finally, we have
\begin{equation*}
    \proj_{\{\varphi\}^\perp \times \{\varphi\}^\perp}(F) w = \proj_{\{\varphi\}^\perp} \left( F \proj_{\{\conj{\varphi}\}}(w) \right).
\end{equation*}
By doing so, we avoid the challenging projection of the marginals of \( F \) when computing the equation of motion associated to \( \chi_2 \).

\subsection{Time integration of the equations of motion}
\label{section: numerical integration of equations of motion}

We discuss in this section the time integration of the equations of motion \eqref{eqn: equations of motion} along with \eqref{eqn: equations of motion angles}, during the step of the splitting where \( H=V \) (see Section \ref{section: operator splitting}).
We recall that the function \( \varphi \) is solution to \eqref{eqn: operator splitting - potential hartree equation with phi0} which yields an exact time integration given by \eqref{eqn: exact integration potential hartree}.
The equation of motion for the vector of coefficients \( \mathbf{c} \) is of the form \( \ii \dot{\mathbf{c}} = \check{\mathbf{V}} \mathbf{c} \) with \( \check{\mathbf{V}} \) the Hermitian matrix described in \eqref{eqn: setting n>1 - equations of motion coeffs} for $H=V$.
The equation of motion for \( \chi_2 \) is numerically integrated using a complex symmetric variant of the BUG integrator, as mentioned in Section \ref{section: LRA of excitation functions}.
The other excitation functions are reconstructed from \( \chi_2 \) using the procedure detailed in Section \ref{section: integration eqns of motion n>1}.

Overall, after performing the operator splitting into the kinetic and potential parts, the differential system to integrate within the splitting step for the potential part reads as follows:
\begin{equation}
    \left\{
    \begin{aligned}
        \dot{\mathbf{c}} &= -\ii \check{\mathbf{V}} \mathbf{c} \\
        \dot{\chi}_2 &= \frac{1}{c_2} \sqrt{\binom{N}{2}} \proj_{\mathcal{N}_\varphi^{(2)}(u)} \braket{\varphi^{\otimes N-2}}{-\ii Vu}_{x_{\intinterval{3}{N}}} - \sqrt{2} \varphi \otimes_s \braket{\dot{\varphi}(y)}{\chi_2(y, \cdot)}_y \\
        \dot{\theta}_{2j} &= \frac{(2j)!}{2^j (j-1)!} \frac{\Im\braket{\widetilde{\chi}_{2j}}{\dot{\chi}_2 \otimes \chi_2^{\otimes j-1}}}{\braket{\widetilde{\chi}_{2j}}{\widetilde{\chi}_{2j}}},\quad 2 \leq j \leq n, \\
        \dot{\varphi} &= -\ii (v\ast\abs{\varphi}^2 - \mu_\varphi) \varphi.
    \end{aligned}
    \right.
    \label{eqn: differential system for potential part}
\end{equation}
Due to the reconstruction approach for the excitation functions, only \( \chi_2 \) needs to be integrated in time and not all of them.
Taken separately, each equation from the differential system can easily be integrated, but the time integration becomes more challenging when the differential system is considered as a whole.
This motivates the introduction of a splitting of the differential system: similarly to the operator splitting of Section \ref{section: operator splitting}, we use a combination of basic steps in order to simplify the time integration of the system, by paying the price of an error that depends on the splitting scheme and the time step \( \Delta t \).

Let us define 
\begin{equation}
    \label{eqn: dotchi21 and dotchi22}
    \dot{\chi}_2^{(1)} := \frac{1}{c_2} \sqrt{\binom{N}{2}} \proj_{\mathcal{N}_\varphi^{(2)}(u)} \braket{\varphi^{\otimes N-2}}{-\ii Vu}_{x_{\intinterval{3}{N}}}
    \quad\text{and}\quad \dot{\chi}_2^{(2)} := - \sqrt{2} \varphi \otimes_s \braket{\dot{\varphi}(y)}{\chi_2(y, \cdot)}_y,
\end{equation}
then the equation of motion for \( \chi_2 \) is \( \dot{\chi}_2 = \dot{\chi}_2^{(1)} + \dot{\chi}_2^{(2)} \).
The term \( \dot{\chi}_2^{(2)} \) enforces the orthogonality \eqref{eqn: bogoliubov orthogonality relation} that is imposed in the manifold \( \mathcal{M} \) in the Dirac--Frenkel time-dependent variational principle, and is only due to the motion of \( \varphi \).
For this reason, it would make sense to integrate simultaneously the equation of motion for \( \varphi \) -- which can be integrated exactly -- and the differential equation \( \dot{\chi}_2 = \dot{\chi}_2^{(2)} \).
This differential equation cannot be integrated exactly but because the role of its exact integration is clear, we can approximate it by replacing \( \chi_2 \) with an approximation that satisfies the orthogonality \eqref{eqn: bogoliubov orthogonality relation}.
The rest of the differential system \eqref{eqn: differential system for potential part} is dealt with using a Strang splitting.

Ultimately, the Step 2. in Section \ref{section: operator splitting} is dealt with using a Strang splitting for the differential system \eqref{eqn: differential system for potential part}, made of the following steps:
\begin{enumerate}
    \item Integrate the vector of coefficients \( \mathbf{c} = (c_0, c_2, \dots, c_{2n})^\top \), letting all other quantities constant with respect to time.
    A consequence is that the matrix \( \check{\mathbf{V}} \) is time-independent, and we obtain the unitary evolution \( \mathbf{c}(t) = e^{-\ii \check{\mathbf{V}}(0) t} \mathbf{c}(0) \).
    
    \item Integrate the angles \( \theta_2, \dots, \theta_{2n} \) according to \eqref{eqn: equations of motion angles}, letting all other quantities constant with respect to time.

    \item Integrate the equation for \( {\varphi} \), letting all other quantities constant with respect to time.
    The integration for \( \varphi \) is performed using \eqref{eqn: exact integration potential hartree}.

    \item Replace \( \chi_2 \) by an approximation \( \widehat{\chi}_2 \) that satisfies \eqref{eqn: bogoliubov orthogonality relation} with the current \( \varphi \).
    Details on how to obtain \( \widehat{\chi}_2 \) are given below.

    \item Integrate the excitation function \( \chi_2 \) by integrating the differential equation \( \dot{\chi}_2 = \dot{\chi}_2^{(1)} \), with initial condition \( \widehat{\chi}_2 \).
    The tensorization \eqref{eqn: Bogoliubov tensorization}-\eqref{eqn: Bogoliubov tensorization normalized} is used to reconstruct the functions \( \chi_4, \dots, \chi_{2n} \) from the function \( \chi_2 \) and the angles \( \theta_4, \dots, \theta_{2n} \).
\end{enumerate}

\subsubsection{Preserving the orthogonality}
\label{section: preserving bogoliubov orthogonality after integrating varphi}

Let us detail the approximation \( \widehat{\chi}_2 \) of \( \chi_2 \) that satisfies \eqref{eqn: bogoliubov orthogonality relation}, involved in Step 4. above: after having integrated \( \varphi \) over a time interval \( [0, t] \), the function \( \varphi \) is known at time \( t \) but the excitation function \( \chi_2 \) is still only known at time \( 0 \).
In particular, the orthogonality \eqref{eqn: bogoliubov orthogonality relation} is lost and the results from Sections \ref{section: equations of motion} and \ref{section: computational complexity equations of motion} do not hold anymore.
In order to enforce the orthogonality, we replace \( \chi_2 \) by some \( \widehat{\chi}_2 \) close to \( \chi_2 \) which satisfies \eqref{eqn: bogoliubov orthogonality relation}.
This is done by solving the following optimization problem:
\begin{equation}
    \label{eqn: integration chi22 initial minimization problem}
    \begin{gathered}
        \norm{\widehat{\chi}_2 - \chi_2(0)}{L^2}^2 = \min_{\widehat{\chi}_2 \in L_s^2(\Omega^2)} !\\
        s.t.\quad \begin{aligned}[t]
            &\braket{\varphi(t, y)}{\widehat{\chi}_2(y, \cdot)}_y = 0\\
            &\norm{\widehat{\chi}_2}{L^2} = 1
        \end{aligned}
    \end{gathered}
\end{equation}
The solution to this optimization problem can be easily obtained.
Indeed, a decomposition of \( \chi_2(0) \) depending on \( \varphi(t) \) gives
\begin{equation*}
    \chi_2(0) = a \varphi(t) \otimes \varphi(t) + \varphi(t) \otimes_s \zeta + \kappa,
\end{equation*}
for some coefficient \( a\in \mathbb{C} \) and some functions \( \zeta \in L^2(\Omega) \) such that \( \braket{\varphi(t)}{\zeta} = 0 \) and \( \kappa \in L^2_s(\Omega^2) \) such that \( \braket{\varphi(t, y)}{\kappa(y, \cdot)}_y = 0 \).
Then, 
\begin{equation*}
    \norm{\widehat{\chi}_2 - \chi_2(0)}{L^2}^2 
    = \norm{\widehat{\chi}_2 - a \varphi(t) \otimes \varphi(t) - \varphi(t) \otimes_s \zeta - \kappa}{L^2}^2.
\end{equation*}
Due to the orthogonality \( \braket{\varphi(t, y)}{\widehat{\chi}_2(y, \cdot)}_y = 0 \), the minimizer from \eqref{eqn: integration chi22 initial minimization problem} is the same minimizer as for the following problem:
\begin{equation*}
    \begin{gathered}
        \norm{\widehat{\chi}_2 - \kappa}{L^2}^2 = \min_{\widehat{\chi}_2 \in L_s^2(\Omega^2)} !\\
        s.t.\quad \begin{aligned}[t]
            &\braket{\varphi(t, y)}{\widehat{\chi}_2(y, \cdot)}_y = 0\\
            &\norm{\widehat{\chi}_2}{L^2} = 1
        \end{aligned}
    \end{gathered}
\end{equation*}
The minimum is attained for \( \widehat{\chi}_2 = \kappa / \norm{\kappa}{L^2} \).
Note that \( \kappa = \proj_{\{\varphi(t)\}^\perp \times_s \{\varphi(t)\}^\perp}(\chi_2(0)) \), and due to the diagonal low-rank approximation of \( \chi_2 \) we get 
\begin{equation*}
    \kappa = \sum_{i=1}^r b_{i, i} \proj_{\{\varphi(t)\}^\perp} (\eta_i) \otimes \proj_{\{\varphi(t)\}^\perp} (\eta_i).
\end{equation*}
This expression for \( \kappa \) is a low-rank expression of which the functions do not form an orthonormal family.
In order to have a unitary family for the low-rank expression of \( \widehat{\chi}_2 \), an Autonne--Takagi decomposition can be performed on \( \kappa \).\footnote{In practice, the Autonne--Takagi decomposition of the 2-particle function \( \kappa \) is expensive.
It is better to perform a QR decomposition of the family \( \{\proj_{\{\varphi(t)\}^\perp}(\eta_i) : i=1,\dots, r\} \) to obtain \( \kappa = QR B (QR)^\top \), and to then perform an Autonne--Takagi decomposition on the matrix \( RBR^\top = W\Sigma W^\top \).
The legs of the low-rank approximation of \( \widehat{\chi}_2 \) are given by the matrix \( QW \), and the diagonal matrix of coefficients is \( \Sigma \).
This way, only small or thin matrices are computed and not large square matrices.}
Therefore, the function \( \widehat{\chi}_2 \) is easily obtained from \( \chi_2 \).

\subsection{Algorithm}

The numerical method is summarized in Algorithm \ref{alg: n arbitrary}.

\begin{algorithm}
    \caption{Time-integration of the approximation \( u \) to the Schrödinger equation, over the time interval \( [0, \Delta t] \).}
    \label{alg: n arbitrary}
    \begin{flushleft}
        \textbf{Input:} \( \Delta t>0 \), \( c_0(0), \dots, c_{2n}(0) \in \mathbb{C}^* \), \( \varphi(0) \in \mathbb{C}^{M^d} \), \( B(0) \in \mathbb{C}^{r\times r} \), \( \eta_1(0), \dots, \eta_r(0) \in \mathbb{C}^{M^d} \), \( \theta_4(0), \dots, \theta_{2n}(0) \in \mathbb{R} \), \( r_{\max} \geq 1 \).

        \textbf{Requires:}
        The matrix \( B \) is symmetric with Frobenius norm equal to one, and the family \( \{\varphi(0), \eta_1(0), \dots, \eta_r(0)\} \) is orthonormal.

        \textbf{Steps:}
        \begin{enumerate}
            \item Integrate the free Schrödinger equation \( \ii \dot{u} = T u \), \( T=-\Delta \), over the time interval \( [0, \Delta t/2] \), using \eqref{eqn: free Schrodinger equations on phi and etas}, with initial conditions  \( \varphi(0) \) and \( \chi_2(0) \).
            Let \( \varphi^{T, -}(\Delta t/2) \) and \( \chi_2^{T, -}(\Delta t/2) \) be the updated functions, which are the only quantities modified during this step.
            Let \( \varphi^{V}(0) := \varphi^{T, -}(\Delta t/2) \) and \( \chi_2^{V}(0) := \chi_2^{T, -}(\Delta t/2) \).
            
            \item Integrate the vector of coefficients \( \mathbf{c} \) over the time interval \( [0, \Delta t/2] \), letting \( \varphi=\varphi^{V}(0)  \) and all excitation functions \( \chi_{2j}=\chi_{2j}^{V}(0)  \), \( 1 \leq j \leq n \), constant with respect to time (defined from \( \chi_2^{V}(0) \) using the tensorization \eqref{eqn: Bogoliubov tensorization}-\eqref{eqn: Bogoliubov tensorization normalized}).
            Then the matrix \( \check{\mathbf{V}} \) is time-independent, and \( \mathbf{c}(\Delta t/2) = e^{-\ii\frac{\Delta t}{2} \check{\mathbf{V}}(0)} \mathbf{c}(0) \).

            \item Integrate the angles \( \theta_{2j} \), \( 2 \leq j \leq n \), according to Lemma \ref{lemma: equations of motion angles} over the time interval \([0, \Delta t/2]\), letting other quantities constant with respect to time.

            \item Integrate the function \( \varphi \) over the time interval \( [0, \Delta t/2] \) with initial condition \( \varphi^{V}(0) \), letting all other quantities constant with respect to time.
            The integration of \( \varphi \) is performed using Lemma \ref{lemma: operator splitting - potential hartree equation with phi0}: \( \varphi^V(\Delta t/2) = e^{-\ii \frac{\Delta t}{2}  (v\ast \abs{\varphi(0)}^2 -\mu_{\varphi(0)})(x)} \varphi^{V}(0, x) \).

            \item Replace \( \chi_2^{V}(0) \) by its approximation \( \widehat{\chi}_2^V(0) \) as detailed in Section \ref{section: preserving bogoliubov orthogonality after integrating varphi}.
            The function \( \widehat{\chi}_2^V(0) \) satisfies the orthogonality \eqref{eqn: bogoliubov orthogonality relation} with \( \varphi^V(\Delta t/2) \).
            
            \item Obtain \( \chi_2^V(\Delta t) \) by integrating the differential equation \( \dot{\chi}_2 = \dot{\chi}_2^{(1)} \) with initial condition \( \chi_2^V(0)=\widehat{\chi}_2^V(0) \) over the time interval \( [0, \Delta t] \), letting all other quantities constant with respect to time (including the higher-order excitation functions \( \chi_4^V(0), \dots, \chi_{2n}^V(0) \) obtained using the tensorization \eqref{eqn: Bogoliubov tensorization}-\eqref{eqn: Bogoliubov tensorization normalized}).
            The expression of \( \dot{\chi}_2^{(1)} \) is given in \eqref{eqn: dotchi21 and dotchi22}.
            The integration is performed using the complex symmetric BUG(\( 2r \)) low-rank integrator, and the rank \( r \) is updated to \( \tilde{r} \) according to criterion \eqref{eqn: truncation criterion BUG}.

            \item Repeat step 4. over the time interval \( [\Delta t/2, \Delta t] \) with initial condition \( \varphi^V(\Delta t/2) \) to obtain \( \varphi^V(\Delta t) \).
            Let \( \varphi^{T, +}(\Delta t/2) := \varphi^V(\Delta t) \).

            \item Repeat step 5. to replace \( \chi_2^V(\Delta t) \) by its approximation \( \widehat{\chi}_2^V(\Delta t) \) that satisfies the orthogonality \eqref{eqn: bogoliubov orthogonality relation} with \( \varphi^V(\Delta t) \).
            Let \( \chi_2^{T, +}(\Delta t/2) := \widehat{\chi}_2^V(\Delta t) \).        
            
            \item Repeat step 3. over the time interval \( [\Delta t/2, \Delta t] \) to obtain \( \theta_4(\Delta t), \dots, \theta_{2n}(\Delta t) \).

            \item Repeat step 2. over the time interval \( [\Delta t/2, \Delta t] \) to obtain \( \mathbf{c}(\Delta t) \).
            
            \item Repeat step 1. over the time interval \( [\Delta t/2, \Delta t] \) with initial conditions \( \varphi^{T, +}(\Delta t/2) \) and \( \chi_2^{T, +}(\Delta t/2) \).
            Let \( \varphi(\Delta t) \) and \( \chi_2(\Delta t) \) be the updated functions.
        \end{enumerate}

        \textbf{Output:} \( c_{2k}(\Delta t)\in \mathbb{C} \) for \( 0 \leq k \leq n \), \( \varphi(\Delta t) \in \mathbb{C}^{M^d} \), \( B(\Delta t) \in \mathbb{C}^{\tilde{r}\times \tilde{r}} \), \( \eta_k(\Delta t) \in \mathbb{C}^{M^d} \) for \( 1 \leq k \leq \tilde{r} \), \( \theta_{2k}(\Delta t) \in \mathbb{R} \) for \( 2\leq k \leq n \).
        The family \( \{\varphi(\Delta t), \eta_1(\Delta t), \dots, \eta_{\tilde{r}}(\Delta t) \} \) is orthonormal, and \( B(\Delta t) \) is a symmetric matrix with Frobenius norm one.
    \end{flushleft}
\end{algorithm}

\section{Numerical experiments}
\label{section: numerical experiments}

The Julia code used for the numerical experiments is freely available at the following address: \url{https://gitlab.com/ylehenaff/numerical-bogoliubov-approximation-bosons}.

In the numerical experiments presented, the domain \( \Omega = [-15, 15]^d \) is discretized uniformly along each dimension using \( M = 2^8 = 256 \) points.
The simulations are done on the time interval \( [0, t_{\text{final}}] \) with a time step \( \Delta t \).
We use a rank-adaptive BUG integrator detailed in \cites{cerutiRankadaptiveRobustIntegrator2022,cerutiRobustSecondorderLowrank2024} with the midpoint method used for the time integration of the substeps, and use \( \vartheta \) as the \emph{relative} rank-truncation tolerance: for \( S \) a diagonal matrix of singular values on which we want to perform the rank truncation, the updated rank \( \tilde{r} \) is set as follows:
\begin{equation}
    \label{eqn: truncation criterion BUG}
    \tilde{r} = \min \left(r_{\max}, \min \left\{ r : \sqrt{\sum_{\ell\geq r+1} S_{\ell, \ell}^2} \leq \vartheta \norm{S}{2} \right\} \right),
\end{equation}
where \( r_{\max} \geq 1 \) is some given maximal rank.
We choose for all our numerical experiments \( \vartheta = 10^{-6} \).
A confining potential is considered, changing the Hamiltonian \( H \) from \eqref{eqn: full schrodinger hamiltonian} to
\begin{equation*}
    H\psi(x) = \sum_{j=1}^N \left(-\Delta_{x_j} + U(x_j)\right)\psi(x) + \frac{1}{N-1} \sum_{1\leq j < k \leq N} v(x_j - x_k) \psi(x),
\end{equation*}
where \( U(x) \) is the confining potential.
The equations of motion and all results presented in the previous sections hold when replacing the kinetic operator \( -\Delta_{x_j} \) by the one-particle operator \( T_{x_j} = -\Delta_{x_j} + U(x_j) \).
The ansatz \eqref{eqn: final ansatz u} holds in the presence of a confining potential \cites{mitrouskasBogoliubovCorrectionsTrace2019,lewinFluctuationsHartreeStates2015}.

To assess the efficiency of the numerical method we will plot the relative evolution of the energy \( E_N(t) = \braket{u}{Hu}(t) \) scaled by \( N \), i.e. \( |E_N(t)-E_N(0)|/N \), as well as the evolution of the coefficients \( \abs{c_{2j}}^2 \) and the evolution of the smallest eigenvalue of the matrix \( B \) that appears in the low-rank approximation of \( \chi_2 \).

The numerical experiments we present below are meant to investigate the numerical method itself, and not its relevance for precise physics applications for the reasons laid out in Remark \ref{remark: scope and applicability of this work}.
We will present two different experiments, one using a contact interaction potential as in \cite{alonMulticonfigurationalTimedependentHartree2008}, and the second one using a soft Coulomb potential.
For both of them, the parameter \( n=6 \) is used.
The initial condition is obtained as follows: an imaginary-time relaxation is performed, during which the coefficients \( c_{2j} \) are set constant equal to \( c_{2j} = e^{-j/2} \), and the BUG integrator is allowed to increase the rank of the low-rank approximation of \( \chi_2 \) up to \( r_{\max} \).
The confining potential for the relaxation is \( U_0 \).
The coefficients \( c_{2j} \) are chosen so that the result of this pseudo-relaxation of \( u \) is a fluctuation around the Hartree state \( \varphi^{\otimes N} \) that incorporates some interactions.
For the real-time evolution, the confining potential is switched to \( U \), and all quantities in \eqref{eqn: final ansatz u} are evolved according to Algorithm \ref{alg: n arbitrary}.
The BUG(\( 2r \)) integrator is used with \( r_{\max} = 15 \), and it has to be noted that even though this integrator has been shown to be of order one, it has also been observed in practice to exhibit second-order error in some situations \cite{cerutiRankadaptiveRobustIntegrator2022}.

\subsection{Double well with contact interaction}

We perform the same one-dimensional experiment as in \cite{alonMulticonfigurationalTimedependentHartree2008}, where the interaction potential is a contact interaction given by \( v(x-y) = 2\lambda \delta(x-y) \) with \( \lambda = 9.99 \).
After space discretization, this yields
\begin{equation*}
    v(x) = \begin{cases}
        \frac{2\lambda}{\Delta x} & \text{ if } x \in (-\Delta x/2, \Delta x/2] \\
        0 & \text{else.}
    \end{cases}\quad\text{ with } \lambda = 9.99.
\end{equation*}
where \( \Delta x \) is the step of the uniform grid described in Section \ref{section: time and space discretizations}.

The confining potential for the imaginary-time relaxation is \( U_0(x) = x^2/\sigma^2 + 16 e^{-x^2/(2\sigma^2)} \), \( \sigma=2.6 \).
At time \( t=0 \) the barrier between the two wells is halved, and the confining potential is shifted to the left to get \( U(x) = \frac{(x+2)^2}{\sigma^2} + 8 e^{-(x+2)^2/(2\sigma^2)} \), \( \sigma=2.6 \).
From time \( t=0 \) onwards, the real-time evolution is performed for different values of \( N \in \{10^2, 10^6, 10^{10}\} \), using time steps \( \Delta t \in \{10^{-2}, 10^{-3}\} \).

The numerical experiments were run in parallel on a regular office computer, and the computational time was around 2500 minutes for the most demanding computations.
It has been checked that the computational time is indeed independent of \( N \).
The results are shown in Figure \ref{fig: alon - tfinal 20 - xmax=15 - N=1e2-6-10 d=1 r=15 n=6}.
The top-left plots give the relative evolution of the scaled energy, the center top plots give the evolution of the smallest eigenvalue of the matrix \( B \), and the other plots the evolution of the magnitudes \( \abs{c_{2j}}^2 \) for \( j=0, \dots, n \).

Let us discuss the conservation of energy.
We recall that the main non energy-conserving part is due to imposing \( \varphi \) in the variational principle as the solution of the Hartree equation \eqref{eqn: hartree equation}, which is an equation obtained in the limit \( N\to\infty \) \cite{benedikterEffectiveEvolutionEquations2016}*{p.8}.
For \( N\to\infty \), we expect that imposing this evolution for \( \varphi \) has no effect onto the energy conservation because of the validity of the argument leading to the Hartree equation -- in other words, if the equation of motion for \( \varphi \) were to be obtained via a variational argument, we would expect to recover the Hartree equation in the large \( N \) limit.
Hence we would expect no energy loss when imposing the Hartree evolution for \( \varphi \) in the case \( N\to\infty \).
However, for finite \( N \), we expect a loss in the energy conservation that depends on \( N \) and decreases as \( N \) increases, let us call it the \( N \)-error.
Moreover, the loss of energy conservation may also be due to the time step \( \Delta t \), leading to a competition between these two sources of error.
We can see this phenomenon by comparing the energy conservation for \( \Delta t\in \{10^{-2}, 10^{-3}\} \) in Figure \ref{fig: alon - tfinal 20 - xmax=15 - N=1e2-6-10 d=1 r=15 n=6}: for small \( N \), the loss of the energy conservation is mainly due to the \( N \)-error, and the error coming from the integration scheme (depending on \( \Delta t \)) is negligible compared to it, for both large and small values of \( \Delta t \).
This leads to an energy conservation for \( N=10^2 \) that is qualitatively independent of \( \Delta t\in \{10^{-2}, 10^{-3}\} \).
For intermediate values of \( N \), we expect the \( \Delta t \)-error to be dominant for large \( \Delta t \), and negligible compared to the \( N \)-error for small \( \Delta t \).
This also explains the behavior observed in Figure \ref{fig: alon - tfinal 20 - xmax=15 - N=1e2-6-10 d=1 r=15 n=6} for \( N=10^6 \), where decreasing \( \Delta t \) does improve the energy conservation but not up to second-order as expected.
For large values of \( N \), we expect the \( N \)-error to be negligible and thus the main error is the \( \Delta t \)-error.
In this situation we can observe an energy conservation that is second-order in time.
The competition between the \( \Delta t \)-error and the \( N \)-error also explains why the numerical results for \( \Delta t = 10^{-2} \) yield results that are quantitatively very similar for \( N = 10^6 \) and \( N=10^{10} \) (the \( \Delta t \)-error dominates), but are different from the results with \( N=10^2 \) (the \( N \)-error dominates).

The smallest eigenvalue of \( B \) is relatively robust to the value of \( N \), and is slightly smaller for smaller values of \( \Delta t \).
This behavior is also expected: the function \( \chi_2 \) encodes interactions in the system, but the size of the interaction is parameterized by \( \lambda \) which is independent of \( N \).
Thus, the low-rank approximation of \( \chi_2 \) is expected to be qualitatively similar for all values of \( N \), and the stepsize \( \Delta t \) simply allows for more precise computations of \( \chi_2 \).

Finally, regarding the coefficients \( \abs{c_{2j}}^2 \), we can observe that the relative order \( \abs{c_2} > \abs{c_4} > \dots > \abs{c_{2n}} \) imposed at time \( t=0 \) is approximately preserved.
The values of \( N \) have little influence on the magnitudes \( \abs{c_{2j}}^2 \).
The overall increase of all magnitudes \( \abs{c_2}^2, \dots, \abs{c_{2n}}^2 \) accompanied by a sharp decrease of \( \abs{c_0}^2 \) shows in particular that the Hartree product state is a very poor approximation of the solution, and that the ansatz \eqref{eqn: final ansatz u} allows us to recover a better approximation to the solution.

\setlength{\widthfigure}{0.95\linewidth} 

\renewcommand{\currentImgsLargeDt}{dt=0.01_tfinal=20.0_N=1e2-1e6-1e10_M=256_dimOmega=1_xmin=-15.0_xmax=15.0_n=6_r=15_U=alon2008confining_V=alon2008interaction_order-BUG=1}
\renewcommand{\currentImgsSmallDt}{dt=0.001_tfinal=20.0_N=1e2-1e6-1e10_M=256_dimOmega=1_xmin=-15.0_xmax=15.0_n=6_r=15_U=alon2008confining_V=alon2008interaction_order-BUG=1}
\begin{figure}
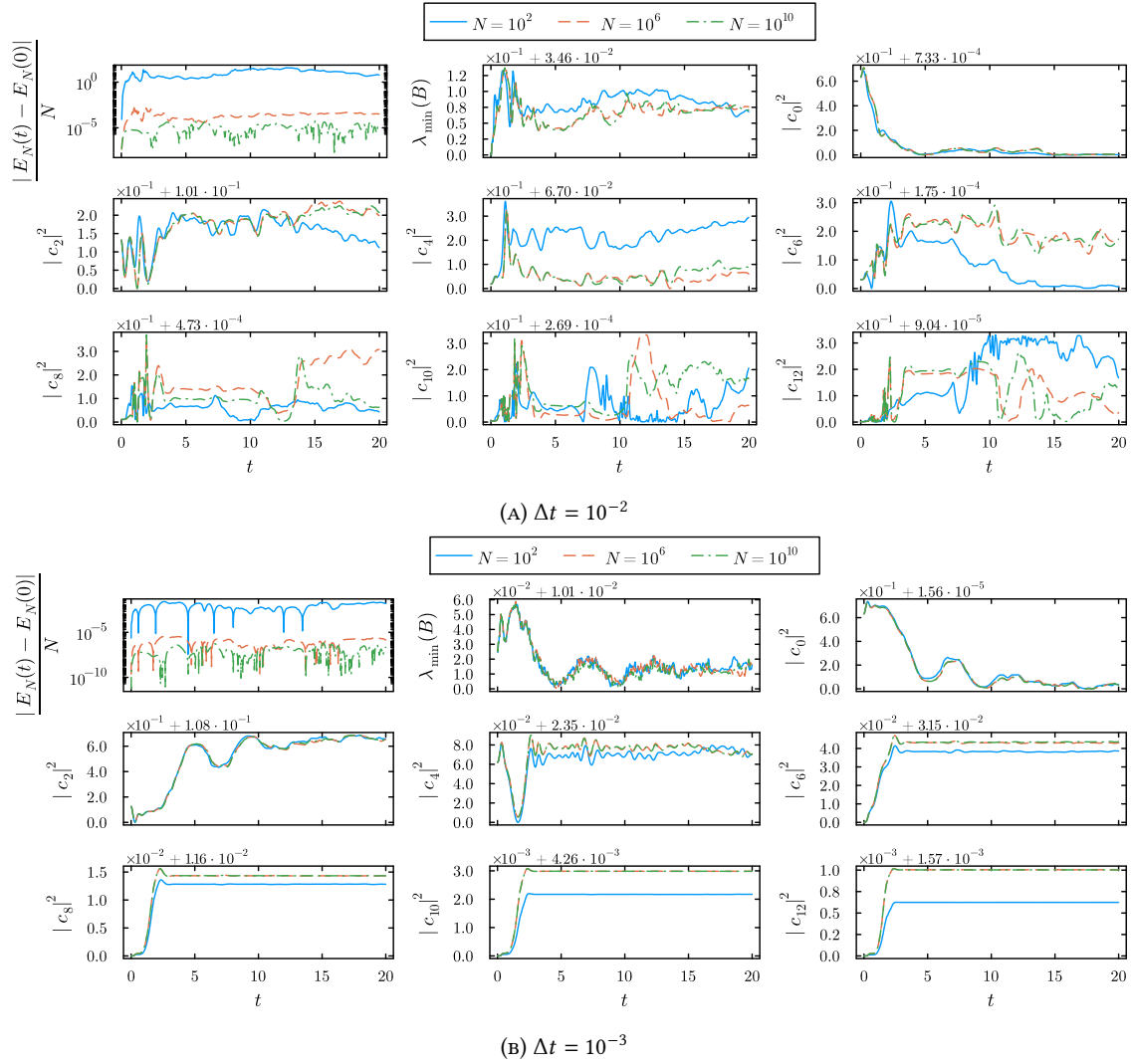

    \centering
    \begin{subfigure}[b]{\textwidth}
        \centering
        \includegraphics[width=\widthfigure]{\folderImages/\currentImgsLargeDt/energy_and_abs2c.pdf}
        \caption{\( \Delta t = 10^{-2} \)}
    \end{subfigure}
    \begin{subfigure}[b]{\textwidth}
        \centering
        \includegraphics[width=\widthfigure]{\folderImages/\currentImgsSmallDt/energy_and_abs2c.pdf}
        \caption{\( \Delta t = 10^{-3} \)}
    \end{subfigure}
    \caption{Double well confinement, contact interaction, \( d=1 \), \( n=6 \). We obtained \( r_{\text{final}}=r_{\max}=15 \).
    \textbf{Left to right, top to bottom}: relative energy conservation, smallest eigenvalue of \( B \), magnitudes \( \abs{c_{2j}}^2 \) for \( j=0, \dots, n \). The scale of the \( y \)-axis for the coefficients \( \abs{c_{2j}}^2 \) is indicated above each plot.}
    \label{fig: alon - tfinal 20 - xmax=15 - N=1e2-6-10 d=1 r=15 n=6}
\end{figure}

\subsection{Harmonic potential with Coulomb interaction}

For the second experiment, the interactions are of soft Coulomb-type: \( v(x) = 1/\sqrt{\abs{x}^2+\varepsilon} \), \( \varepsilon=10^{-3} \).
For the confining potentials, we choose a double-well \( U_0(x) = \abs{x}^2 /2 + 10 e^{-\frac{1}{2}\abs{x}^2} \) for the imaginary-time relaxation and a harmonic confining potential for the real-time evolution \( U(x) = \abs{x}^2/2 \).
The results are shown in Figure \ref{fig: coulomb potential - tfinal 20 - xmax=15 - N=1e2-6-10 d=1 r=15 n=6}.

Similarly to the previous experiment, there is a competition in the conservation of energy between the \( N \)-error and the \( \Delta t \)-error, and when the \( \Delta t \)-error dominates the method is seemingly of second-order in time.
The relative order of coefficients is also approximately preserved, and shows a faster decrease in the magnitudes \( \abs{c_{2j}}^2 \) with \( j \) compared to the previous experiment.
In contrast to the previous experiment, \( \abs{c_0}^2 \) does not decrease rapidly, meaning that the Hartree product state remains in this situation a reasonable approximation to the solution.
The overall increase of the magnitude \( \abs{c_2}^2 \) is also rather slow compared to the previous experiment.
Finally, the smallest eigenvalue of the matrix \( B \) is again robust to the value of \( N \).

\renewcommand{\currentImgsLargeDt}{dt=0.01_tfinal=20.0_N=1e2-1e6-1e10_M=256_dimOmega=1_xmin=-15.0_xmax=15.0_n=6_r=15_U=harmonic_V=coulomb_order-BUG=1}
\renewcommand{\currentImgsSmallDt}{dt=0.001_tfinal=20.0_N=1e2-1e6-1e10_M=256_dimOmega=1_xmin=-15.0_xmax=15.0_n=6_r=15_U=harmonic_V=coulomb_order-BUG=1}
\begin{figure}
    \centering
    \begin{subfigure}[b]{\textwidth}
        \centering
        \includegraphics[width=\widthfigure]{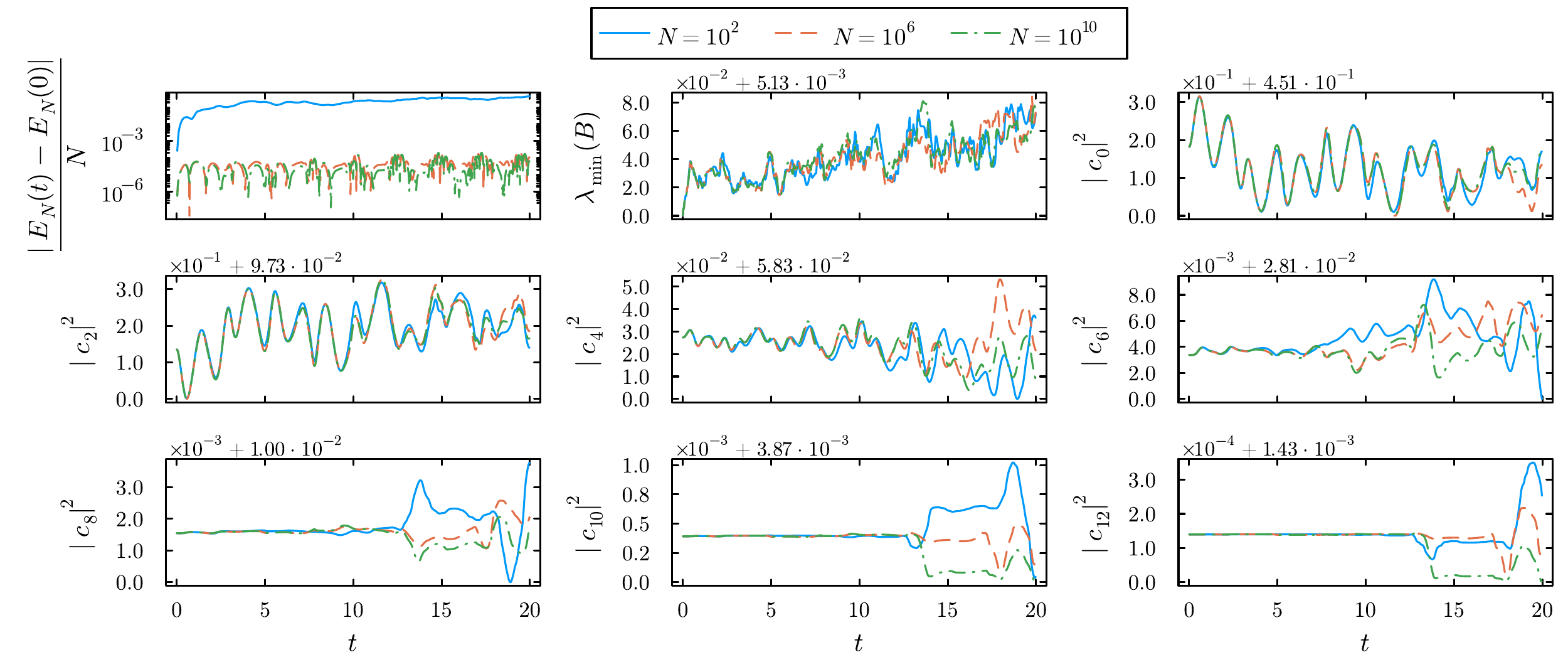}
        \caption{\( \Delta t = 10^{-2} \)}
    \end{subfigure}
    \begin{subfigure}[b]{\textwidth}
        \centering
        \includegraphics[width=\widthfigure]{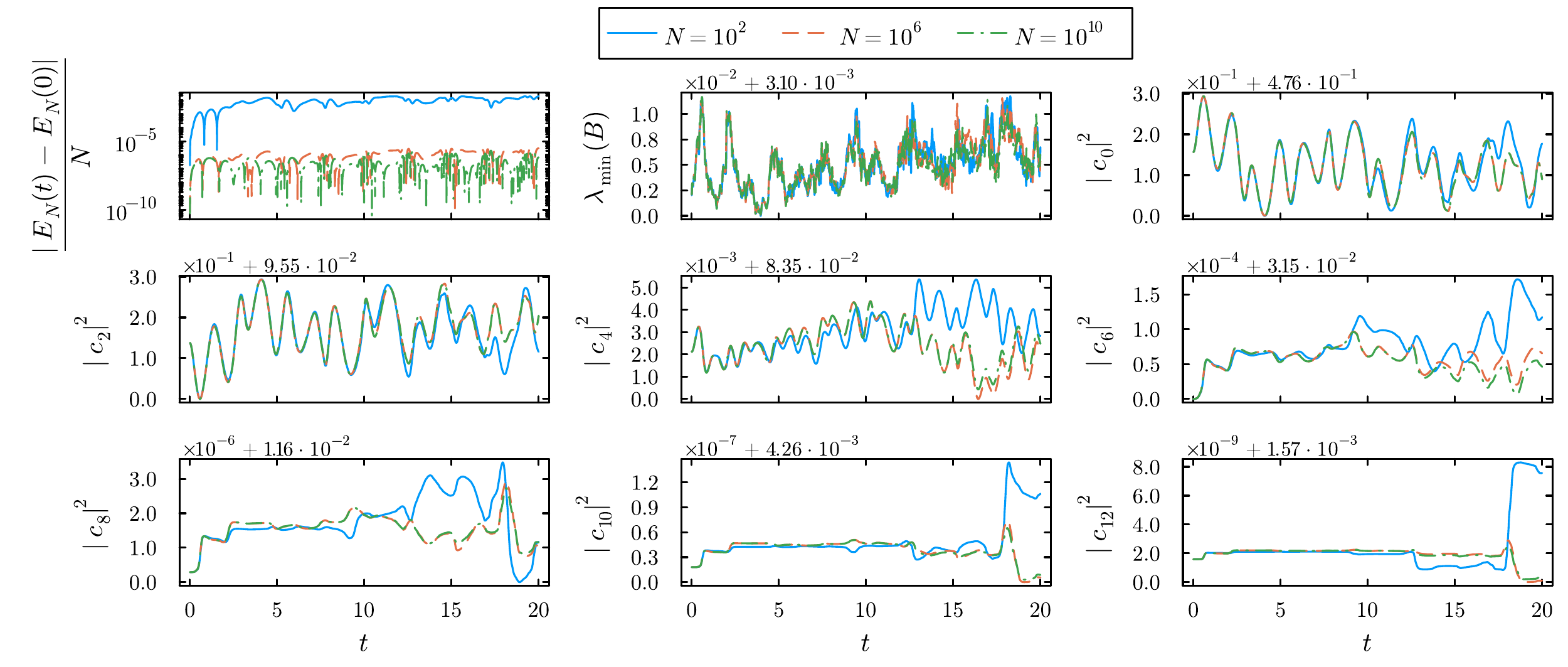}
        \caption{\( \Delta t = 10^{-3} \)}
    \end{subfigure}
    \caption{Harmonic confinement, Coulomb interaction, \( d=1 \), \( n=6 \). We obtained \( r_{\text{final}}=r_{\max}=15 \).
    \textbf{Left to right, top to bottom}: relative energy conservation, smallest eigenvalue of \( B \), magnitudes \( \abs{c_{2j}}^2 \) for \( j=0, \dots, n \). The scale of the \( y \)-axis for the coefficients \( \abs{c_{2j}}^2 \) is indicated above each plot.}
    \label{fig: coulomb potential - tfinal 20 - xmax=15 - N=1e2-6-10 d=1 r=15 n=6}
\end{figure}

\section{Conclusion}
\label{section: conclusion}

In this work we have presented a numerical method for the approximation of the solution to the Schrödinger equation with a two-body interaction potential, in the bosonic case and in the mean-field scaling regime.
The approximation is based on the Bogoliubov theory, and it is particularly suited to study excitations of the solution around the Hartree product state.

The main feature of the numerical ansatz is that it approximates the wave function via the condensate function \( \varphi \) and only a few excitation functions \( \chi_{2k}, k\geq 1 \).
As is shown here, this implies that the full condensate never needs to be explicitly formed, and that one can obtain a numerical method whose computational effort is uniformly bounded with respect to the number of particles considered.
The time evolution of the condensate function \( \varphi \) is imposed by the Hartree equation, and only the coefficients \( c_{2k} \) and the excitation function \( \chi_{2k} \), \( k\geq 1 \), need to be approximated.

After a splitting of the Hamiltonian operator \( H=-\Delta+V \), the Laplacian part can be solved exactly and only the potential part has to be approximated numerically.
This is based on the Dirac--Frenkel time-dependent variational principle, which in this case yields equations of motion that are smooth uniformly in \( N \). The time discretization uses further splittings of the equations of motion.
The excitation functions \( \chi_{2k} \), \( k\geq 1 \), depend on more than one variable, and they are represented using a rank-\( r \) approximation, where \( r \) is updated dynamically during the time evolution of the solution using a rank-adaptive BUG integrator.

The resulting algorithm has a computational cost independent of the number of particles, which sets it apart from the existing literature on numerical methods for many-body quantum dynamics.
In numerical experiments the algorithm shows the accuracy of a second-order scheme and satisfying conservation properties.

\section*{Acknowledgments}

All three authors were funded by the Deutsche Forschungsgemeinschaft (DFG, German Research Foundation) -- TRR 352 -- Project-ID 470903074.
YLH would also like to thank Georgios Vretinaris for valuable discussions regarding the implementation of the algorithm in Julia, in particular for the BUG integrator.

\section*{AI disclosure}

The authors used the AI assistant Vibe (Mistral; GLM model by Z.ai) for proofreading and copyediting support.
No AI-based tool has been used to implement the algorithm and the numerical experiments apart, possibly, from the external libraries for which it is difficult to check.

\appendix

\section{Technical lemmas}
\label{appendix: technical lemma}

We need some technical lemmas throughout this work, and we collect some of them in this section.
We do not claim the novelty of these results, but we did not find a reference for them.

\begin{lemma}[Associativity of symmetric tensor product]
    \label{lemma: associativity symmetric tensor product}
    Let \( i, j, k \in \mathbb{N} \), and \( f \in L^2_s(\Omega^i) \), \( g \in L^2_s(\Omega^j) \), \( h \in L^2_s(\Omega^k) \).
    Then,
    \begin{equation*}
        (f \otimes_s g) \otimes_s h = f \otimes_s (g \otimes_s h).
    \end{equation*}
\end{lemma}

\begin{proof}
    We have
    \begin{align*}
        ((f \otimes_s g) \otimes_s h)(x)
        &= \frac{1}{\mathsf{K}^{k+i+j}_k} \sum_{\tau\in S_{i+j+k}} (f\otimes_s g)(x_{\tau(\intinterval{k+1}{i+j+k})}) h(x_{\tau(\intinterval{1}{k})}) \\
        &= \frac{1}{\mathsf{K}^{k+i+j}_k \mathsf{K}^{i+j}_j} \sum_{\tau\in S_{i+j+k}} \sum_{\sigma \in S_{\tau(\intinterval{k+1}{i+j+k})}} f(x_{\sigma(\intinterval{1}{i})}) g(x_{\sigma(\intinterval{i+1}{i+j})}) h(x_{\tau(\intinterval{1}{k})}) \\
        &= \frac{1}{\mathsf{K}^{k+i+j}_k \mathsf{K}^{i+j}_j} \sum_{\tau_1 \in S_{i+j+k}^k} \sum_{\tau_2 \in S_{i+j+k \setminus \tau_1}} \sum_{\sigma \in S_{\tau_2}} f(x_{\sigma(\intinterval{1}{i})}) g(x_{\sigma(\intinterval{i+1}{i+j})}) h(x_{\tau_1}) \\
        &= \frac{(i+j)!}{\mathsf{K}^{k+i+j}_k \mathsf{K}^{i+j}_j} \sum_{\tau_1 \in S_{i+j+k}^k} \sum_{\sigma \in S_{i+j+k \setminus \tau_1}} f(x_{\sigma(\intinterval{1}{i})}) g(x_{\sigma(\intinterval{i+1}{i+j})}) h(x_{\tau_1}) \\
        &= \frac{(i+j)!}{\mathsf{K}^{k+i+j}_k \mathsf{K}^{i+j}_j} \sum_{\tau \in S_{i+j+k}} f(x_{\tau(\intinterval{k+1}{k+i})}) g(x_{\tau(\intinterval{k+i+1}{k+i+j})}) h(x_{\tau(\intinterval{1}{k})}).
    \end{align*}
    Similarly, 
    \begin{equation*}
        (f\otimes_s (g\otimes_s h))(x)
        = \frac{(j+k)!}{\mathsf{K}_i^{i+j+k} \mathsf{K}^{j+k}_j} \sum_{\tau \in S_{i+j+k}} f(x_{\tau(\intinterval{k+1}{k+i})}) g(x_{\tau(\intinterval{k+i+1}{k+i+j})}) h(x_{\tau(\intinterval{1}{k})}).
    \end{equation*}
    Thus,
    \begin{align*}
        ((f \otimes_s g) \otimes_s h)(x)
        &= \frac{(i+j)!}{\mathsf{K}^{k+i+j}_k \mathsf{K}^{i+j}_j} \frac{\mathsf{K}_i^{i+j+k} \mathsf{K}^{j+k}_j}{(j+k)!} (f\otimes_s (g\otimes_s h))(x) \\
        &= \frac{(i+j)!}{\sqrt{(i+j+k)!(i+j)!k! (i+j)!i!j!}} \frac{\sqrt{(i+j+k)!(j+k)!i! (j+k)!j!k!}}{(j+k)!} (f\otimes_s (g\otimes_s h))(x) \\
        &= (f\otimes_s (g\otimes_s h))(x).
    \end{align*}
\end{proof}

\begin{lemma}
    \label{lemma: braket hf against v rewritten as h against fv}
    Let \( i,j \in \mathbb{N} \), and \( f\in L^2(\mathbb{R}^i) \), \( g\in L^2(\mathbb{R}^{j}) \), \( h\in L^2_s(\mathbb{R}^{i+j}) \). Then,
    \begin{equation*}
        \braket{f\otimes_s g}{h} = \sqrt{\binom{i+j}{j}} \braket{f \otimes g}{h}.
    \end{equation*}
\end{lemma}

\begin{proof}
    By symmetry of \( h \),
    \begin{align*}
        \braket{f\otimes_s g}{h}
        &= \frac{1}{\mathsf{K}^{i+j}_j} \sum_{\sigma\in S_{i+j}} \braket{f(x_{\sigma(\intinterval{1}{i})}) g(x_{\sigma(\intinterval{i+1}{i+j})})}{h(x)} 
            = \frac{1}{\mathsf{K}^{i+j}_j} \sum_{\sigma\in S_{i+j}} \braket{f \otimes g}{h}
            = \sqrt{\binom{i+j}{j}} \braket{f \otimes g}{h}.
    \end{align*}
\end{proof}

\begin{lemma}
    \label{lemma: symmetric tensor product of a function and itself}
    Let \( f \in L^2(\Omega) \), \( p \in \mathbb{N} \).
    Then
    \begin{equation*}
        (f^{\otimes p} \otimes_s f) = \sqrt{p+1} f^{\otimes p+1}.
    \end{equation*}
\end{lemma}

\begin{proof}
    We have
    \begin{equation*}
        f^{\otimes p} \otimes_s f
        = \frac{1}{\mathsf{K}^{p+1}_1} \sum_{\sigma\in S_{p+1}} f^{\otimes p} (x_{\sigma(\intinterval{1}{p})}) f(x_{\sigma(p+1)})
        = \frac{(p+1)}{\mathsf{K}^{p+1}_1} f^{\otimes p+1} = \sqrt{p+1} f^{\otimes p+1}.
    \end{equation*}
\end{proof}

\begin{lemma}
    \label{lemma: Derivative of k-fold tensor product}
    Let \( \varphi=\varphi(t,\cdot) \in L^2(\Omega) \) a differentiable function with respect to time \( t \), and \( k \geq 1 \).
    Then 
    \begin{equation}
        \frac{\partial}{\partial t} \left( \varphi^{\otimes k} \right) = \sqrt{k} \varphi^{\otimes k-1} \otimes_s \dot{\varphi}.
    \end{equation}
\end{lemma}

\begin{proof}
    On the one hand, we have
    \begin{align*}
        \frac{\partial}{\partial t} \left( \varphi^{\otimes k} \right) 
        = \dot{\varphi}\otimes \varphi \otimes \cdots \otimes \varphi + \cdots + \varphi\otimes \cdots \varphi \otimes \dot{\varphi}
        = \sum_{j=1}^k \dot{\varphi}(x_j) \varphi^{\otimes k-1} (x\setminus x_j).
    \end{align*}
    On the other hand, 
    \begin{equation*}
        \varphi^{\otimes k-1} \otimes_s \dot{\varphi}
        = \frac{1}{\mathsf{K}^k_1} \sum_{\sigma\in S_k} \varphi^{\otimes k-1} (x_{\sigma(\intinterval{1}{k-1})}) \dot{\varphi}(x_{\sigma(k)})
        = \frac{1}{\sqrt{k}} \sum_{j=1}^k \dot{\varphi}(x_j) \varphi^{\otimes k-1} (x\setminus x_j),
    \end{equation*}
    which yields the result.
\end{proof}

\section{Components of the tridiagonal Hermitian matrix \texorpdfstring{\( \mathbf{H} \)}{H}}
\label{appendix: components bfH}

Write \( \mathbf{H} = \mathbf{T} + \mathbf{V} \), where \( \mathbf{T} \) is the matrix associated to the kinetic operator and \( \mathbf{V} \) the matrix associated to the potential operator.
By the Bogoliubov orthogonality, the matrix \( \mathbf{T} \) is diagonal and by Lemmas \ref{lemma: reducing number of effective kinetic operators} and \ref{lemma: braket hf against v rewritten as h against fv}:
\begin{align*}
    \mathbf{T}_{j,j} 
    &= (N-2j) \sqrt{\binom{N}{2j}} \braket{\varphi^{\otimes N-2j} \otimes \chi_{2j}}{-\Delta_{x_1} \varphi^{\otimes N-2j} \otimes_s \chi_{2j}} + 2j \sqrt{\binom{N}{2j}} \braket{\varphi^{\otimes N-2j} \otimes \chi_{2j}}{-\Delta_{x_{N-2j+1}} \varphi^{\otimes N-2j} \otimes_s \chi_{2j}} \\
    &= (N-2j) \braket{\varphi}{-\Delta \varphi} + 2j \braket{\chi_{2j}}{-\Delta_{x_{N-2j+1}}\chi_{2j}}.
\end{align*}
By Lemmas \ref{lemma: reducing number of effective potential operators} and \ref{lemma: braket hf against v rewritten as h against fv}, we have
\begin{equation*}
    \mathbf{V}_{j,k} 
    = \begin{aligned}[t]
        &\frac{(N-2j)(N-2j-1)}{2(N-1)} \sqrt{\binom{N}{2j}} \braket{\varphi^{\otimes N-2j} \otimes \chi_{2j}}{v(x_1-x_2) \varphi^{\otimes N-2k} \otimes_s \chi_{2k}} \\
        &+\frac{2j(N-2j)}{N-1} \sqrt{\binom{N}{2j}} \braket{\varphi^{\otimes N-2j} \otimes \chi_{2j}}{v(x_1-x_{N-2j+1}) \varphi^{\otimes N-2k} \otimes_s \chi_{2k}} \\
        &+\frac{2j(2j-1)}{2(N-1)} \sqrt{\binom{N}{2j}} \braket{\varphi^{\otimes N-2j} \otimes \chi_{2j}}{v(x_{N-2j+1}-x_{N-2j+2}) \varphi^{\otimes N-2k} \otimes_s \chi_{2k}}.
    \end{aligned}
\end{equation*}

We compute the necessary brakets: first, when \( k=j \),
\begin{multline*}
    \braket{\varphi^{\otimes N-2j} \otimes \chi_{2j}}{v(x_1-x_2) \varphi^{\otimes N-2j} \otimes_s \chi_{2j}} \\
    \begin{aligned}
        &= \frac{1}{\mathsf{K}^N_{2j}} \sum_{\substack{\tau\in S_N\\\tau(\intinterval{1}{2j})=_s \intinterval{N-2j+1}{N}}} \braket{\varphi^{\otimes N-2j} \otimes \chi_{2j}}{v(x_1-x_2) \varphi^{\otimes N-2j}(x_{\tau(\intinterval{2j+1}{N})}) \chi_{2j}(x_{\tau(\intinterval{1}{2j})})} \\
        &= \frac{(2j)!(N-2j)!}{\mathsf{K}^N_{2j}} \braket{\varphi^2}{v(x_1-x_2) \varphi^2}
            = \frac{1}{\sqrt{\binom{N}{2j}}} \braket{\varphi^2}{v(x_1-x_2) \varphi^2},
    \end{aligned}
\end{multline*}
\begin{multline*}
    \braket{\varphi^{\otimes N-2j} \otimes \chi_{2j}}{v(x_1-x_{N-2j+1}) \varphi^{\otimes N-2j} \otimes_s \chi_{2j}} \\
    = \frac{1}{\sqrt{\binom{N}{2j}}} \Big( \braket{\chi_{2j} \otimes \varphi}{v(x_1-x_{2j+1})} \varphi \otimes \chi_{2j} + \braket{\chi_{2j} \otimes \varphi}{v(x_1 - x_{2j+1}) \chi_{2j} \otimes \varphi} \Big),
\end{multline*}
and
\begin{align*}
    &\braket{\varphi^{\otimes N-2j} \otimes \chi_{2j}}{v(x_{N-2j+1} - x_{N-2j+2}) \varphi^{\otimes N-2j} \otimes_s \chi_{2j}}
    = \frac{1}{\sqrt{\binom{N}{2j}}} \braket{\chi_{2j}}{v(x_1-x_2) \chi_{2j}}.
\end{align*}
Thus, 
\begin{equation*}
    \mathbf{V}_{j,j} 
    = \begin{aligned}[t]
        &\frac{(N-2j)(N-2j-1)}{2(N-1)} \braket{\varphi^2}{v(x_1-x_2) \varphi^2} \\
        &+\frac{2j(N-2j)}{N-1} \Big( \braket{\chi_{2j} \otimes \varphi}{v(x_1-x_{2j+1})} \varphi \otimes \chi_{2j} + \braket{\chi_{2j} \otimes \varphi}{v(x_1 - x_{2j+1}) \chi_{2j} \otimes \varphi} \Big) \\
        &+\frac{2j(2j-1)}{2(N-1)} \braket{\chi_{2j}}{v(x_1-x_2) \chi_{2j}}.
    \end{aligned}
\end{equation*}

For the \( \mathbf{V}_{j, j+1} \) term, we have
\begin{equation*}
    \mathbf{V}_{j,j+1} 
    = \begin{aligned}[t]
        &\frac{(N-2j)(N-2j-1)}{2(N-1)} \sqrt{\binom{N}{2j}} \braket{\varphi^{\otimes N-2j} \otimes \chi_{2j}}{v(x_1-x_2) \varphi^{\otimes N-2j-2} \otimes_s \chi_{2j+2}} \\
        &+\frac{2j(N-2j)}{N-1} \sqrt{\binom{N}{2j}} \braket{\varphi^{\otimes N-2j} \otimes \chi_{2j}}{v(x_1-x_{N-2j+1}) \varphi^{\otimes N-2j-2} \otimes_s \chi_{2j+2}} \\
        &+\frac{2j(2j-1)}{2(N-1)} \sqrt{\binom{N}{2j}} \braket{\varphi^{\otimes N-2j} \otimes \chi_{2j}}{v(x_{N-2j+1}-x_{N-2j+2}) \varphi^{\otimes N-2j-2} \otimes_s \chi_{2j+2}}.
    \end{aligned}
\end{equation*}
The Bogoliubov orthogonality implies that the last two brakets vanish.
The first one is given by
\begin{equation*} 
    \braket{\varphi^{\otimes N-2j} \otimes \chi_{2j}}{v(x_1-x_2) \varphi^{\otimes N-2j-2} \otimes_s \chi_{2j+2}}
    = \frac{1}{\sqrt{\binom{N}{2j+2}}} \braket{\varphi^{\otimes 2} \otimes \chi_{2j}}{v(x_1-x_2) \chi_{2j+2}}.
\end{equation*}
Finally,
\begin{align*}
    \mathbf{V}_{j,j+1} 
    &= \frac{\sqrt{(N-2j)(N-2j-1)}}{2(N-1)} \sqrt{(2j+2)(2j+1)} \braket{\varphi^{\otimes 2} \otimes \chi_{2j}}{v(x_1-x_2) \chi_{2j+2}}
\end{align*}
and thus
\begin{align*}
    \mathbf{H}_{j, j} &= \mathbf{V}_{j,j} + (N-2j) \braket{\varphi}{-\Delta \varphi} + 2j \braket{\chi_{2j}}{-\Delta_{x_{N-2j+1}}\chi_{2j}}, \\
    \mathbf{H}_{j, j+1} &= \frac{\sqrt{(N-2j)(N-2j-1)}}{2(N-1)} \sqrt{(2j+2)(2j+1)} \braket{\varphi^{\otimes 2} \otimes \chi_{2j}}{v(x_1-x_2) \chi_{2j+2}}.
\end{align*}

\section{BUG integrator for complex symmetric matrices}
\label{appendix: BUG for complex symmetric matrices}

For the convenience of the reader we provide here some details regarding the BUG integrator used in this work, which differs from \cites{cerutiRobustSecondorderLowrank2024,cerutiTimeIntegrationSymmetric2020,cerutiRankadaptiveRobustIntegrator2022,cerutiUnconventionalRobustIntegrator2022} in the sense that we are subject to some implementation constraints.
We focus on the following differential equation
\begin{equation*}
    \frac{d}{dt} (UBU^\top)  = F(UBU^\top),
\end{equation*}
for some matrix-valued function \( F : \mathbb{C}^{M \times M} \to \mathbb{C}^{M \times M} \), with \( U \in \mathbb{C}^{M \times r} \) unitary and \( B \in \mathbb{C}^{r\times r} \) invertible.
In our case, the main difference with \cites{cerutiRobustSecondorderLowrank2024,cerutiTimeIntegrationSymmetric2020,cerutiRankadaptiveRobustIntegrator2022,cerutiUnconventionalRobustIntegrator2022} is that \( F \) can only be computed for symmetric arguments.

\subsection{Classical BUG integrator}

The BUG integrator from \cites{cerutiRobustSecondorderLowrank2024,cerutiTimeIntegrationSymmetric2020,cerutiRankadaptiveRobustIntegrator2022,cerutiUnconventionalRobustIntegrator2022} applies to differential equations of the form 
\begin{equation*}
    \frac{d}{dt} (USV^*) = F(USV^*),
\end{equation*}
with \( U, V \in \mathbb{C}^{M \times r} \) unitary and \( S \in \mathbb{C}^{r\times r} \) invertible, and relies on the following steps:
\begin{itemize}[leftmargin=*]
    \item \( K \)-step -- Write \( K = US \), and let \( V \) be constant with respect to time. Then the differential equation on \( K \) reads
        \begin{equation*}
            \dot{K}(t) = F(K(t) V(t_0)^*) V(t_0),\ t\in [t_0, t_1], \quad K(t_0) = U(t_0) S(t_0).
        \end{equation*}
        Perform an orthonormalization of the columns of \( K(t_1) \) to obtain the matrix \( U(t_1) \).
    \item \( L \)-step -- Write \( L = VS^* \), and let \( U \) be constant with respect to time. Then the differential equation on \( L \) reads
        \begin{equation*}
            \dot{L}(t) = F( U(t_0) L(t)^*)^* U(t_0),\ t\in [t_0, t_1], \quad L(t_0) = V(t_0) S(t_0)^*.
        \end{equation*}
        Perform an orthonormalization of the columns of \( L(t_1) \) to obtain the matrix \( V(t_1) \).
    \item \( S \)-step -- Write \( \widetilde{S} := U(t_1)^* U(t_0) S V(t_0)^* V(t_1) \), and integrate
        \begin{equation*}
            \dot{\widetilde{S}}(t) = U(t_1)^* F(U(t_1) \widetilde{S}(t) V(t_1)^*) V(t_1),\ t\in [t_0, t_1],\quad \widetilde{S}(t_0) = U(t_1)^* U(t_0) S(t_0) V(t_0)^* V(t_1).
        \end{equation*}
        Truncate the basis by getting rid of the smallest singular values of \( S \), based on some given smallness condition.
\end{itemize}

\subsection{Preserving symmetry constraints during the BUG integration}

In our situation, we have the following constraints:
\begin{enumerate}[label=(\roman*)]
    \item We wish to retain the complex symmetric form \( \widetilde{U} \widetilde{B} \widetilde{U}^\top \), with \( \widetilde{U} \) unitary and \( \widetilde{B} \) invertible, at every (sub-)step of the integration.

    \item the right-hand side \( F \) of the differential equations can only be evaluated with symmetric arguments.
\end{enumerate}

These two restrictions come from the fact that the differential equations \eqref{eqn: equations of motion} have been obtained under the assumption that \( \chi_2 \) is symmetric.

For the constraint (i), it actually helps us because it reduces the computational complexity of the usual BUG integrator, since we can write
\begin{equation*}
    \chi_2 = U B V^*
\end{equation*}
where \( V = \conj{U} \). 
This form must be satisfied at all times, which means that once \( U(t_1) \) is obtained, we just need to set \( V(t_1) = \conj{U(t_1)} \).
Therefore, one only needs to perform the \( K \)- and \( S \)-substeps of the BUG integrator, since the \( L \)-substep is now equivalent to the \( K \)-substep.

For the point (ii), the issue lies in the usual conventions for the differential equations appearing in the \( K \)-substep of the BUG integrator.
Recall the \( K \)-substep consists in integrating the following differential equation on \( [t_0, t_1] \):
\begin{equation}
    \label{eqn: appendix - usual Kstep BUG integrator}
    \dot{K}(t) = F(K(t) V_0^*) V_0, \quad K(t_0) = U_0 S_0.
\end{equation}
In our symmetric situation \( V_0 = \conj{U_0} \), at time \( t=t_0 \), the argument of \( F \) is symmetric since \( K(t_0)V_0^* = U_0 S_0 U_0^\top \).
For a high-order time integration of this differential equation, one typically needs to compute \( F(K(t)V_0^*) \) for different values of \( t\in [t_0,t_1] \).
For any such time \( t \), the argument \( K(t) V_0^* \) cannot be guaranteed to be symmetric but in our setting \( F \) can only be computed for symmetric arguments.
In \eqref{eqn: appendix - usual Kstep BUG integrator}, the argument of the function \( F \) is \( K(t) V_0^* \) because \( K(t) = U(t)S(t) \) is evolved by letting \( V(t) \) fixed.
This is a nice property in general because it allows to perform the \( K \)- and \( L \)-steps in parallel, but it is not usable in our setting.
Instead, we want to compute \( F(K(t)V(t)^*) = F(K(t) U(t)^\top) \) to guarantee symmetric arguments for the function \( F \).

Our goal is thus to compute \( K(t)U(t)^\top = U(t)S(t)U(t)^\top \) given only the matrix \( K(t) \).
We may omit the time-dependence of the matrices \( U,K \) and \( S \) from now on.
The matrix \( S \) is symmetric, so we can apply an Autonne--Takagi decomposition: this yields \( S = W \Sigma W^\top \) for \( \Sigma \) diagonal and \( W \) unitary, and thus \( K U^\top = U S U^\top = \tilde{U} \Sigma \tilde{U}^\top \), with \( \tilde{U} := UW \) unitary.
See Subsection \ref{appendix: autonne takagi factorization} for more detail on the Autonne--Takagi factorization.
The main idea now is to apply the BUG integrator not on the matrix \( \chi_2 = USU^\top \), but on \( \chi_2 = \tilde{U} \Sigma \tilde{U}^\top \).
These are two expressions for the same matrix \( \chi_2 \), but as we will see the latter is more convenient.
By applying the BUG integrator to update \( \chi_2 = \tilde{U} \Sigma \tilde{U}^\top \), the \( K \)-step now reads
\begin{equation*}
    \dot{K}(t) = F(K(t) \tilde{U}(t)^\top) V_0, \quad K(t_0) = \tilde{U}_0 \Sigma_0.
\end{equation*}
The matrix \( K(t) = \tilde{U}(t) \Sigma(t) \) is an SVD decomposition of \( K(t) \), where the matrix of right eigenvectors is the identity matrix \( I \).
This is not an SVD we necessarily obtain numerically, instead we can only get \( K = L D R^* \) with \( D \) diagonal and \( L,R \) unitary matrices.
However, the diagonal factor in the SVD is unique, and the left/right eigenvectors are unique up to a unitary factor when the matrix is invertible (see \cite{hornMatrixAnalysis2013}*{Theorem~2.6.5}): in other words, if \( L_1 D_1 R_1^* = L_2 D_2 R_2^* \) are two SVD decompositions of an invertible matrix, then \( D_1 = D_2 \) and there is a unitary \( X \) such that \( L_1 = L_2 X \) and \( R_1 = R_2 X \).
In our case, \( L_1 = \tilde{U} \), \( D_1 = \Sigma \), \( R_1 = I \), and \( L_2 = L \), \( D_2=D \), \( R_2 = R \). 
Therefore, \( D = S \) and there is a unitary \( X \) such that \( \tilde{U} = LX \), \( I = RX \).
We deduce \( X = R^* \) and thus \( \tilde{U} = LR^* \).

Finally, the matrix \( K U^\top \) can be computed from an SVD of \( K = LDR^* \) as 
\begin{equation*}
    K U^\top = (LR^*) D (LR^*)^\top.
\end{equation*}

In the usual rank-adaptive BUG integrator \cite{cerutiRankadaptiveRobustIntegrator2022}, the rank truncation at the end of each time step is usually performed using an SVD.
In our situation, we want to keep the symmetric factorization of the matrix \( \chi_2 = \tilde{U} \Sigma \tilde{U}^\top \) after the rank truncation, so we perform an Autonne--Takagi factorization instead of an SVD.

\subsection{Enforcing a diagonal core matrix}
\label{appendix: BUG enforcing diagonal core matrix}

It is known that the standard BUG integrator preserves symmetry.
A symmetric matrix can be diagonalized (whether by SVD or by an Autonne--Takagi decomposition), but the standard BUG integrator only uses this property at the end of every time integral in order to perform the rank truncation.
However, the intermediate substeps of the BUG integrator also preserve the symmetry, making it possible to consider only core matrices that are diagonal during the computation of the right-hand side of the equations of motion.

Let us make it clearer.
Assume we are integrating in time the low-rank approximation \( \chi_2 = U B U^\top \), for which the equations of motion are of the form \( \dot{\chi}_2 = F(\chi_2) \), with a given initial unitary basis \( U_0 \) and an initial diagonal core matrix \( B_0 \).
The function \( F \) is further assumed to be computed more efficiently when \( B \) is diagonal.
The standard BUG integrator would not use this property, but we can adapt it easily.
We consider the integration over the whole time interval via a Runge-Kutta method, which means that we only need to study explicit Euler substeps.

The \( S \)-step in the BUG integrator with complex symmetric matrices is of the form
\begin{equation*}
    \dot{S}(t) = U_0^* F(U_0 S(t) U_0^\top) \conj{U_0}
\end{equation*}
Using the explicit Euler substeps in the Runge-Kutta integration, we obtain
\begin{equation*}
    S(c_{k+1}\Delta t) = S(c_{k} \Delta t) + \Delta t U_0^* F(U_0 S(c_k \Delta t) U_0^\top) \conj{U_0}.
\end{equation*}
for some coefficients \( c_k, c_{k+1} \) given by the choice of the Runge-Kutta method.
By diagonalizing \( S(c_k \Delta t) = W_k \Sigma_k W_k^\top \) using an Autonne--Takagi decomposition, we get
\begin{equation*}
    S(c_{k+1}\Delta t) 
    = W_k \Sigma_k W_k^\top + \Delta t U_0^* F(U_0 W_k \Sigma_k W_k^\top U_0^\top) \conj{U_0}
    = W_k \Sigma_k W_k^\top + \Delta t U_0^* F(U_0 W_k \Sigma_k (U_0 W_k)^\top) \conj{U_0}.
\end{equation*}
The computation of \( F(U_0 W_k \Sigma_k (U_0 W_k)^\top) \) can be done efficiently, since now the core matrix \( \Sigma_k \) is diagonal, and the basis \( U_k := U_0 W_k \) is still unitary.

\subsection{Autonne--Takagi factorization}
\label{appendix: autonne takagi factorization}

The Autonne--Takagi factorization is a symmetric equivalent to the SVD decomposition of a complex symmetric matrix: for any complex symmetric matrix \( A \in \mathbb{C}^{n\times n} \), there are matrices \( W \in \mathbb{C}^{n\times n} \) and \( \Sigma \in \mathbb{C}^{n\times n} \), with \( W \) unitary and \( \Sigma \) the diagonal matrix of singular values of \( A \), such that
\begin{equation*}
    A = W \Sigma W^\top.
\end{equation*}
We refer the reader to \cite{hornMatrixAnalysis2013}*{Corollary 4.4.4} for a proof of this result, and to \cite{houdeMatrixDecompositionsQuantum2024}*{Section 4} for an algorithm to compute the Autonne--Takagi factorization from the usual SVD.

\section{Implementation details}
\label{appendix: implementation details}

\subsection{Computation of \texorpdfstring{\( \chi_{2k} \)}{chi2k}}

We address in this section the efficient computation of \( (\chi_{2k})_{i_1, \dots, i_{2k}} \), which are the values of \( \chi_{2k} \) after the space discretization at the multi-index \( (i_1, \dots, i_{2k}) \) on the discretized grid.
If we discretize space using \( M \) points as described in Section \ref{section: time and space discretizations}, then \( i_\ell \in \intinterval{1}{M} \) for \( \ell = 1, \dots, 2k \).
A naive approach would imply storing \( M^{2k} \) values, which is simply not doable in practice for large \( M \).

Recall that \( \chi_2 \) is approximated via a low-rank representation of rank \( r \), which implies by the tensorization \eqref{eqn: Bogoliubov tensorization}-\eqref{eqn: Bogoliubov tensorization normalized} that the discrete functions \( \chi_{2k} \) are also low-rank of rank \( r \) (see Section \ref{section: LRA of excitation functions}).
It means that we only need to store the core tensor of \( \chi_{2k} \), of size \( r^{2k} \).
It is a huge improvement over storing \( M^{2k} \) values, but it is still not easily doable in practice for reasonably large \( r \) and \( k \).
Therefore, we need an efficient way to compute the values \( (\chi_{2k})_{i_1, \dots, i_{2k}} \), which does not rely on storing them.

The tensorization \eqref{eqn: Bogoliubov tensorization}-\eqref{eqn: Bogoliubov tensorization normalized} was used in \eqref{eqn: core tensor chi2j using tensorization} to obtain the core tensor of \( \chi_{2k} \) based on the core matrix of \( \chi_2 \):
\begin{equation*}
    b_{i_1, \dots, i_{2k}} = \frac{\widetilde{b}_{i_1, \dots, i_{2k}}}{\norm{\widetilde{b}}{F}},\qquad
    \widetilde{b}_{i_1, \dots, i_{2k}} := \sum_{\tau\in P_{2k}} b_{i_{\tau(1)}, i_{\tau(2)}} \cdots b_{i_{\tau(2k-1)}, i_{\tau(2k)}}.
\end{equation*}
Equivalently, we can replace the sum over \( \tau \in P_{2k} \) by a sum over \( S_{2k} \):
\begin{equation*}
    b_{i_1, \dots, i_{2k}} = \frac{\widehat{b}_{i_1, \dots, i_{2k}}}{\norm{\widehat{b}}{F}},\qquad
    \widehat{b}_{i_1, \dots, i_{2k}} = \sum_{\tau\in S_{2k}} b_{i_{\tau(1)}, i_{\tau(2)}} \cdots b_{i_{\tau(2k-1)}, i_{\tau(2k)}} = 2^k k! \widetilde{b}_{i_1, \dots, i_{2k}}.
\end{equation*}
By using that the core matrix \( b \) of \( \chi_2 \) is diagonal, only the permutations \( \tau \in S_{2k} \) such that \( i_{\tau(1)}=i_{\tau(2)}, \dots, i_{\tau(2k-1)}=i_{\tau(2k)} \) yield a nonzero summand.
In particular, this condition can only be achieved if the multi-index \( (i_1, \dots, i_{2k}) \) is such that each value appears an even number of times in the multi-index.
Hence, let us assume from now on that each value in the multi-index appears an even number of times.
We define \emph{subgroups} \( \mathsf{S}^\ell_{(i_1, \dots, i_{2k})} := \{j : i_j = \ell, j=1, \dots, 2k\} \).
In the expression of \( \widehat{b}_{i_1, \dots, i_{2k}} \) above, the permutations \( \tau \) that yield nonzero summands are those that pair together indices of the same subgroup.
There are \( \abs{\mathsf{S}^\ell_{(i_1, \dots, i_{2k})}}! \) such permutations corresponding to the subgroup associated to the value \( \ell \), and therefore \( \prod_{\ell=1}^r \abs{\mathsf{S}^\ell_{(i_1, \dots, i_{2k})}}! \) such permutations in total.
Each one of them leaves the order of the terms in the product \( b_{i_{\tau(1)}, i_{\tau(2)}} \dots b_{i_{\tau(2k-1)}, i_{\tau(2k)}} \) unchanged, but among all permutations \( \tau \in S_{2k} \) there are also permutations that change this order.
There are \( k! \) of them.
Among those \( k! \) permutations that rearrange the order of terms in the products, they may also swap two terms that correspond to the same subgroup.
We have already taken those permutations into account with \( \prod_{\ell=1}^r \abs{\mathsf{S}^\ell_{(i_1, \dots, i_{2k})}}! \), which means that the actual number of permutations changing the order of the terms in the product needs to be divided by the number of pairs of indices each subgroup contributes to.
We thus have \( k! / \prod_{\ell=1}^r \frac{\abs{\mathsf{S}^\ell_{(i_1, \dots, i_{2k})}}}{2}! \) permutations in total that change the order of terms in the summand but preserve the order of terms corresponding to a same subgroup.
We can assume without loss of generality that the multi-index \( (i_1, \dots, i_{2k}) \) is sorted, and we end up with
\begin{equation*}
    \widehat{b}_{i_1, \dots, i_{2k}} = \sum_{\tau\in S_{2k}} b_{i_{\tau(1)}, i_{\tau(2)}} \dots b_{i_{\tau(2k-1)}, i_{\tau(2k)}}
     = b_{i_1, i_2} \dots b_{i_{2k-1}, i_{2k}} k!\frac{\prod_{\ell=1}^r \abs{\mathsf{S}^\ell_{(i_1, \dots, i_{2k})}}!}{\prod_{\ell=1}^r \frac{\abs{\mathsf{S}^\ell_{(i_1, \dots, i_{2k})}}}{2}!}. 
\end{equation*}
Finally, owing to \( \widehat{b}_{i_1, \dots, i_{2k}} = 2^k k! \widetilde{b}_{i_1, \dots, i_{2k}} \), we obtain
\begin{equation*}
    \widetilde{b}_{i_1, \dots, i_{2k}} = \sum_{\tau\in S_{2k}} b_{i_{\tau(1)}, i_{\tau(2)}} \dots b_{i_{\tau(2k-1)}, i_{\tau(2k)}}
    = b_{i_1, i_2} \dots b_{i_{2k-1}, i_{2k}} \frac{\prod_{\ell=1}^r \abs{\mathsf{S}^\ell_{(i_1, \dots, i_{2k})}}!}{2^k \prod_{\ell=1}^r \frac{\abs{\mathsf{S}^\ell_{(i_1, \dots, i_{2k})}}}{2}!}. 
\end{equation*}
When the multi-index \( (i_1, \dots, i_{2k}) \) is sorted, the computation of the right-hand side can be performed in a single iteration over the multi-index \( (i_1, \dots, i_{2k}) \), i.e. in time \( \mathcal{O}(k) \), and can thus be performed on-the-fly.
The norm \( \norm{\widetilde{b}_{i_1, \dots, i_{2k}}}{F} \) can be computed once and stored anytime \( b \) is updated.

\subsection{Computation of \texorpdfstring{\( \mathsf{P}^{(p,q)}_{i'_{\mathcal{P}^C}, i_{\mathcal{Q}^C}} \)}{P (p,q) (i'calPc, icalQc)}}

In this section we give some details about the efficient computation of the coefficient \( \mathsf{P}^{(p,q)}_{i'_{\mathcal{P}^C}, i_{\mathcal{Q}^C}} \) given in \eqref{eqn: coeff Ppq calPc calQc}:
\begin{equation}
    \label{eqn: sum sfP in appendix}
    \mathsf{P}^{(p,q)}_{i'_{\mathcal{P}^C}, i_{\mathcal{Q}^C}} := \sum_{i'_{\mathcal{P}}=1}^r \conj{b_{i'_\mathcal{P}, i'_{\mathcal{P}^C}}} b_{i'_\mathcal{P}, i_{\mathcal{Q}^C}},
\end{equation}
Recall that the core tensors \( b \) that appear here are obtained using the tensorization \eqref{eqn: Bogoliubov tensorization}-\eqref{eqn: Bogoliubov tensorization normalized}, which gives \eqref{eqn: core tensor chi2j using tensorization}.
The summand is nonzero if and only if both terms of the product are simultaneously nonzero.
We study the core tensor \( b_{i'_\mathcal{P}, i'_{\mathcal{P}^C}} \) of size \( 2p \), and write \( i' = (i'_\mathcal{P}, i'_{\mathcal{P}^C}) \): it is nonzero if and only if there is at least one permutation \( \tau \in P_{2p} \) such that one of the summands is nonzero in the sum \eqref{eqn: core tensor chi2j using tensorization} that defines this core tensor.
For the multi-index \( i' \), the diagonality of \( B \) means that \( b_{i'} \) is nonzero if and only if there is a permutation \( \tau \in P_{2p} \) such that \( i'_{\tau(1)}=i'_{\tau(2)} \), \dots, \( i'_{\tau(2p-1)}=i'_{\tau(2p)} \).
This can only occur when each index in \( i' \) is repeated an even number of times.
For example with \( p=2 \) and \( i'=(1, 1, 2, 3) \) we would get \( b_{i'} = 0 \) since the values \( 2 \) and \( 3 \) only appear an odd number of times. But \( i'=(1, 1, 2, 2) \) would lead to \( b_{i'} \neq 0 \) since all the values in the multi-index \( i' \) are repeated an even number of times.
We introduce \( G^{\ell}_{i'} \) the number of indices in \( i' \) that are equal to \( \ell \).
The argument above can be rewritten as
\begin{equation}
    \label{eqn: simplification odd size subgroups coeff P}
    \exists \ell \in \intinterval{1}{r} \text{ such that } G^{\ell}_{i'} \text{ is odd} \implies b_{i'} = 0
\end{equation}
The same reasoning holds for \( b_{i'_\mathcal{P}, i_{\mathcal{Q}^C}} \), with the multi-index \( i = (i'_\mathcal{P}, i_{\mathcal{Q}^C}) \) of size \( 2q \).
The criterion \eqref{eqn: simplification odd size subgroups coeff P} is particularly helpful when computing the sum defining \( \mathsf{P}^{(p,q)}_{i'_{\mathcal{P}^C}, i_{\mathcal{Q}^C}} \). 
Define the sets 
\begin{equation*}
    \mathsf{O}_{i'_{\mathcal{P}^C}} := \left\{ \ell \in \intinterval{1}{r} : G^\ell_{i'_{\mathcal{P}^C}} \ \text{ is odd} \right\}
    \quad \text{and}\quad
    \mathsf{O}_{i_{\mathcal{Q}^C}} := \left\{ \ell \in \intinterval{1}{r} : G^\ell_{i_{\mathcal{Q}^C}} \ \text{ is odd} \right\}.
\end{equation*}
If they are different, we are in the situation where
\begin{equation}
    \label{appendix: condition coeff Ppq zero}
    G^{\ell}_{(i'_\mathcal{P}, i'_{\mathcal{P}^C})}\ \text{ or }\ G^{\ell}_{(i'_\mathcal{P}, i'_{\mathcal{Q}^C})}\ \text{ is odd}
\end{equation}
for some \( \ell \in \intinterval{1}{r} \), and thus \( \mathsf{P}^{(p,q)}_{i'_{\mathcal{P}^C}, i_{\mathcal{Q}^C}} = 0 \) by \eqref{eqn: simplification odd size subgroups coeff P}.
If the two sets are the same, we don't need to iterate over \( i'_\mathcal{P} \in \intinterval{1}{r}^{|\mathcal{P}|} \) for the computation of \( \mathsf{P}^{(p,q)}_{i'_{\mathcal{P}^C}, i_{\mathcal{Q}^C}} \): indeed, we can restrict the sum in \eqref{eqn: sum sfP in appendix}  to the multi-indices \( i'_{\mathcal{P}} \) such that \( G^\ell_{i'_{\mathcal{P}^C}, i'_{\mathcal{P}}} \) and \( G^\ell_{i'_{\mathcal{Q}^C}, i'_{\mathcal{P}}} \) are both even for all \( \ell \).
This enforces \( i'_{\mathcal{P}} \) to be of the form \( i'_{\mathcal{P}} = (\mathsf{O}_{i'_{\mathcal{P}^C}}, \mathsf{R}_s, \mathsf{R}_s) \), where \( \mathsf{R}_s \in \intinterval{1}{r}^s \) and \( s := \frac{\abs{\mathcal{P}}-\abs{\mathsf{O}_{i'_{\mathcal{P}^C}}}}{2} \).
The iteration over \( i'_{\mathcal{P}} \) in \eqref{eqn: sum sfP in appendix} can now be replaced by an iteration over \( \mathsf{R}_s = (k_1, \dots, k_s) \in \intinterval{1}{r}^s \), and the multi-index \( i'_\mathcal{P} \) is given by:
\begin{equation*}
    i'_\mathcal{P} = (\mathsf{O}_{i'_{\mathcal{P}^C}}, k_1, k_1, \dots, k_s, k_s).
\end{equation*}
Of course, this makes sense only if \( s \) is an even, non-negative integer, i.e. if \( \abs{\mathsf{O}_{i'_{\mathcal{P}^C}}} \leq |\mathcal{P}| \) and if \( \abs{\mathcal{P}}-\abs{\mathsf{O}_{i'_{\mathcal{P}^C}}} \) is even.
If this is not the case, then the condition \eqref{appendix: condition coeff Ppq zero} is again satisfied and \( \mathsf{P}^{(p,q)}_{i'_{\mathcal{P}^C}, i_{\mathcal{Q}^C}} = 0 \).
Ultimately, we can replace the summation over \( i'_{\mathcal{P}} \in \intinterval{1}{r}^{\abs{\mathcal{P}}} \) in \eqref{eqn: sum sfP in appendix} by a summation over \( \mathsf{R}_s \in \intinterval{1}{r}^s \) with \( s = \frac{\abs{\mathcal{P}}-\abs{\mathsf{O}_{i'_{\mathcal{P}^C}}}}{2} \).
This is a substantial improvement for either large \( \abs{\mathcal{P}} \) or large \( r \).

\bibliography{biblio}

\end{document}